\documentclass{amsart}

\RequirePackage[l2tabu,orthodox]{nag}

\usepackage{amssymb,amsxtra,latexsym}

\usepackage[full]{textcomp}
\usepackage[T1]{fontenc}
\usepackage{ebgaramond}
\usepackage[scaled=0.95,ebgaramond,bigdelims]{newtxmath}
\usepackage[scaled=0.95]{cabin}
\usepackage[scaled=0.92,varqu,varl,var0]{inconsolata}
\DeclareSymbolFont{numbers}{T1}{EBGaramond-LF}{m}{n}
\SetSymbolFont{numbers}{bold}{T1}{EBGaramond-LF}{bx}{n}
\DeclareMathSymbol{0}\mathalpha{numbers}{"30}
\DeclareMathSymbol{1}\mathalpha{numbers}{"31}
\DeclareMathSymbol{2}\mathalpha{numbers}{"32}
\DeclareMathSymbol{3}\mathalpha{numbers}{"33}
\DeclareMathSymbol{4}\mathalpha{numbers}{"34}
\DeclareMathSymbol{5}\mathalpha{numbers}{"35}
\DeclareMathSymbol{6}\mathalpha{numbers}{"36}
\DeclareMathSymbol{7}\mathalpha{numbers}{"37}
\DeclareMathSymbol{8}\mathalpha{numbers}{"38}
\DeclareMathSymbol{9}\mathalpha{numbers}{"39}

\usepackage[pagebackref]{hyperref,xcolor}
\definecolor{ceruleanblue}{rgb}{0.16, 0.32, 0.75}
\hypersetup{colorlinks=true,allcolors=ceruleanblue}

\usepackage{tikz-cd}
\tikzcdset{arrow style=math font}
\tikzcdset{diagrams={nodes={inner sep=3pt}}}

\usepackage{multicol}

\makeatletter

\def\theglossary{\@restonecoltrue\if@twocolumn\@restonecolfalse\fi
\columnseprule\z@ \columnsep 35\p@
\let\item\@idxitem
\parindent\z@ \parskip\z@\@plus.3\p@
\relax\footnotesize}

\makeatother

\makeglossary

\newcommand\note[1]%
                {$^\dagger$\marginpar{\footnotesize{$^\dagger${#1}}}}

\newcommand\mr{\mathring}

\newcounter{num}

\newenvironment{numerate}%
{\begin{list}{\hskip\labelsep\hskip\parindent(\roman{num})}%
{\usecounter{num}%
\setlength\leftmargin{0em}\setlength\topsep{0em}%
\setlength\parsep{0em}\setlength\partopsep{0em}%
\setlength\itemsep{0em}\setlength\labelwidth{0em}}}%
{\end{list}}

\numberwithin{equation}{subsection}

\swapnumbers

\newtheorem{theorem}[equation]{Theorem}
\newtheorem{lemma}[equation]{Lemma}
\newtheorem{proposition}[equation]{Proposition}
\newtheorem{corollary}[equation]{Corollary}

\theoremstyle{definition}
\newtheorem{definition}[equation]{Definition}
\newtheorem{example}[equation]{Example}

\newtheorem{remark}[equation]{Remark}
\newtheorem{remarks}[equation]{Remarks}

\newtheorem{notation-definition}[equation]{Notation and definition}

\newcommand\eu{\mathfrak}
\newcommand\lie{\mathfrak}

\newcommand{\f}{\lie{f}} 
\newcommand{\g}{\lie{g}} 
\newcommand{\h}{\lie{h}}
 
\newcommand{\n}{\lie{n}}

\newcommand{\liea}{\lie{a}}

\newcommand{\liel}{\lie{l}}

\newcommand\bb{\mathbf}

\newcommand\R{\bb{R}}

\newcommand\ca{\mathscr}

\DeclareMathOperator\ad{ad} 

\DeclareMathOperator\ann{ann}

\DeclareMathOperator\bad{\mathbf{ad}} 
\DeclareMathOperator\an{\mathbf{an}}

\DeclareMathOperator\Der{Der}
\DeclareMathOperator\diag{diag}

\DeclareMathOperator\ev{ev}

\DeclareMathOperator\id{id}
\DeclareMathOperator\im{im}

\DeclareMathOperator\Lie{Lie}

\DeclareMathOperator\pr{pr}
\DeclareMathOperator\rank{rank}

\DeclareMathOperator\sspan{span}

\DeclareMathOperator\stab{stab}

\newcommand\group[1]{{\mathbf{#1}}}

\newcommand\abs[1]{\lvert#1\rvert}

\newcommand\inner[1]{\langle#1\rangle}
\newcommand\biginner[1]{\bigl\langle#1\bigr\rangle}

\newcommand\qu[1][\kern.3ex]{/\kern-.7ex/_{\kern-.4ex#1}}
\newcommand\bigqu[1][\,\,]{\big/\kern-.85ex\big/_{\!\!#1}}

\newcommand\powl{[\kern-.6ex[}
\newcommand\powr{]\kern-.6ex]}

\newcommand\bigpowl{\bigl[\kern-.7ex\bigl[}
\newcommand\bigpowr{\bigr]\kern-.7ex\bigr]}

\newcommand\Bigpowl{\Bigl[\kern-.7ex\Bigl[}
\newcommand\Bigpowr{\Bigr]\kern-.7ex\Bigr]}

\newcommand\biggpowl{\biggl[\kern-.7ex\biggl[}
\newcommand\biggpowr{\biggr]\kern-.7ex\biggr]}

\newcommand\sur{\mathrel{\to\kern-1.8ex\to}}

\newcommand\longto{\longrightarrow}

\newcommand\longsur{\mathrel{\longrightarrow\kern-1.8ex\to}}

\newcommand\pardif[2]{\frac{\partial#1}{\partial#2}}

\newcommand\tangent[1]{\partial\mspace{-2mu}/\mspace{-2mu}\partial#1}

\newcommand\eps{\varepsilon}
\renewcommand\phi{\varphi}
\newcommand\bu{{\scriptscriptstyle\bullet}}

\newcommand\hor{{\mathrm{hor}}}

\newcommand\can{{\mathrm{can}}}

\newcommand\F{{\ca{F}}}

\newcommand\barM{M\kern-0.7em\overline{\phantom I}\kern0.18em}

\begin{document}


\title[Symplectic Lie algebroids]{Symplectic reduction and a
  Darboux-Moser-Weinstein theorem for Lie algebroids}

\author{Yi Lin}

\address{
Georgia Southern University, Statesboro, Georgia 30460, USA}

\email{yilin@georgiasouthern.edu}

\author{Yiannis Loizides}

\address{
Cornell University, Ithaca, New York 14853, USA}

\curraddr{
George Mason University, Fairfax, Virginia 22030, USA}

\email{yloizide@gmu.edu}

\author{Reyer Sjamaar}

\address{
Cornell University, Ithaca, New York 14853, USA}

\email{sjamaar@math.cornell.edu}

\author{Yanli Song}

\address{
Washington University, St. Louis, Missouri 63130, USA}

\email{yanlisong@wustl.edu}

\date\today

\dedicatory{For Victor Guillemin}


\begin{abstract}
We extend the Marsden-Weinstein reduction theorem and the
Darboux-Moser-Weinstein theorem to symplectic Lie algebroids.  We also
obtain a coisotropic embedding theorem for symplectic Lie algebroids.
\end{abstract}


\maketitle

\tableofcontents


\section{Introduction}

In this paper we extend the Marsden-Weinstein reduction theorem and
the Darboux-Moser-Weinstein theorem to the setting of log symplectic
manifolds.  This work is a building block in our recently established
``quantization commutes with reduction'' theorem for log symplectic
manifolds~\cite{lin-li-sjamaar-song;log-qr}.

By a log symplectic manifold we mean a real manifold equipped with a
symplectic form that has first-order poles along a divisor (real
hypersurface) with normal crossings.  The log tangent bundle of such a
manifold, i.e.\ the vector bundle whose sections are the vector fields
tangent to the divisor, is an example of a Lie algebroid, and most of
our arguments extend to the case of an arbitrary Lie algebroid.  To
underscore the utility of Lie algebroids in ``desingularizing''
certain Poisson structures we have chosen to formulate our results, in
as far as possible, in terms of Poisson and symplectic structures on
general Lie algebroids.


Poisson structures on Lie algebroids were introduced under the name of
triangular Lie bialgebroids by Mackenzie and
Xu~\cite{mackenzie-xu;bialgebroids-poisson}.  Such structures include
Poisson structures in the usual sense as well as triangular Lie
bialgebras.  Symplectic structures on Lie algebroids were introduced
by Nest and
Tsygan~\cite{nest-tsygan;deformations-symplectic-lie-algebroids} in
order to extend Fedosov's work on deformation quantization from
symplectic manifolds to a wider class of Poisson manifolds, and by
Mart\'{\i}nez~\cite{martinez;mechanics-lie-algebroids},
\cite{martinez;lagrangian-lie-algebroids} to develop a version of
Lagrangian mechanics on Lie algebroids initiated by
Weinstein~\cite{weinstein;lagrangian-groupoids}.  Examples of
symplectic Lie algebroids include symplectic manifolds, log symplectic
manifolds, $b^m$-symplectic manifolds, complex symplectic manifolds,
and constant rank Poisson structures.

Our Hamiltonian reduction theorem,
Theorem~\ref{theorem;mikami-weinstein}, is a version of the
Mikami-Weinstein reduction
theorem~\cite{mikami-weinstein;moments-reduction-groupoids} carried
over to the context of Lie algebroids.  In the special case where the
target of the moment map is $\g^*$, the dual of a Lie algebra $\g$
equipped with the tangent Lie algebroid $T\g^*$, our theorem was
obtained earlier by Marrero et
al.~\cite[Theorem~3.11]{marrero-padron-rodriguez;symplectic-like}.
However, Hamiltonians on log symplectic manifolds may have logarithmic
poles, and one of the purposes of allowing more general momentum
codomains than $\g^*$ is to enable us to reduce at poles of the moment
map.  A novel feature of reduction ``at infinity'' is that it involves
not only the choice of a point in the codomain, but also a choice of a
subalgebra of its Lie algebroid stabilizer.

Versions of Moser's trick and the Darboux-Moser-Weinstein theorem in
the context of Lie algebroids have been found by many authors,
including Cavalcanti and
Gualtieri~\cite{cavalcanti-gualtieri;stable-generalized}, Cavalcanti
et al.~\cite{cavalcanti-klaasse-witte;self-crossing}, Geudens and
Zambon~\cite {geudens-zambon;coisotropic-b-symplectic}, Guillemin et
al.~\cite{guillemin-miranda-weitsman;desingularizing-bm-symplectic},
Kirchhof-Lukat~\cite{kirchhoff-lukat;lagrangian-stable-generalized},
Klaasse and
Lanius~\cite[\S\,4.3]{klaasse;geometric-structures-lie-algebroids},
\cite{klaasse-lanius;splitting-poisson}, Miranda and
Scott~\cite[\S\,2]{miranda-scott;e-manifolds}, and
Smilde~\cite{smilde;linearization-poisson}.  Our version,
Theorem~\ref{theorem;dmw}, overlaps with these results and contains
some of them as special cases.  It applies to situations where a Lie
subalgebroid is locally a deformation retract of the ambient Lie
algebroid.  Our proof relies on recent work of Bischoff et
al.~\cite{bischoff-bursztyn-lima-meinrenken} and Bursztyn et
al.~\cite{bursztyn-lima-meinrenken;splitting}, which provides us with
a method to produce Lie algebroid homotopies from so-called Euler-like
sections.  One corollary of our result is a normal form for transverse
coisotropic submanifolds, Theorem~\ref{theorem;coisotropic}, which is
an ingredient in our paper~\cite{lin-li-sjamaar-song;log-qr}.

We review Poisson and symplectic Lie algebroids in
Section~\ref{section;symplectic}.  The Hamiltonian reduction theorem
is in Section~\ref{section;hamilton}.  Section~\ref{section;normal}
contains a discussion of homotopies of the Lie algebroid de Rham
complex, as well as the Darboux-Moser-Weinstein theorem and the
coisotropic embedding theorem.  We spell out some consequences for the
log symplectic case in Section~\ref{section;log}.
Appendix~\ref{section;algebroid} is a review of Lie algebroids.

\subsection{Notation and terminology}\label{subsection;notation}

See Appendix~\ref{section;notation} for a notation index.  By a
\emph{manifold} we mean a finite-dimensional Hausdorff second
countable smooth ($\ca{C}^\infty$) real manifold without boundary,
typically denoted by~$M$.  By a \emph{submanifold} we mean an
injectively immersed, but not necessarily embedded, submanifold.  We
denote the inclusion map of a submanifold $N$ of\/ $M$ by~$i_N$.  Let
$E$ be a real vector bundle over~$M$.  We denote the vector bundle
projection by $\pi$ or $\pi_E$, the space of smooth sections over an
open subset $U$ of\/ $M$ by $\Gamma(U,E)$, and the space of global
smooth sections by $\Gamma(E)=\Gamma(M,E)$.  We denote the zero bundle
over $M$ by $0_M$.  By a \emph{subbundle} we mean a (not necessarily
embedded) submanifold $F$ of\/ $E$ such that $N=\pi_E(F)$ is a
submanifold of\/ $M$ and $\pi_F=\pi_E|_F\colon F\to N$ is a vector
bundle.  For instance, the annihilator $F^\circ$ of a subbundle $F$ is
a subbundle of the dual bundle~$E^*$.  By a \emph{foliation} we mean a
nonsingular (i.e.\ constant rank) smooth foliation.  We say that the
leaf space $M/\F$ of a foliation $\F$ of\/ $M$ \emph{is a manifold} if
it has a (necessarily unique) manifold structure that makes the
quotient map $M\to M/\F$ a submersion.  By a \emph{Poisson Lie
  algebroid} we mean a Lie algebroid equipped with a Poisson structure
(involutive $2$-section) and by a \emph{symplectic Lie algebroid} a
Lie algebroid equipped with a nondegenerate Poisson structure.
\glossary{iN@$i_N\colon N\to M$, inclusion of immersed submanifold}
\glossary{[@$^\circ$, annihilator of subspace or subbundle}
\glossary{W@$W^\circ$, annihilator of subspace or subbundle}
\glossary{MF@$(M,\F)$, foliated manifold}
\glossary{pi@$\pi_E\colon E\to M$, vector bundle projection}
\glossary{GammaUE@$\Gamma(U,E)$, smooth sections over $U$}
\glossary{GammaE@$\Gamma(E)$, global smooth sections}
\glossary{0@$0_M$, zero vector bundle over $M$}
\glossary{MFF@$M/\F$, leaf space}
%

\section{Poisson Lie algebroids}\label{section;symplectic}

In this section we review the notions of a Poisson structure and a
symplectic structure on a Lie algebroid and extend some standard
results of Poisson geometry to the wider context of Lie algebroids.
These include symplectization and reduction theorems for presymplectic
Lie algebroids, Theorems~\ref{theorem;symplectization}
and~\ref{theorem;reduction}.

\subsection{Poisson and symplectic structures on Lie algebroids}
\label{subsection;symplectic}

Let $A$ be a Lie algebroid over a manifold~$M$.  We denote the
projection by $\pi=\pi_A\colon A\to M$, the anchor by $\an=\an_A\colon
A\to TM$, and the Lie bracket on sections by
\[
[{\cdot},{\cdot}]=
[{\cdot},{\cdot}]_A\colon\Gamma(A)\times\Gamma(A)\longto\Gamma(A).
\]
The Lie algebroid has a \emph{de Rham complex}
$(\Omega_A^\bu(M),d_A)$.  Its elements, which we call \emph{Lie
  algebroid forms}, or \emph{$A$-forms}, or just \emph{forms}, are
sections of the exterior algebra bundle $\Lambda^\bu A^*$, and its
differential $d_A$ is defined in terms of the anchor and the Lie
bracket.  (See Appendix~\ref{subsection;cartan} for a review.)  The
Lie bracket on the space of sections $\Gamma(A)$ extends to a
$-1$-shifted graded Lie bracket $[\cdot,\cdot]_A$ on the algebra of
multisections $\Gamma(\Lambda^\bu A)$, known as the
\emph{Schouten-Nijenhuis bracket}.  (See
e.g.~\cite[\S\,2]{mackenzie-xu;bialgebroids-poisson}.)  In particular,
for each multisection $\sigma\in\Gamma(\Lambda^pA)$ we can form the
multisection $[\sigma,\sigma]\in\Gamma(\Lambda^{2p-1}A)$.  Let us call
$\sigma$ \emph{involutive} if\/ $[\sigma,\sigma]=0$.

If\/ $A=TM$ is the ordinary tangent bundle, then a $2$-section
$\lambda\in\Gamma(\Lambda^2A)$ defines a Poisson structure on $M$ if
and only if it is involutive.  A result of Coste et
al.~\cite[\S\,III.2]{coste-dazord-weinstein;groupoides} says that a
Poisson structure on $M$ makes the cotangent bundle $T^* M$ a Lie
algebroid.  Their result was extended by Mackenzie and Xu as follows.
Item~\eqref{item;dual-lie} of this statement is a reformulation
of~\cite[Theorem~4.3]{mackenzie-xu;bialgebroids-poisson} and
item~\eqref{item;schouten} summarizes the material
of~\cite[\S\,4]{mackenzie-xu;bialgebroids-poisson}.

\begin{theorem}[Mackenzie and Xu~\cite{mackenzie-xu;bialgebroids-poisson}]
\label{theorem;poisson}
Let $A\to M$ be a Lie algebroid and let\/
$\lambda\in\Gamma(\Lambda^2A)$ be a $2$-section of~$A$.  Define the
vector bundle map $\lambda^\sharp\colon A^*\to A$ by
$\beta(\lambda^\sharp(\alpha))=\lambda(\alpha,\beta)$ for $\alpha$,
$\beta\in\Omega_A^1(M)=\Gamma(A^*)$.  Define antisymmetric bracket
operations on functions and on $1$-forms by
\begin{equation}\label{equation;bracket-function}
\{f,g\}=\{f,g\}_\lambda=\iota_A(\lambda)(d_Af\wedge d_Ag)
\end{equation}
for $f$, $g\in\ca{C}^\infty(M)$ and
\begin{equation}\label{equation;bracket-form}
\{\alpha,\beta\}=\{\alpha,\beta\}_\lambda=
\iota_A(\lambda^\sharp\alpha)d_A\beta-\iota_A(\lambda^\sharp\beta)d_A\alpha+
d_A\iota_A(\lambda)(\alpha\wedge\beta)
\end{equation}
\glossary{lambda@$\lambda$, Poisson structure}
\glossary{lambda#@$\lambda^\sharp\colon A^*\to A$, structure map of
  Poisson structure}
\glossary{{}@$\{{\cdot},{\cdot}\}_\lambda$, Poisson bracket}
\glossary{dlambda@$d_\lambda$, Poisson differential}
for $\alpha$, $\beta\in\Omega_A^1(M)$.
\begin{enumerate}
\item\label{item;dual-lie}
$\lambda$ is involutive if and only if\/
  $\lambda^\sharp\{\alpha,\beta\}_\lambda=
  [\lambda^\sharp\alpha,\lambda^\sharp\beta]_A$ for all $\alpha$,
  $\beta\in\Omega_A^1(M)$.
\item\label{item;schouten}
Suppose $\lambda$ is involutive.  Then the
bracket~\eqref{equation;bracket-function} is a Poisson structure on
$M$ with associated Poisson tensor\/
$\an_A(\lambda)\in\Gamma(\Lambda^2TM)$, and the
bracket~\eqref{equation;bracket-form} is a Lie algebroid structure on
the dual bundle $A^*$ with anchor\/
$\an_\lambda={\an_A}\circ\lambda^\sharp$.  The map\/
$\lambda^\sharp\colon A^*\to A$ and the co-anchor $\an_A^*\colon T^*
M\to A^*$ (the transpose of the anchor $\an_A$) are Lie algebroid
morphisms.  The differential of the Lie algebroid $A^*$ is the
operator $d_\lambda\colon\Gamma(\Lambda^\bu
A)\to\Gamma(\Lambda^{{\bu}+1}A)$ given by $d_\lambda a=[\lambda,a]$.
\end{enumerate}
\end{theorem}

This theorem motivates the following definition.

\begin{definition}\label{definition;poisson}
An \emph{$A$-Poisson structure} on $M$ is an involutive $2$-section
$\lambda\in\Gamma(\Lambda^2A)$.  A \emph{Poisson Lie algebroid} is a
pair $(A,\lambda)$ consisting of a Lie algebroid $A$ over $M$ and an
$A$-Poisson structure $\lambda$ on~$M$.
\end{definition}

Thus an $A$-Poisson structure can be regarded as an ordinary Poisson
structure on $M$ together with a lift of the Poisson tensor to the Lie
algebroid~$A$.  Unlike an ordinary Poisson structure, an $A$-Poisson
structure is in general not determined by the Poisson bracket on
functions~\eqref{equation;bracket-function} alone.  Mackenzie and
Xu~\cite{mackenzie-xu;bialgebroids-poisson} refer to the triple
$(A,A^*,\lambda)$ as a \emph{triangular Lie bialgebroid}.

One use of Poisson Lie algebroids lies in the fact that sometimes an
ordinary Poisson structure on $M$ can be lifted to a ``less singular''
Poisson structure on a Lie algebroid $A$ over~$M$.  For instance, it
may happen that a degenerate Poisson structure lifts to a
nondegenerate, i.e.\ symplectic, $A$-Poisson structure.  (One such
situation is described in Section~\ref{subsection;phase}.)

\begin{definition}\label{definition;symplectic}
An \emph{$A$-symplectic form} on $M$ is a $2$-form
$\omega\in\Omega_A^2(M)$ that is $d_A$-closed, i.e.\ $d_A\omega=0$,
and non-degenerate.  We write $\omega^{-1}\in\Gamma(\Lambda^2A)$ for
the $2$-section and
\[
\begin{tikzcd}\omega^\flat\colon A\ar[r,"\cong"]&A^*\end{tikzcd},
\qquad
\begin{tikzcd}
\omega^\sharp\colon A^*\ar[r,"\cong"]&A
\end{tikzcd}
\]
\glossary{omega@$\omega$, symplectic structure}
\glossary{omegainv@$\omega^{-1}$, Poisson structure inverse to
  $\omega$}
\glossary{omegasharp@$\omega^\sharp\colon A^*\to A$, inverse of\/
  $\omega^\flat$}
\glossary{omegaflat@$\omega^\flat\colon A\to A^*$, structure map of\/
  $\omega$}
for the bundle isomorphisms determined by an $A$-symplectic
form~$\omega$.  A \emph{symplectic Lie algebroid} is a pair
$(A,\omega)$ consisting of a Lie algebroid $A$ and an $A$-symplectic
form $\omega$ on the base of~$A$.
\end{definition}

This definition follows Nest and
Tsygan~\cite{nest-tsygan;deformations-symplectic-lie-algebroids}.
(The term ``symplectic Lie algebroid'' is used by Coste et
al.~\cite[\S\,III.2]{coste-dazord-weinstein;groupoides} to mean
something different, namely the Lie algebroid of a local symplectic
groupoid.)  If\/ $\omega\in\Omega_A^2(M)$ is any nondegenerate $2$-form,
then a calculation using
Theorem~\ref{theorem;poisson}\eqref{item;dual-lie} shows that
$d_A\omega=0$ if and only if\/ $[\omega^{-1},\omega^{-1}]=0$.  So just
as in ordinary Poisson geometry an $A$-symplectic structure amounts to
an $A$-Poisson structure $\lambda$ such that the morphism
$\lambda^\sharp\colon A^*\to A$ is invertible.

Let $\lambda$ be an $A$-Poisson structure on~$M$.  To each function
$f\in\ca{C}^\infty(M)$ is associated a section
\begin{equation}\label{equation;hamilton-section}
\sigma_f=d_\lambda f=[\lambda,f]=\lambda^\sharp(d_Af)\in\Gamma(A)
\end{equation}
called the \emph{Hamiltonian section} of~$f$.  Conversely, if a
section $\sigma\in\Gamma(A)$ is of the form $\sigma=\sigma_f$ for some
function $f$, we say $f$ is a \emph{Hamiltonian} for~$\sigma$.  Like
any section of\/ $A$, the section $\sigma_f$ generates a flow on the
total space of\/ $A$ (see review in Appendix~\ref{subsection;cartan}),
which we call the \emph{Hamiltonian flow} of~$f$.  (In ordinary
Poisson geometry, this flow is the tangent flow on $A=TM$ of what one
usually calls the Hamiltonian flow of\/ $f$ on~$M$.)  The Poisson
structure $\lambda$ is invariant under the Hamiltonian flow of\/ $f$,
i.e.\
\[
\ca{L}_A(\sigma_f)\lambda=[\sigma_f,\lambda]=
d_\lambda\lambda^\sharp(d_Af)=\lambda^\sharp\bigl(d_A^2f\bigr)=0,
\]
 where $\ca{L}_A$ denotes the Lie algebroid Lie derivative.  The
 \emph{Hamiltonian correspondence} is the map
\[\ca{H}\colon\ca{C}^\infty(M)\longto\Gamma(A)\]
given by $\ca{H}(f)=\sigma_f$.  The Hamiltonian correspondence is a
Lie algebra homomorphism, and its kernel is the Lie ideal of\/
$A^*$-invariant functions.  In the symplectic case
($\lambda=\omega^{-1}$) the Poisson
bracket~\eqref{equation;bracket-function} is given by
$\{f,g\}=\omega(\sigma_f,\sigma_g)$.

\subsection{The phase space of a Lie algebroid}\label{subsection;phase}

The cotangent bundle (``phase space'') of a manifold has a natural
symplectic structure.  As noted by
Mart\'{\i}nez~\cite{martinez;mechanics-lie-algebroids} (see also de
Le\'on et al.~\cite{leon-marrero-martinez;lagrangian-lie-algebroid}
and Marrero et al.~\cite{marrero-padron-rodriguez;symplectic-like})
this familiar fact has an analogue in the world of Lie algebroids.
Let $B\to N$ be an arbitrary Lie algebroid.  The projection $\pi\colon
B^*\to N$ of the dual bundle $B^*$ is a submersion, so we can form the
pullback
\[A=\pi^!B=TB^*\times_{TN}B,\]
which is a Lie algebroid over~$B^*$.  (Pullbacks of Lie algebroids are
reviewed in Appendix~\ref{subsection;pull}.)  Elements of\/ $A=\pi^!B$
are tuples $(x,p,v,b)$, where $x\in N$, $p\in B_x^*$, $v\in T_pB^*$,
$b\in B_x$ satisfy $T_p\pi(v)=\an_B(b)$.  We define the
\emph{Liouville form} or \emph{canonical $1$-form}
$\alpha_\can\in\Omega_A^1(B^*)$ by
\[\alpha_\can(x,p,v,b)=p(b)\]
and the \emph{canonical $2$-form} $\omega_\can\in\Omega_A^2(B^*)$
by
\[\omega_\can=-d_A\alpha_\can.\]
We call the pair $(A,\omega_\can)$ the \emph{phase space} of the Lie
algebroid~$B$.  The following result says that the phase space is a
symplectic Lie algebroid and that the Poisson structure on $B^*$
determined by the canonical $2$-form is identical to the natural
linear Poisson structure that exists on the dual of any Lie algebroid.
We review the proof of item~\eqref{item;canonical} because we need the
details in Section~\ref{subsection;symplectization}.  By abuse of
language we say that a smooth map $f\colon P\to M$ \emph{cleanly}
(resp.\ \emph{transversely}) \emph{intersects} a Lie algebroid $C\to
M$ if its tangent map $Tf\colon TP\to TM$ cleanly
(resp.\ transversely) intersects the anchor $\an_C\colon C\to TM$ of\/
$A$ (Definition~\ref{definition;clean}).

\begin{proposition}[{\cite[\S\S\,3.2,~3.5,~7]%
{marrero-padron-rodriguez;symplectic-like}}]\label{proposition;phase}
  Let $B\to N$ be a Lie algebroid, let $\pi\colon B^*\to N$ be the
  dual vector bundle, and let $A=\pi^!B$ be the pullback of\/ $B$ to
  $B^*$\@.
\begin{enumerate}
\item\label{item;canonical}
The canonical $2$-form $\omega_\can\in\Omega_A^2(B^*)$ is an
$A$-symplectic form on $B^*$\@.
\item\label{item;phase-poisson}
The Poisson structure
$\an_A(\omega_\can^{-1})\in\Gamma(\Lambda^2TB^*)$ associated with
$\omega_\can$ is equal to the linear Poisson structure on $B^*$
determined by the Lie algebroid structure on~$B$.
\item\label{item;zero}
The zero section $\zeta\colon N\to B^*$ is transverse to $A$ and we
have a natural isomorphism $B\cong\zeta^!A$\@.  We have
$\zeta_!^*\omega_\can=0$, where $\zeta_!\colon B\cong\zeta^!A\to A$
denotes the canonical Lie algebroid
morphism~\eqref{equation;pull-morphism} induced by $\zeta$\@.
\end{enumerate}
\end{proposition}

\begin{proof}[Proof of~\eqref{item;canonical}]
Let $r$ be the rank of the vector bundle $B$ and let $x\in N$.  The
fibre of the vector bundle $A=\pi^!B$ over $B^*$ at $p\in B_x^*$ is
$A_p=T_pB^*\times_{T_xN}B_x$, so the rank of\/ $A$ is~$2r$.  For a
sufficiently small neighbourhood $U$ of\/ $x$ we will exhibit a frame
$e_1$, $e_2,\dots$,~$e_r$, $f_1$, $f_2,\dots$,~$f_r$ of\/ $A$ defined
on $\pi^{-1}(U)$ with respect to which the matrix of the bilinear form
$\omega_\can$ is invertible.  A section $\sigma$ of\/ $A$ can be
described as a pair $\sigma=(v,b)$ consisting of a vector field $v$ on
$B^*$ and a smooth map $b\colon B^*\to B$ satisfying $T\pi\circ
v=\an_B\circ b$.  The anchor of\/ $\sigma$ is then $\an_A(\sigma)=v$.
A section $\beta$ of\/ $B$ gives rise to a fibrewise linear smooth
function $\beta^\dagger$ on $B^*$ defined by
\begin{equation}\label{equation;fibre-linear}
\beta^\dagger(p)=\inner{p,\beta(\pi(p))},
\end{equation}
where $\inner{{\cdot},{\cdot}}$ denotes the dual pairing between $B^*$
and~$B$.  From~\eqref{equation;pull-bracket} and~\eqref{equation;d} we
obtain, for any pair of sections of\/ $A$ of the form
$\sigma_1=(v_1,\beta_1\circ\pi)$, $\sigma_2=(v_2,\beta_2\circ\pi)$,
\begin{equation}\label{equation;canonical}
\omega_\can(\sigma_1,\sigma_2)=-v_1\cdot\beta_2^\dagger+
v_2\cdot\beta_1^\dagger-[\beta_1,\beta_2]^\dagger.
\end{equation}
Now choose a frame $b_1$, $b_2,\dots$,~$b_r$ of the vector bundle $B$
defined on a neighbourhood $U$ of~$x$.  Let $b_1^*$,
$b_2^*,\dots$,~$b_r^*$ be the dual frame of~$B^*$.  We have a short
exact sequence of vector bundles over~$B^*$
\begin{equation}\label{equation;phase}
\begin{tikzcd}
\pi^*B^*\ar[r,hook]&TB^*\ar[r,two heads]&\pi^*TN,
\end{tikzcd}
\end{equation}
so each $b_i^*$ can be thought of as a vector field on $B^*$ tangent
to the fibres of\/ $\pi$, and the pair $f_i=(b_i^*,0)$ defines a
section of\/ $A$ over the open subset $\pi^{-1}(U)$ of~$B^*$.  Then
\begin{equation}\label{equation;ff}
\omega_\can(f_i,f_j)=0
\end{equation}
by~\eqref{equation;canonical}.  The frame $b_1^*$,
$b_2^*,\dots$,~$b_r^*$ determines a trivialization of the vector
bundle $B^*$ over $U$ and in particular a linear connection on $B^*$,
which gives a splitting $\theta\colon TN\to TB^*$ of the
sequence~\eqref{equation;phase}.  Each section $b_i$ then gives rise
to a vector field $v_i=\theta\circ\an_B(b_i)$ on~$B^*$.  For $1\le
i\le r$ the pair $e_i=(v_i,b_i\circ\pi)$ is a section of\/ $A$ defined
over $\pi^{-1}(U)$.  It follows from~\eqref{equation;canonical} that
\begin{equation}\label{equation;ef}
\omega_\can(e_i,f_j)=b_j^*\cdot{b}_i^\dagger=
\inner{b_j^*,b_i}=\delta_{ij}.
\end{equation}
The vector fields $v_i$ are horizontal and the functions ${b}_i^\dagger$
are covariantly constant, so
\begin{equation}\label{equation;ee}
\omega_\can(e_i,e_j)=-v_i\cdot{b}_j^\dagger+
v_j\cdot{b}_i^\dagger-[b_i,b_j]^\dagger=-[b_i,b_j]^\dagger.
\end{equation}
We conclude that the $2r$-tuple
$(e_1,e_2,\dots,e_r,f_1,f_2,\dots,f_r)$ is a frame of\/ $A$ and that
the matrix of\/ $\omega_\can$ relative to this frame is
$\bigl(\begin{smallmatrix}-C&-I_r\\I_r&0\end{smallmatrix}\bigr)$,
  where $C$ is the $r\times r$-matrix
  $\bigl([b_i,b_j]^\dagger\bigr)_{i,j}$.  In particular $\omega_\can$
  is nondegenerate.
\end{proof}

\subsection{Poisson morphisms and coisotropic subalgebroids}
\label{subsection;coisotropic}

The notions of a Poisson map and a coisotropic submanifold admit
straightforward extensions to the world of Poisson Lie algebroids.  If
$\phi\colon A\to E$ is a Lie algebroid morphism with base map
$\mr\phi\colon M\to P$, we say that multisections
$u\in\Gamma(\Lambda^\bu A)$ and $v\in\Gamma(\Lambda^\bu E)$ are
\emph{$\phi$-related}, notation $u\sim_\phi v$, if
$\phi(u_x)=v_{\mr\phi(x)}$ for all $x\in M$.

\begin{definition}\label{definition;poisson-morphism}
Let $(A\to M,\lambda_A)$ and $(E\to P,\lambda_E)$ be Poisson Lie
algebroids (Definition~\ref{definition;poisson}).  A \emph{Poisson
  morphism} from $A$ to $E$ is a Lie algebroid morphism $\phi\colon
A\to E$ such that the Poisson structures $\lambda_A$ and $\lambda_E$
are $\phi$-related.
\end{definition}

The $2$-sections $\lambda_A$ and $\lambda_E$ are $\phi$-related if and
only if the square
\[
\begin{tikzcd}
A_x^*\ar[r,"\lambda_A^\sharp"]&A_x\ar[d,"\phi"]\\
E_{\mr\phi(x)}^*\ar[r,"\lambda_E^\sharp"]\ar[u,"\phi^*"]&E_{\mr\phi(x)}
\end{tikzcd}
\]
commutes for all $x\in M$, where $\phi^*$ denotes the transpose
of~$\phi$.

Let $V$ be a vector space equipped with a constant Poisson structure
$\lambda$ and let $W$ be a subspace of~$V$.  We define
$W^\lambda=\lambda^\sharp(W^\circ)$ of\/ $V$, where $W^\circ\subseteq
V^*$ denotes the annihilator of~$W$.  We say $W$ is \emph{coisotropic}
if\/ $W^\lambda$ is contained in~$W$.  If\/ $\lambda$ is the inverse
of a symplectic structure $\omega$ on $V$, we write
$W^\lambda=W^\omega$ and call $W^\omega$ the \emph{symplectic
  orthogonal} of~$W$.
%
\glossary{[@$^\circ$, annihilator of subspace or subbundle}
\glossary{W@$W^\circ$, annihilator of subspace or subbundle}
\glossary{Wlambda@$W^\lambda=\lambda^\sharp(W^\circ)$, Poisson
  ``orthogonal'' of subspace or subbundle}
\glossary{Womega@$W^\omega=\omega^\sharp(W^\circ)$, symplectic
  orthogonal of subspace or subbundle}

\begin{definition}\label{definition;coisotropic}
Let $A\to M$ be a Poisson Lie algebroid with Poisson structure
$\lambda\in\Gamma(\Lambda^2A)$.  A Lie subalgebroid $B\to P$ of\/ $A$
is \emph{coisotropic at $x\in P$} if the subspace $B_x$ of\/ $A_x$ is
coisotropic.  A submanifold $N$ of\/ $M$ is \emph{coisotropic at $x\in
  N$} if the subspace $\an^{-1}(T_xN)$ of\/ $A_x$ is coisotropic.  We
say $B$, resp.\ $N$, is \emph{coisotropic} if it is coisotropic at all
$x\in P$, resp.\ $x\in N$.  A submanifold $N$ is \emph{clean
  coisotropic} (resp.\ \emph{transverse coisotropic}) if\/ $N$ is
coisotropic and $TN$ cleanly (resp.\ transversely) intersects the
anchor $\an\colon A\to TM$.
\end{definition}

If a submanifold $N$ of\/ $M$ cleanly intersects the anchor of\/ $A$,
we have a well-defined pullback Lie algebroid $i_N^!A$ over $N$ whose
fibre at $x\in N$ is $\an_A^{-1}(T_xN)$ (see
Appendix~\ref{subsection;pull}), and in that case the submanifold $N$
is coisotropic if and only if the Lie subalgebroid $i_N^!A$ is
coisotropic in the sense of Definition~\ref{definition;coisotropic}.

Coisotropic subalgebroids behave in the expected way under Poisson
morphisms.  For submanifolds there is no distinction between being
coisotropic relative to $\lambda_A$ and being coisotropic relative to
the underlying Poisson structure $\an_A(\lambda_A)$ on~$M$.

\begin{lemma}\label{lemma;poisson-coisotropic}
Let $(A\to M,\lambda_A)$ and $(E\to P,\lambda_E)$ be Poisson Lie
algebroids and $\phi\colon A\to E$ a Poisson morphism.
\begin{enumerate}
\item\label{item;coisotropic-preimage}
Let $F\to Q$ be a Lie subalgebroid of\/ $E$ which cleanly intersects
the morphism~$\phi$.  Let $B=\phi^{-1}(F)$ and $N=\mr\phi^{-1}(Q)$.
The Lie subalgebroid $B$ of\/ $A$ is coisotropic if and only if\/ $F$
is coisotropic at $y$ for all $y\in\mr\phi(N)$.
%
\item\label{item;poisson-base}
The base map $\mr\phi\colon M\to P$ is a Poisson map in the usual
sense relative to the Poisson structures $\an_A(\lambda_A)$ on $M$ and
$\an_P(\lambda_E)$ on~$P$.
\item\label{item;coisotropic-base}
A submanifold of\/ $M$ is coisotropic relative to $\lambda_A$ if and
only if it is coisotropic in the usual sense, i.e.\ relative to the
Poisson structure $\an_A(\lambda_A)$ on~$M$.
\end{enumerate}
\end{lemma}

\begin{proof}
\eqref{item;coisotropic-preimage}~It follows from
Proposition~\ref{proposition;clean-regular} that $B$ is a Lie
subalgebroid of\/ $A$ whose base is the submanifold $N$ of\/ $M$.  Now
use the following straightforward linear algebra fact: if\/ $f\colon
V_1\to V_2$ is a linear Poisson map between Poisson vector spaces
$(V_1,\lambda_1)$ and $(V_2,\lambda_2)$, $W_2$ is a subspace of\/
$V_2$, and $W_1=f^{-1}(W_2)$, then
\begin{equation}\label{equation;linear-poisson}
f\bigl(W_1^{\lambda_1}\bigr)=W_2^{\lambda_2}.
\end{equation}
Hence $W_1$ is coisotropic if and only $W_2$ is coisotropic.

\eqref{item;poisson-base}~This follows from the fact that the anchor
map $\an_A$ is a Poisson morphism from $(A,\lambda_A)$ to
$(TM,\an_A(\lambda_A))$.

\eqref{item;coisotropic-base}~Let $N$ be a submanifold of\/ $M$.
Then~\eqref{equation;linear-poisson} yields
$\an_A\bigl(B_x^{\lambda_A}\bigr)=(T_xN)^{\lambda_M}$ for all $x\in
N$, where $B_x=\an_A^{-1}(T_xN)$ and $\lambda_M=\an_A(\lambda_A)$.
Hence $B_x$ is coisotropic if and only if\/ $T_xN$ is coisotropic.
\end{proof}

\subsection{Presymplectic Lie algebroids: symplectization}
\label{subsection;symplectization}

A Lie algebroid equipped with a closed $2$-form of constant rank can
be turned into a symplectic Lie algebroid in two ``opposite'' ways.
In this section we explain the first method: symplectization.

\begin{definition}\label{definition;presymplectic}
Let $B\to N$ be a Lie algebroid.  A \emph{$B$-presymplectic form} on
$N$ is a $2$-form $\omega_B\in\Omega_B^2(N)$ that is $d_B$-closed and
of constant rank.  A \emph{presymplectic Lie algebroid} is a Lie
algebroid equipped with a presymplectic form.
\end{definition}

As in ordinary symplectic geometry, presymplectic Lie algebroids arise
naturally from symplectic ones through coisotropic embeddings.
Indeed, let $(A\to M,\omega)$ be a symplectic Lie algebroid and let
$B$ be a coisotropic Lie subalgebroid of\/ $A$
(Definition~\ref{definition;coisotropic}).  The pullback
$\omega_B=i_B^*\omega\in\Omega_B^2(N)$ is a closed $2$-form.  The rank
of\/ $\omega_B$ is constant equal to
$\rank(B/B^\omega)=2\rank(B)-\rank(A)$.  Hence $(B,\omega_B)$ is a
presymplectic Lie algebroid.

The symplectization theorem below asserts the converse: every
presymplectic Lie algebroid $B\to N$ arises as a pullback Lie
algebroid $i_N^!\bb{A}$ of a model symplectic Lie algebroid
$\bb{A}\to\bb{M}$ via a transverse coisotropic embedding $i_N\colon
N\to\bb{M}$.

The model $\bb{A}$ is constructed as follows.  The input data is any
presymplectic Lie algebroid $(B\to N,\omega_B)$.  Let
$K=\ker(\omega_B)$ be the kernel of\/ $\omega_B$, i.e.\ the bundle
whose fibre at $x\in N$ is
\[
K_x=\{\,b\in B_x\mid\text{$\omega_B(b,b')=0$ for all $b'\in B_x$}\,\}.
\]
We let $\bb{M}=K^*$ be the dual bundle of\/ $K$.  The bundle
projection $p\colon\bb{M}\to N$ is transverse to $B$, so we can form
the pullback Lie algebroid $\bb{A}=p^!B$ over~$\bb{M}$.  We identify
$N$ with the zero section $\group{j}\colon N\to\bb{M}=K^*$ and $B$
with the Lie algebroid $\group{j}^!\bb{A}\cong\group{j}^!p^!B$.  The
inclusion $\group{j}$ has a natural lift to a morphism
$\group{j}_!\colon B\to\bb{A}$.  The definition of the symplectic
structure on $\bb{A}$ involves the choice of a complement of the
subbundle $K$ of\/ $B$, i.e.\ a splitting $s\colon K^*\to B^*$ of the
natural surjection $B^*\to K^*$.  Let $\pi\colon B^*\to N$ be the
bundle projection of~$B^*$.  Then $p=\pi\circ s\colon K^*\to N$, so
\[\bb{A}=p^!B=(\pi\circ s)^!B=s^!\pi^!B.\]
We have canonical Lie algebroid morphisms
\[
p_!\colon\bb{A}\longto B,\qquad\pi_!\colon\pi^!B\longto B,\qquad
s_!\colon\bb{A}\longto\pi^!B
\]
satisfying $p_!=\pi_!\circ s_!$.
We define a closed $\bb{A}$-form of degree $2$ on $\bb{M}$ by
\[
\omega^s=p_!^*\omega_B+s_!^*\omega_\can,
\]
where $\omega_\can$ is the canonical symplectic form on the phase
space $\pi^!B$ (Proposition~\ref{proposition;phase}).  We call the
tuple
\begin{equation}\label{equation;model}
\bigl(\bb{A}\longto\bb{M},\omega^s,\group{j}_!\colon
B\longto\bb{A}\bigr)
\end{equation}
the \emph{symplectization} of\/ $(B,\omega_B)$.

\begin{theorem}[symplectization]\label{theorem;symplectization}
Let\/ $(B\to N,\omega_B)$ be a presymplectic Lie algebroid.  Let\/
$(\bb{A},\omega^s,\group{j}_!)$ be the
symplectization~\eqref{equation;model}.  There is an open
neighbourhood\/ $\bb{U}$ of\/ $N$ in\/ $\bb{M}$ such that\/
$\omega^s|_{\bb{U}}$ is symplectic.  The embedding\/ $\group{j}$ is
transverse coisotropic and\/ $\group{j}_!^*\omega^s=\omega_B$.
\end{theorem}

\begin{proof}
Identify $\bb{M}=K^*$ with the subbundle $s(K^*)$ of~$B^*$.  Then we
have $B=K\oplus L$, where $L=(K^*)^\circ$ is the annihilator of~$K^*$.
The bilinear form $\omega_B$ is nondegenerate on the subbundle~$L$.
Let $x\in N$.  Choose a frame $b_1$, $b_2,\dots$,~$b_r$ of\/ $B$
defined in a neighbourhood $U$ of\/ $x$ such that $K$ is spanned by
$b_1$, $b_2,\dots$,~$b_l$, and $L$ is spanned by $b_{l+1}$,
$b_{l+2},\dots$,~$b_r$.  Let $b_1^*$, $b_2^*,\dots$,~$b_r^*$ be the
dual frame of~$B^*$.  The subbundle $K^*$ is spanned by $b_1^*$,
$b_2^*,\dots$,~$b_l^*$.  The frame $b_1$, $b_2,\dots$,~$b_r$ gives
rise to a frame $e_1$, $e_2,\dots$,~$e_r$, $f_1$, $f_2,\dots$,~$f_r$
of\/ $\pi^!B$ over $\pi^{-1}(U)$ as in the proof of
Proposition~\ref{proposition;phase}.  The fibre $\bb{A}_x$ of\/
$\bb{A}$ is spanned by (the values at $x$ of) the sections $e_1$,
$e_2,\dots$,~$e_l$, $f_1$, $f_2,\dots$,~$f_r$.  The anchor of\/ $f_i$
is $\an(f_i)=b_i^*$, so the image of the anchor $\an_x(\bb{A})$
contains the span of\/ $b_1^*$, $b_2^*,\dots$,~$b_l^*$, i.e.\ the
fibre~$K_x^*$.  Therefore the zero section $\group{j}\colon
N\to\bb{M}=K^*$ is transverse to~$\bb{A}$.  It follows
from~\eqref{equation;ff}--\eqref{equation;ee} that
\begin{equation}\label{equation;model-form}
\begin{aligned}
\omega^s_x(f_i,f_j)&=0&&\quad\text{for $1\le i,j\le r$,}\\
\omega^s_x(e_i,e_j)&=0&&\quad\text{for $1\le i\le l$, $1\le j\le
  r$,}\\
\omega^s_x(f_i,e_j)&=\delta_{ij}&&\quad\text{for $1\le i\le l$, $1\le
  j\le r$,}\\
\omega^s_x(e_i,e_j)&=\omega_{B,x}(b_i,b_j)&&\quad\text{for $l+1\le
  i,j\le r$.}
\end{aligned}
\end{equation}
This shows that $\group{j}_!^*\omega^s=\omega_B$.  Also, the fibre
$(\bb{A}_x,\omega^s_x)$ is an orthogonal direct sum of two symplectic
subspaces $K_x\oplus K_x^*$ and $(L_x,\omega_{B,x})$.  Hence the form
$\omega^s$ is symplectic near~$N$.  The fibre $B_x$ is spanned by
$e_1$, $e_2,\dots$,~$e_r$, so its orthogonal $B_x^{\omega^s}$ is
spanned by $e_1$, $e_2,\dots$,~$e_l$.  This shows that the embedding
$\group{j}$ is transverse to $\bb{A}$ and coisotropic.
\end{proof}

\begin{remark}\label{remark;path}
We will see in Section~\ref{subsection;gotay} that the form $\omega^s$
is independent of the splitting $s$ up to Lie algebroid automorphisms
of\/ $A$ fixing $N$.  For now we note that any two splittings $s_0$,
$s_1\colon\bb{M}=K^*\to B^*$ of the surjection $B^*\to K^*$ can be
joined by a path $s_t=(1-t)s_0+ts_1$ for $0\le t\le1$.  The
corresponding path of\/ $\bb{A}$-symplectic forms
$\omega_t=\omega^{s_t}$ satisfies $\dot{\omega}_t=-d_A\beta_t$, where
$\beta_t\in\Omega_{\bb{A}}^1(\bb{M})$ is defined by
\[\beta_t=\frac{d}{dt}(s_t)_!^*\alpha_\can.\]
Since $s_t(x)=x$ for all $x\in N$ and the Liouville form
$\alpha_\can\in\Omega_{\pi^!B}(N)$ vanishes along $N$, the form
$\beta_t$ vanishes along $N$ for all $t$.
\end{remark}

\subsection{Presymplectic Lie algebroids: reduction}
\label{subsection;reduction}

As we saw in Section~\ref{subsection;symplectization}, every
presymplectic Lie algebroid can be symplectized, i.e.\ coisotropically
embedded in a symplectic Lie algebroid.  A second method to produce
symplectic Lie algebroids out of presymplectic Lie algebroids, which
works only under favourable conditions, is \emph{symplectic
  reduction}, which means taking the quotient by the null foliation.
This is based on the following facts.

\begin{lemma}\label{lemma;constant-rank}
Let $(B\to N,\omega_B)$ be a presymplectic Lie algebroid.  The kernel
$K=\ker(\omega_B)$ is a Lie subalgebroid of\/ $B$\@.  The form
$\omega_B$ is $K$-basic in the sense that for all sections $\sigma$ of
$K$ we have $\iota_B(\sigma)\omega_B=\ca{L}_B(\sigma)\omega_B=0$.
\end{lemma}

\begin{proof}
This follows from $d_B\omega_B=0$ and
$\ca{L}_B(\sigma)=[\iota_B(\sigma),d_B]$.
\end{proof}

We call the Lie subalgebroid $K$ of Lemma~\ref{lemma;constant-rank}
the \emph{null Lie algebroid} of the form~$\omega_B$.

A \emph{foliation Lie algebroid} is a Lie algebroid whose anchor is
injective.  A foliation Lie algebroid over a manifold $P$ is
equivalent to an involutive subbundle of\/ $TP$, in other words a
(nonsingular) foliation of~$P$.

\begin{lemma}\label{lemma;reduction}
Let $(B\to N,\omega_B)$ be a presymplectic Lie algebroid.  Suppose the
null Lie algebroid $K=\ker(\omega_B)$ is a foliation Lie algebroid and
therefore defines a foliation $\ca{K}$ of~$N$\@.  Let $i_S\colon S\to
N$ be a transverse slice of\/ $\ca{K}$\@.  Then $(i_S)_!^*\omega_B$ is
an $i_S^!B$-symplectic form on~$S$.
\end{lemma}

\begin{proof}
Since $S$ is transverse to the leaves of\/ $\ca{K}$, it is also
transverse to the Lie algebroid $B$
(Proposition~\ref{proposition;pull}\eqref{item;pull-clean-transverse}),
so the pullback $i_S^!B=\an^{-1}(TS)\subseteq B$ and the morphism
$(i_S)_!\colon i_S^!B\to B$ are well-defined.  Let $x\in S$.  Since
$S$ is a transverse slice to the foliation, the tangent space to $N$
is a direct sum $T_xN=\an(K_x)\oplus T_xS$.  Hence, the anchor
$\an\colon K_x\to T_xN$ being injective, the fibre of\/ $B$ is
likewise a direct sum
$B_x=K_x\oplus\an^{-1}(T_xS)=K_x\oplus(i_S^!B)_x$.  It follows that
$\omega_B$ restricts to a nondegenerate form on~$i_S^!B$.  Therefore
$\bigl(i_S^!B,(i_S)_!^*\omega_B\bigr)$ is a symplectic Lie algebroid.
\end{proof}

We call the foliation $\ca{K}$ of Lemma~\ref{lemma;reduction} the
\emph{null foliation} of~$\omega_B$.

Let $B\to N$ and $K\to N$ be Lie algebroids over the same base and let
$f\colon K\to B$ be a morphism over the identity map of~$N$.  We
define a \emph{quotient} of\/ $B$ by $K$ to be a pair $(C\to Q,q)$
consisting of a Lie algebroid $C\to Q$ and a morphism $q\colon B\to C$
with $q\circ f(K)=0$ which has the following universal property: for
every Lie algebroid $A\to M$ and every morphism $g\colon B\to A$ with
$g\circ f(K)=0$ there is a unique morphism $g_C\colon C\to A$ with
$g=g_C\circ q$, as in the diagram
\[
\begin{tikzcd}
B\ar[d,"q"']\ar[r,"g"]&A\\
C\ar[ur,dashed,"g_C"']
\end{tikzcd}
\]

Clearly a quotient Lie algebroid, if it exists, is determined uniquely
up to isomorphism by the Lie algebroid morphism $f\colon K\to B$.  The
next result, which is a special case
of~\cite[Theorem~4.5]{higgins-mackenzie;algebraic-lie-algebroids} and
which does not involve any presymplectic structures, states a
sufficient condition for a quotient Lie algebroid to exist.  We say
that the leaf space $Q=P/\F$ of a foliation $\F$ of a manifold $P$
\emph{is a manifold} if\/ $Q$ has a manifold structure which makes the
quotient map $P\to Q$ a submersion.

\begin{proposition}[quotient Lie algebroids]
\label{proposition;quotient}
Let $B\to N$ be a Lie algebroid and $K\to N$ a Lie subalgebroid
of~$B$.  Suppose that $K$ is a foliation Lie algebroid with associated
foliation $\ca{K}$ and that the leaf space $Q=N/\ca{K}$ is a manifold.
Let $\mr{q}\colon N\to Q$ be the quotient map.  Let $\bar{B}$ be the
bundle $B/K$ over $N$ and identify sections
$\bar\tau\in\Gamma(\bar{B})\cong\Gamma(B)/\Gamma(K)$ with equivalence
classes of sections $\tau$ of\/ $B$ modulo sections of\/ $K$\@.
Suppose that the flat $K$-connection
\[
\nabla\colon\Gamma(K)\times\Gamma(\bar{B})\longto\Gamma(\bar{B})
\]
on $\bar{B}$ defined by
$\nabla_\sigma\bar\tau=\overline{[\sigma,\tau]}$ has trivial holonomy.
\begin{enumerate}
\item\label{item;horizontal}
The $\nabla$-horizontal subspaces of\/ $T\bar{B}$ define a foliation
$\ca{L}$ of\/ $\bar{B}$ which makes the pair $(\bar{B},\ca{L})$ a
foliated vector bundle over the foliated manifold $(N,\ca{K})$.
\item\label{item;quotient}
The leaf space $C=\bar{B}/\ca{L}$ is a manifold.  Let $q\colon B\to C$
the quotient map and $C\to Q$ the projection induced by the bundle
projection $\bar{B}\to N$.  Then $(C\to Q,q)$ is a quotient Lie
algebroid of\/ $B$ by $K$\@.  The quotient morphism $q\colon B\to C$
induces an isomorphism $B\cong q^!C$.
\item\label{item;section}
Let $s\colon U\to N$ be a section of\/ $\mr{q}$ defined over an open
subset $U$ of~$Q$.  The natural map $q\circ s_!\colon s^!B\to B\to C$
is a Lie algebroid isomorphism from $s^!B$ onto~$C|_U$.
\item\label{item;normal}
The Lie algebra of sections of\/ $C$ is isomorphic to\/
$\Gamma(C)\cong\n(\Gamma(K))/\Gamma(K)$, where $\n(\Gamma(K))$ denotes
the normalizer of\/ $\Gamma(K)$ in $\Gamma(B)$.
\end{enumerate}
\end{proposition}

Under the conditions of Proposition~\ref{proposition;quotient} the
quotient morphism $q\colon B\to C$ induces an isomorphism of complexes
\begin{equation}\label{equation;basic}
\begin{tikzcd}
q^*\colon\Omega_C^\bu(Q)\ar[r,"\cong"]&
\Omega_B^\bu(N)_{\text{\rmfamily$K$-bas}},
\end{tikzcd}
\end{equation}
where the subscript ``$K$-bas'' refers to the subcomplex of Lie
algebroid forms that are $K$-basic in the sense of
Lemma~\ref{lemma;constant-rank}.  If\/ $\alpha\in\Omega_B^\bu(N)$ is
$K$-basic, then $(q^*)^{-1}\alpha$ is determined by the following
fact: for a local section $s\colon U\to N$ of\/ $q$ defined over an open
subset $U$ of\/ $Q$ we have
$\bigl((q^*)^{-1}\alpha)\bigr)\big|_U=s_!^*\alpha$.  Together with
Lemma~\ref{lemma;reduction} this gives us the following criterion for
when a presymplectic Lie algebroid can be reduced to a symplectic Lie
algebroid.

\begin{theorem}[symplectic reduction]
\label{theorem;reduction}
Let\/ $(B\to N,\omega_B)$ be a presymplectic Lie algebroid.  Suppose
that the null Lie algebroid\/ $K=\ker(\omega_B)$ is a foliation Lie
algebroid, that the leaf space $N/\ca{K}$ of the foliation\/ $\ca{K}$
defined by $K$ is a manifold, and that the flat $K$-connection on
$B/K$ has trivial holonomy.  Then there is a unique form\/
$\omega_C\in\Omega_C^2(Q)$ on the quotient Lie algebroid $C\to Q$
satisfying\ $q^*\omega_C=\omega_B$.  The form\/ $\omega_C$ is
$C$-symplectic.
\end{theorem}

\section{Hamiltonian actions and reduction}\label{section;hamilton}

Marsden and Weinstein%
~\cite{marsden-weinstein;reduction-symplectic-manifolds-symmetry}
showed how to reduce a symplectic manifold $M$ with respect to a
moment map, i.e.\ a Poisson map $M\to\g^*$ to the dual of a Lie
algebra~$\g$.  A version of their result for symplectic Lie algebroids
was obtained by Marrero et
al.~\cite[Theorem~3.11]{marrero-padron-rodriguez;symplectic-like}.
However, symplectic Lie algebroids include log symplectic manifolds,
where the symplectic structure has first-order poles and Hamiltonian
functions may have logarithmic poles.  How do we reduce log symplectic
manifolds at ``infinite'' values of momentum?  Following Mikami and
Weinstein~\cite{mikami-weinstein;moments-reduction-groupoids}, we will
deal with such situations by allowing moment maps to take values in
Poisson manifolds more general than~$\g^*$.  The upshot is a Lie
algebroid version of the Mikami-Weinstein reduction theorem,
Theorem~\ref{theorem;mikami-weinstein}.

As noted in Section~\ref{subsection;symplectic}, a Poisson structure
on a Lie algebroid $A\to M$ gives rise to a Poisson structure on $M$
in the usual sense, and in that sense
Theorem~\ref{theorem;mikami-weinstein} is a special case of the
Poisson reduction theorems of Marsden and
Ratiu~\cite{marsden-ratiu;reduction-poisson} and Cattaneo and
Zambon~\cite{cattaneo-zambon;supergeometric-poisson}.  But
Theorem~\ref{theorem;mikami-weinstein} offers the extra information
that the quotient Poisson structure lifts to an appropriate quotient
Lie algebroid over the reduced space.  A Poisson structure on a Lie
algebroid $A$ can be seen as a special type of Dirac structure on the
Courant algebroid $A\oplus A^*$.  Presumably the Dirac reduction
theorems of Bursztyn and Crainic~\cite[Theorem~4.11]%
{bursztyn-crainic;dirac-structures-moment-maps} and Bursztyn et
al.~\cite[Proposition~3.16]{bursztyn-iglesias-severa;courant-moment},
which are formulated there only for standard Courant algebroids
$TM\oplus T^*M$, can be extended to incorporate our setting, but we
will leave that for another day.

\subsection{Lie algebroid actions and Poisson morphisms}
\label{subsection;action-morphism}

In this section we review the notion of a Lie algebroid action on a
Lie algebroid and how Poisson maps give rise to such actions.

An action of a Lie algebra $\g$ on a Lie algebroid $A\to M$ simply
means a Lie algebra homomorphism $\g\to\Gamma(A)$.  Equivalently, a
$\g$-action on $A$ can be defined as a $\g$-action on $M$ together
with a Lie algebroid morphism $\g\ltimes M\to A$ from the action Lie
algebroid to~$A$.  This notion generalizes as follows.

\begin{definition}\label{definition;action}
Let $A\to M$ and $C\to P$ be Lie algebroids and let $\mu_0\colon M\to
P$ be a smooth map.  An \emph{action} of\/ $C$ on $A$ with
\emph{anchor} $\mu_0$ is a Lie algebra homomorphism
$\rho\colon\Gamma(C)\to\Gamma(A)$ that is $\ca{C}^\infty$-linear with
respect to $\mu_0$ in the sense that
$\rho(f\tau)=(\mu_0^*f)\rho(\tau)$ for $f\in\ca{C}^\infty(P)$ and
$\tau\in\Gamma(C)$.  The sections $\rho(\tau)\in\Gamma(A)$, where
$\tau$ ranges over the space of sections of\/ $C$, are called the
\emph{generating sections} of the action.
\end{definition}

Given a $C$-action $\rho$ on $A$ there is a unique smooth vector
bundle map $\bar\rho\colon\mu_0^*C\to A$ over the identity $\id_M$
such that the triangle
\begin{equation}\label{equation;algebroid-action}
\begin{tikzcd}
\Gamma(C)\ar[d,"\mu_0^*"']\ar[r,"\rho"]&\Gamma(A)\\
\Gamma(\mu_0^*C)\ar[ur,"\bar\rho_*"']
\end{tikzcd}
\end{equation}
commutes.  The pullback bundle $\mu_0^*C$ is equipped with an anchor
$\an_{\mu_0^*C}={\an_A}\circ\bar\rho\colon\mu_0^*C\to TM$.  There is a
unique Lie bracket on $\Gamma(\mu_0^*C)$ that makes the
maps~\eqref{equation;algebroid-action} Lie algebra homomorphisms.
With respect to this bracket $\mu_0^*C$ is a Lie algebroid over $M$,
and $\bar\rho$ and the natural map $\mu_0^*C\to C$ are Lie algebroid
morphisms.  Thus an action of\/ $C$ on $A$ can be alternatively
defined as consisting of a smooth map $\mu_0\colon M\to P$, a Lie
algebroid structure on the pullback bundle $\mu_0^*C$ such that
$\mu_0^*C\to C$ is a Lie algebroid morphism, and a Lie algebroid
morphism $\bar\rho\colon\mu_0^*C\to A$.

\begin{notation-definition}\label{notation;action}
To lighten the notation we will from now on denote the action map
$\rho\colon\Gamma(C)\to\Gamma(A)$, the vector bundle map
$\bar\rho\colon\mu_0^*C\to A$, and the pushforward map
$\bar\rho_*\colon\Gamma(\mu_0^*C)\to\Gamma(A)$ all by the same
letter,~$\rho$.  We will also denote by $\rho_x\colon C_{\mu_0(x)}\to
A_x$ the restriction of\/ $\bar\rho$ to a point $x\in M$.  We say that
the action $\rho$ is \emph{locally free} at $x$ if\/ $\rho_x$ is
injective and \emph{transitive} at $x$ if\/ $\rho_x$ is surjective.
\end{notation-definition}

The fact that $\mu_0^*C\to C$ is a Lie algebroid morphism can be
viewed as an equivariance property of the anchor $\mu_0$; it implies
that the square
\[
\begin{tikzcd}
\mu_0^*C\ar[r]\ar[d,"\an_{\mu_0^*C}"']&C\ar[d,"\an_C"]\\
TM\ar[r,"T\mu_0"]&TP
\end{tikzcd}
\]
commutes.  In particular we have inclusions
\begin{equation}\label{equation;stabilizers}
\stab(\mu_0^*C,x)\subseteq\stab(C,\mu_0(x))
\end{equation}
for all $x\in M$, where we identify the fibre of\/ $\mu_0^*C$ at $x$
with the fibre of\/ $C$ at $\mu_0(x)$.

If\/ $B\to N$ is a Lie subalgebroid of\/ $A$, then the
$\ca{C}^\infty(M)$-module of relative sections
\[\Gamma(A;B)=\{\,\tau\in\Gamma(A)\mid\tau|_N\in\Gamma(B)\,\}\]
is a Lie subalgebra of\/ $\Gamma(A)$ and therefore the
$\ca{C}^\infty(P)$-module $\rho^{-1}(\Gamma(A;B))$ is a Lie subalgebra
of\/ $\Gamma(C)$.  It follows from the Leibniz rule that the
$\ca{C}^\infty(P)$-module
\[
\Gamma\bigl(C;0_{\mu_0(N)}\bigr)=
\{\,\sigma\in\Gamma(C)\mid\text{$\sigma_p=0$ for all
  $p\in\mu_0(N)$}\,\}
\]
is a Lie ideal of\/ $\rho^{-1}(\Gamma(A;B))$.  This leads to the
following definition.

\begin{definition}\label{definition;stable}
Given an action $\rho\colon\Gamma(C)\to\Gamma(A)$ of\/ $C$ on $A$ with
anchor $\mu_0\colon M\to P$ and a Lie subalgebroid $B\to N$ of\/ $A$,
the \emph{stabilizer} of\/ $B$ under the action is the Lie algebra
\[
\stab(\rho,B)=\stab(C,B)=
\rho^{-1}\bigl(\Gamma(A;B)\bigr)\big/\Gamma\bigl(C;0_{\mu_0(N)}\bigr).
\]
Let $D$ be a Lie subalgebroid of\/ $C$ with base manifold $Q\subseteq
P$.  We say $B$ is \emph{stable under $D$} or \emph{$D$-stable} if
$\mu_0(N)\subseteq Q$ and
$\Gamma(C;D)\subseteq\rho^{-1}(\Gamma(A;B))$.
\glossary{stabCB@$\stab(C,B)$, stabilizer of\/ $B$ under $C$-action}
\end{definition}

\begin{remarks}\phantomsection\label{remark;stabilizer}
\begin{numerate}
\item\label{item;lie-rinehart}
The stabilizer $\stab(C,B)$ is a Lie-Rinehart algebra in the sense
of~\cite{huebschmann;history-lie-rinehart}, but as the next remark
shows it is usually not the space of sections of a Lie subalgebroid
of~$C$.
\item\label{item;stabilizer-not-bundle}
The Lie algebroid $A\to M$ acts on the tangent bundle $TM$ via the
anchor $A\to TM$.  The stabilizer of a point $x\in M$ (viewed as the
zero bundle $0_x\subseteq TM$) in the sense of
Definition~\ref{definition;stable} is $\stab(A,x)=\ker(\an_{A,x})$,
which agrees with the usual definition
(Remark~\ref{remark;clean}\eqref{item;point-orbit}).  The stabilizer
of the zero subalgebroid $0_M$ is
$\stab(A,0_M)=\ker(\an_A)\subseteq\Gamma(A)$, which is not the space
of sections of a subbundle of\/ $A$ unless the anchor has constant
rank.
%
\item\label{item;adapted}
Definition~\ref{definition;stable} is correct only if the submanifold
$N$ is closed and embedded, which will always be the case in the
situations that concern us.  (If\/ $N$ is not closed or embedded the
definition must be modified as follows.  Call a pair of open subsets
$U\subseteq M$ and $V\subseteq N$ \emph{adapted} to $N$ if\/ $V$ is a
closed embedded submanifold of\/ $U$.  There is a unique sheaf of
$\ca{C}^\infty_N$-modules $\ca{S}$ such that
$\ca{S}(V)=\stab(C|_U,B|_V)$ for every adapted pair $(U,V)$.  Define
$\stab(C,B)=\ca{S}(N)$ to be the space of global sections of
$\ca{S}$.)
\end{numerate}
\end{remarks}

Lie algebroid actions can be restricted to Lie subalgebroids in the
following way.

\begin{lemma}\label{lemma;action}
Let $A\to M$ and $C\to P$ be Lie algebroids and let
$\rho\colon\Gamma(C)\to\Gamma(A)$ be a $C$-action on $A$ with anchor
$\mu_0\colon M\to P$.  Let $B\to N$ be a Lie subalgebroid of\/ $A$.
The action $\rho$ restricts to a Lie algebra homomorphism
$\stab(C,B)\to\Gamma(B)$.
If\/ $B$ is stable under the action of a Lie subalgebroid $D$ of\/
$C$, we have a natural homomorphism $\Gamma(D)\to\stab(C,B)$ and hence
a homomorphism $\rho_D\colon\Gamma(D)\to\Gamma(B)$, which is a
$D$-action on $B$ with anchor $\mu_0|_N$.
\end{lemma}

\begin{proof}
If a section $\sigma\in\Gamma(C)$ vanishes on $\mu_0(N)$, then by
$\ca{C}^\infty$-linearity $\rho(\sigma)$ vanishes on $N$.  Therefore
$\rho$ descends to a homomorphism
\[
\stab(C,B)=
\rho^{-1}\bigl(\Gamma(A;B)\bigr)\big/\Gamma\bigl(C;0_{\mu_0(N)}\bigr)
\longto\Gamma(B)\cong\Gamma(A;B)/\Gamma(A;0_N).
\]
If\/ $B$ is stable under $D\to Q$, then the inclusion
$\Gamma(C;D)\subseteq\rho^{-1}(\Gamma(A;B))$ gives us a homomorphism
\[
\Gamma(D)\cong\Gamma(C;D)/\Gamma(C;0_Q)\longto\stab(C,B).
\]
The composition of these two homomorphisms is $\ca{C}^\infty$-linear
over $\mu_0|_N$.
\end{proof}

Recall that a Poisson structure on a Lie algebroid $A$ gives rise to a
Lie algebroid structure on the dual bundle $A^*$
(Theorem~\ref{theorem;poisson}).  The following lemma shows that a
Poisson morphism gives rise to a Lie algebroid action in the same way
that a moment map $M\to\g^*$ gives rise to a Hamiltonian Lie algebra
action $\g\to\Gamma(TM)$ on a symplectic manifold.

\begin{lemma}\label{lemma;poisson-action}
Let $(A\to M,\lambda_A)$ and $(E\to P,\lambda_E)$ be Poisson Lie
algebroids.
\begin{enumerate}
\item\label{item;poisson-pull}
Let $\mu\colon A\to E$ be a Poisson morphism.  The pullback map
\[
\mu^*\colon\Gamma(E^*)=\Omega_E^1(P)\longto\Gamma(A^*)= \Omega_A^1(M)
\]
defines an action of the Lie algebroid $E^*$ on the Lie algebroid
$A^*$ with anchor the base map $\mr\mu\colon M\to P$ of~$\mu$.
\item\label{item;poisson-hamiltonian}
Let $\rho\colon\Omega_E^1(P)\to\Omega_A^1(M)$ be an $E^*$-action on
$A^*$ with anchor $\mu_0\colon M\to P$.  Suppose that the conditions
$\rho\circ d_E=d_A\circ\rho$ and
$\rho^*\circ\lambda_A^\sharp\circ\rho=\lambda_E^\sharp$ are satisfied.
Then there exists a unique Poisson morphism $\mu\colon A\to E$ with
base map $\mr\mu=\mu_0$ such that $\mu^*=\rho$.
\end{enumerate}
\end{lemma}

\begin{proof}
\eqref{item;poisson-pull}~The pullback map $\mu^*$ is
$\ca{C}^\infty$-linear.  It is a Lie algebra homomorphism with respect
to the bracket~\eqref{equation;bracket-form} because the $2$-sections
$\lambda_A$ and $\lambda_E$ are $\mu$-related.

\eqref{item;poisson-hamiltonian}~Let $\rho^*\colon A\to\mu_0^*E$ be
the transpose of\/ $\rho$ and $\mu\colon A\to E$ the composition of
$\rho^*$ with the natural map $\mu_0^*E\to E$.  Then $\mu$ is the
unique vector bundle map with base map $\mu_0$ such that the pullback
map on sections $\mu^*\colon\Gamma(E^*)\to\Gamma(A^*)$ coincides
with~$\rho$.  Since $\rho$ commutes with the exterior derivative, by
Vaintrob's theorem~\cite{vaintrob;lie-algebroids-homological}
(cf.\ also~\cite[\S\,12.2]{meinrenken;groupoids-algebroids}) $\mu$ is
a Lie algebroid morphism.  The condition
$\rho^*\circ\lambda_A^\sharp\circ\rho=\lambda_E^\sharp$ gives that
$\mu\circ\lambda_A^\sharp\circ\mu^*=\lambda_E^\sharp$, i.e.\ $\mu$ is
Poisson.
\end{proof}

\begin{definition}\label{definition;hamiltonian}
Let $\mu\colon A\to E$ be a Poisson morphism of Poisson Lie algebroids
$(A\to M,\lambda_A)$ and $(E\to P,\lambda_E)$ and let $\mu^*$ be the
$E^*$-action on $A^*$ of
Lemma~\ref{lemma;poisson-action}\eqref{item;poisson-pull}.  The
$E^*$-action $\gamma$ on $A$ with anchor $\mr\mu$ obtained by
composing the maps
\[
\begin{tikzcd}
\gamma\colon\Omega_E^1(P)\ar[r,"\mu^*"]&
\Omega_A^1(M)\ar[r,"\lambda_A^\sharp"]&\Gamma(A)
\end{tikzcd}
\]
is the \emph{Hamiltonian action} of\/ $E^*$ on $A$ with \emph{moment}
$\mu\colon A\to E$.  A function $f\in\ca{C}^\infty(M)$ is
\emph{collective} for the Hamiltonian action if it is of the form
$f=g\circ\mr\mu$ for some $g\in\ca{C}^\infty(P)$.
\end{definition}

\begin{remark}\label{remark;collective}
Under the Hamiltonian action $\gamma$ a collective function
$f=g\circ\mr\mu$ acts by the Hamiltonian section
\[
\gamma(d_Eg)=\lambda_A^\sharp\mu^*d_Eg=\lambda_A^\sharp
d_A\mr\mu^*g=\lambda_A^\sharp d_Af=\sigma_f
\]
(see~\eqref{equation;hamilton-section}), which leaves $\lambda_A$
invariant.
\end{remark}

For each $x\in M$ the Hamiltonian action $\gamma$ defines a linear map
on the fibres $\gamma_x\colon E_{\mr\mu(x)}^*\to A_x$ (notational
convention~\ref{notation;action}).  The following lemma describes the
kernel and the image of~$\gamma_x$.  Recall that $W^\circ\subseteq
V^*$ denotes the annihilator of a subspace $W$ of a vector space $V$.

\begin{lemma}\label{lemma;moment}
Let $(A\to M,\lambda_A)$ and $(E\to P,\lambda_E)$ be Poisson Lie
algebroids.  Let $\mu\colon A\to E$ be a Poisson morphism and let
$\gamma\colon\Gamma(E^*)\to\Gamma(A)$ be the associated Hamiltonian
action.  Let $x\in M$ and $y=\mr\mu(x)\in P$.  Let $\gamma_x^*\colon
A_x^*\to E_y$ be the transpose of\/ $\gamma_x\colon E_y^*\to A_x$.
\begin{enumerate}
\item\label{item;transpose}
We have the identities
\[
\gamma_x=\lambda_{A,x}^\sharp\circ\mu_x^*,\qquad
\gamma_x^*=-\mu_x\circ\lambda_{A,x}^\sharp,\qquad
\lambda_{E,y}^\sharp=\mu_x\circ\gamma_x=-\gamma_x^*\circ\mu_x^*.
\]
\item\label{item;kernel}
For every linear subspace $\liel$ of\/ $E_y^*$ we have
\[
\bigl(\mu_x\circ\lambda_{A,x}^\sharp\bigr)^{-1}(\liel^\circ)=
(\gamma_x(\liel))^\circ.
\]
In particular
$\ker\bigl(\mu_x\circ\lambda_{A,x}^\sharp\bigr)=(\im(\gamma_x))^\circ$.
\item\label{item;image}
For every linear subspace $L$ of\/ $A_x$ we have
\[
\bigl(\mu_x\circ\lambda_{A,x}^\sharp\bigr)(L^\circ)=
(\gamma_x^{-1}(L))^\circ.
\]
In particular
$\im\bigl(\mu_x\circ\lambda_{A,x}^\sharp\bigr)=(\ker(\gamma_x))^\circ$.
\end{enumerate}
\end{lemma}

\begin{proof}
\eqref{item;transpose}~This follows from the commutativity of the
diagram
\[
\begin{tikzcd}
A_x^*\ar[r,"\lambda_{A,x}^\sharp"]&A_x\ar[d,"\mu_x"]\\
E_y^*\ar[r,"\lambda_{E,y}^\sharp"]\ar[u,"\mu_x^*"]
\ar[ur,"\gamma_x"]&E_y
\end{tikzcd}
\]
(Definition~\ref{definition;poisson-morphism}) and from the fact that
$\lambda_A^\sharp$ and $\lambda_E^\sharp$ are antisymmetric.

\eqref{item;kernel} and~\eqref{item;image} follow immediately
from~\eqref{item;transpose}.
\end{proof}

\begin{remark}\label{remark;equivariant}
The identity morphism $E\to E$ is Poisson and therefore generates a
Hamiltonian action $\lambda_E^\sharp\colon\Gamma(E^*)\to\Gamma(E)$.
For $E=T\g^*$, the tangent bundle of the dual of a Lie algebra $\g$,
this is the coadjoint action of\/ $\g$ on $\g^*$.  The equality
$\lambda_{E,y}^\sharp=\mu_x\circ\gamma_x$ of
Lemma~\ref{lemma;moment}\eqref{item;transpose} expresses the
\emph{equivariance} of the moment $\mu$ with respect to the actions
$\gamma$ and $\lambda_E^\sharp$: for every $1$-form
$\alpha\in\Omega_E^1(P)$ the $1$-form $\gamma(\alpha)\in\Omega_A^1(M)$
and the section $\lambda_E^\sharp(\alpha)\in\Gamma(E)$ are
$\mu$-related.
\end{remark}

\subsection{Hamiltonian reduction}\label{subsection;hamilton}

Let $\mu\colon A\to E$ be a Poisson morphism between two Poisson Lie
algebroids $(A\to M,\lambda_A)$ and $(E\to P,\lambda_E)$ and let
$\gamma$ be the associated Hamiltonian action of the Lie algebroid
$E^*$ on the Lie algebroid $A$
(Definition~\ref{definition;hamiltonian}).  \emph{Reducing} the
Hamiltonian action means choosing a suitable Lie subalgebroid $B$ of
$A$, forming the quotient of\/ $B$ by (a suitable restriction of) the
$E^*$-action, and pushing the $A$-Poisson structure on $M$ down to the
quotient Lie algebroid $B/E^*$.

We will treat here the case of most interest to us, namely where the
Lie algebroid $A$ is symplectic and $B$ is the moment fibre of a
``point'' of\/ $E$.  In the realm of Lie algebroids a ``point'' means
a Lie algebra.  Accordingly, by a ``point'' of\/ $E$ we will mean a
pair $(p,\f)$, where $p$ is a point of\/ $P$ and $\f$ is a Lie
subalgebra of the stabilizer $\stab(E,p)$.  At a point $p$ where the
anchor of\/ $E$ is not injective, various choices for $\f$ are
possible, which lead to different reduced symplectic Lie algebroids.

Theorem~\ref{theorem;mikami-weinstein} departs in two other ways from
the Marsden-Weinstein
theorem~\cite{marsden-weinstein;reduction-symplectic-manifolds-symmetry}
and the Mikami-Weinstein
theorem~\cite{mikami-weinstein;moments-reduction-groupoids}.  Firstly,
whereas in the classical setting regularity of the moment map and
freeness of the action are narrowly related, this relationship is
looser in our context and these two conditions need to be imposed
separately (but see Lemma~\ref{lemma;regular}).  Secondly, the
stabilizer $\stab(E^*,\f)$ of the ``point'' $\f$ does not preserve the
null foliation of the fibre, so to obtain a quotient which is
symplectic we must pass to a proper subalgebra of the stabilizer.


\begin{theorem}[Hamiltonian reduction for symplectic Lie algebroids]%
\label{theorem;mikami-weinstein}
Let\/ $(A\to M,\omega)$ be a symplectic Lie algebroid and let $(E\to
P,\lambda)$ be a Poisson Lie algebroid.  Let $\mu\colon A\to E$ be a
Poisson morphism and let $\gamma\colon\Gamma(E^*)\to\Gamma(A)$ be the
associated Hamiltonian action.  Let $p\in P$ and $N=\mr\mu^{-1}(p)$.
Suppose that $p$ is a regular value of\/ $\mr\mu\colon M\to P$ and
that the $E^*$-action on $TM$ defined by
$\an_A\circ\gamma\colon\Gamma(E^*)\to\Gamma(TM)$ is locally free at
$x$ for all $x\in N$.
\begin{enumerate}
\item\label{item;lie-fibre-act}
Let\/ $\f$ be a Lie subalgebra of\/ $\stab(E,p)$.  Then $\mu$ is
transverse to $\f$, $B=\mu^{-1}(\f)$ is a Lie subalgebroid of\/ $A$,
whose base is the submanifold $N$ of~$M$, and the $2$-form
$\omega_B=i_B^*\omega$ on $B$ is presymplectic.  The null Lie
algebroid $K=\ker(\omega_B)$ is a foliation Lie algebroid and the Lie
subalgebra
\[\h=(\lambda_p^\sharp)^{-1}(\f)\cap\f^\circ\]
of\/ $\stab(E^*,\f)$ acts transitively on~$K$.
\item\label{item;reduction}
Suppose that the leaf space $Q=N/\h$ is a manifold and that the
canonical flat $K$-connection on $B/K$ has trivial holonomy.  Then
$\omega_B$ descends to a $C$-symplectic structure
$\omega_C\in\Omega_C^2(Q)$ on the quotient Lie algebroid $C=(B/K)/\h$.
\end{enumerate}
\end{theorem}

\begin{proof}
\eqref{item;lie-fibre-act}~The stabilizer at $x\in M$ of the
$E^*$-action on $TM$ is
\[
\ker(\an_A\circ\gamma_x)=\gamma_x^{-1}(\stab(A,x))=
\bigl(\mu_x(\stab(A,x)^\omega)\bigr)^\circ,
\]
where we used Lemma~\ref{lemma;moment}\eqref{item;image}.  Our local
freeness hypothesis amounts to $\ker(\an_A\circ\gamma_x)=0$, i.e.\
\begin{equation}\label{equation;free}
\mu_x(\stab(A,x)^\omega)=E_p
\end{equation}
for all $x\in N$.  It now follows from our regular value hypothesis
that the moment $\mu\colon A\to E$ is transverse to the Lie subalgebra
$\f$; see Remark~\ref{remark;clean-regular}.  Hence, by the regular
value theorem, Proposition~\ref{proposition;clean-regular},
$B=\mu^{-1}(\f)$ is a Lie subalgebroid over $N=\mr\mu^{-1}(p)$.  By
Definition~\ref{definition;stable} the stabilizer of\/ $B$ under the
action $\gamma$ is the Lie algebra
\[
\stab(E^*,B)=
\gamma^{-1}\bigl(\Gamma(A;B)\bigr)\big/\Gamma\bigl(E^*;0_{\mr\mu(N)}\bigr).
\]
If\/ $N$ is empty, the theorem is true for trivial reasons.  Suppose
now that $N$ is nonempty.  Then
\begin{align*}
\gamma^{-1}\bigl(\Gamma(A;B)\bigr)
&=\{\,\tau\in\Gamma(E^*)\mid\text{$\mu_x(\gamma_x(\tau))\in\f$ for all
  $x\in N$}\,\}\\
&=\{\,\tau\in\Gamma(E^*)\mid\lambda_p^\sharp(\tau)\in\f\,\},
\end{align*}
where we used the equivariance of\/ $\mu$,
Lemma~\ref{lemma;moment}\eqref{item;transpose}.  Also
\[
\Gamma\bigl(E^*;0_{\mr\mu(N)}\bigr)=\Gamma(E^*;0_p)=
\{\,\tau\in\Gamma(E^*)\mid\tau_p=0\,\},
\]
so
\[\stab(E^*,B)=(\lambda^\sharp)^{-1}(\f)=\stab(E^*,\f),\]
a Lie subalgebra of\/ $\stab(E^*,p)$.  By Lemma~\ref{lemma;action}
this stabilizer acts on $B$, but it does not preserve the intersection
$K=B\cap B^\omega$.  The fibre at $x\in N$ of the symplectic
orthogonal of\/ $B$ is
\[
B_x^\omega=\omega_x^\sharp(B_x^\circ)=
\bigl((\mu_x\circ\omega_x^\sharp)^{-1}(\f)\bigr)^\circ=\gamma_x(\f^\circ),
\]
where we used Lemma~\ref{lemma;moment}\eqref{item;transpose}.  This
yields
\[K_x=B_x\cap B_x^\omega=\mu_x^{-1}(\f)\cap\gamma_x(\f^\circ).\]
The local freeness hypothesis implies that $\gamma_x\colon E_p^*\to
A_x$ is injective for all $x\in N$.  Therefore
\[
\gamma_x^{-1}(K_x)=
\gamma_x^{-1}\bigl(\mu_x^{-1}(\f)\cap\gamma_x(\f^\circ)\bigr)=
\gamma_x^{-1}(\mu_x^{-1}(\f))\cap\f^\circ=
(\lambda_p^\sharp)^{-1}(\f)\cap\f^\circ=\h,
\]
where we used $\lambda_p^\sharp=\mu_x\circ\gamma_x$
(Lemma~\ref{lemma;moment}\eqref{item;transpose}).  Thus the action of
$\h$ preserves the Lie algebroid $K$ and for every $x\in N$ the map
$\gamma_x\colon E_p^*\to A_x$ maps $\h$ bijectively onto $K_x$.  It
follows that $\omega_B$ has constant rank and that $\h$ acts locally
freely and transitively on $K$ in the sense of
Definition~\ref{notation;action}.  In other words, $\gamma$ induces an
isomorphism from the action Lie algebroid $\h\ltimes N$ onto $K$ and
in particular the anchor of\/ $K$ is injective.

\eqref{item;reduction}~This follows from~\eqref{item;lie-fibre-act}
and Theorem~\ref{theorem;reduction}.
\end{proof}

\begin{definition}\label{definition;reduction}
The symplectic Lie algebroid $(C\to Q,\omega_C)$ is the
\emph{symplectic quotient} of\/ $A\to M$ at the ``point''~$(p,\f)$.
\end{definition}

\begin{example}[reduction with respect to the identity map]
\label{example;identity}
Let $\mu=\id_M$ be the identity map of\/ $M$.  The associated
Hamiltonian action is the identity morphism $\gamma=\id\colon A\to A$.
We can reduce $M$ at any ``point'' $(x,\liea)$, where $\liea$ is a Lie
subalgebra of\/ $\stab(A,x)$.  The symplectic quotient of\/ $M$ at the
``point'' $(x,\liea)$ is the point $x$ equipped with the symplectic
Lie algebra $\liea/(\liea\cap\liea^\omega)$.
\end{example}

\begin{remarks}\phantomsection\label{remark;reduction}
\begin{numerate}
\item\label{item;proper-free}
The hypothesis of
Theorem~\ref{theorem;mikami-weinstein}\eqref{item;reduction} is
satisfied if the action of\/ $\h$ on $B$ integrates to a proper and
free action of a Lie group $H$ with Lie algebra $\h$.
\item\label{item;example}
See Section~\ref{subsection;example} for further illustrations of
Theorem~\ref{theorem;mikami-weinstein}.
\end{numerate}
\end{remarks}

In the special case where the anchor of\/ $E$ is bijective at $p$,
i.e.\ $E\cong TP$ in a neighbourhood of~$p$, we have the following
relationship between freeness of the Hamiltonian action and regularity
of the moment map.

\begin{lemma}\label{lemma;regular}
In the context of Theorem~\ref{theorem;mikami-weinstein} suppose that
the anchor of\/ $E$ is bijective at $p$.  Let $x\in\mr\mu^{-1}(p)$ and
let $L$ be the symplectic leaf of\/ $x$ with respect to the Poisson
structure on $M$ determined by the $A$-symplectic form $\omega$.  Then
the action\/ $\an_A\circ\gamma\colon\Gamma(E^*)\to\Gamma(TM)$ is
locally free at $x$ if and only if\/ $x$ is a regular point of the map
$\mr\mu|_L$.
\end{lemma}

\begin{proof}
Since $\an_E$ is bijective at $p$ and
$\an_E\circ\mu=T\mr\mu\circ\an_A$, condition~\eqref{equation;free} is
equivalent to $T\mr\mu\circ\an_A(\stab(A,x)^\omega)=T_pP$.  Let
$\lambda_M$ be the Poisson structure determined by $\omega$.
Using~\eqref{equation;linear-poisson} we see that
$\an_A(\stab(A,x)^\omega)=\{0\}^{\lambda_M}=T_xL$.  So the
$E^*$-action on $M$ is locally free at $x$ if and only if
$T_x\mr\mu(T_xL)=T_pP$.
\end{proof}

\section{Some normal forms for symplectic Lie algebroids}
\label{section;normal}

In this section we establish a Lie algebroid version of the
Darboux-Moser-Weinstein theorem, Theorem~\ref{theorem;dmw} and of the
coisotropic embedding theorem, Theorem~\ref{theorem;coisotropic}.  Our
method is based on Lie algebroid homotopies, which were introduced
(for Lie algebroid paths) by Crainic and
Fernandes~\cite[\S\,1]{crainic-fernandes;integrability-lie}.  We
review in Section~\ref{subsection;homotopy} how a Lie algebroid
homotopy gives rise to a cochain homotopy of the de Rham complex.  Of
particular interest in the context of the Moser trick are deformation
retractions, which are a weak version of Lie algebroid splittings.
Section~\ref{subsection;retract} describes a technique for obtaining
deformation retractions of Lie algebroids based on the notion of an
Euler-like section of Bischoff et
al.~\cite{bischoff-bursztyn-lima-meinrenken} and Bursztyn et
al.~\cite{bursztyn-lima-meinrenken;splitting}.

\subsection{Lie algebroid homotopies}\label{subsection;homotopy}

Let $A\to M$ be a Lie algebroid.  As in ordinary de Rham theory the
operations $d_A$, $\iota_A$, and $\ca{L}_A$ lead naturally to homotopy
formulas and a Poincar\'e lemma for the de Rham complex of~$A$.  Let
$T\R=\R\times\R$ be the tangent bundle of\/ $\R$ and let $T\R\times
A\to\R\times M$ be the direct product Lie algebroid.  The Cartesian
projections $\pr_1\colon T\R\times A\to T\R$ and $\pr_2\colon
T\R\times A\to A$ are Lie algebroid morphisms.  For each $t$ the map
\begin{equation}\label{equation;cylinder-slice}
i_t\colon A\to T\R\times A
\end{equation}
defined by $i_t(a)=(0_t,a)$ (where $a\in A$ and where $0_t$ denotes
the origin of\/ $T_t\R$) is likewise a Lie algebroid morphism.  The
section $\tangent t$ of\/ $T\R$ can be regarded as the section
$(t,x)\mapsto(t,1,0_x)$ of\/ $T\R\times A$ and as such generates a flow
$\Upsilon_t$.  For each $t$ the morphism
\begin{equation}\label{equation;translate}
\Upsilon_t\colon T\R\times A\longto T\R\times A
\end{equation}
is simply the shift by $t$ in the $\R$-direction.  The infinite
cylinder $\R\times M$ contains as a submanifold the finite cylinder
$[0,1]\times M$, and the product Lie algebroid $T[0,1]\times A$ over
$[0,1]\times M$ is a Lie subalgebroid of the product $T\R\times A$.
The operator 
\[
\kappa\colon\Omega_{T[0,1]\times A}^k([0,1]\times
M)\longto\Omega_A^{k-1}(M)
\] 
is defined as follows: for $\alpha\in\Omega_{T[0,1]\times
  A}^k([0,1]\times M)$ and sections $\sigma_1$,
$\sigma_2,\dots$,~$\sigma_{k-1}\in\Gamma(A)$ put
\begin{equation}\label{equation;integral}
(\kappa\alpha)(\sigma_1,\sigma_2,\dots,\sigma_{k-1})=
  \int_0^1\alpha(\tangent t,\sigma_1,\sigma_2,\dots,\sigma_{k-1})\,dt.
\end{equation}
It follows from~\eqref{equation;lie} and~\eqref{equation;cartan} that
\begin{align*}
\kappa d_A\alpha+d_A\kappa\alpha&= \int_0^1\iota_A(\tangent
t)d_A\alpha\,dt+ d_A\int_0^1\iota_A(\tangent t)\alpha\,dt\\
&=\int_0^1\bigl(\iota_A(\tangent t)d_A\alpha+d_A\iota_A(\tangent
t)\alpha\bigr)\,dt\\
&=\int_0^1\ca{L}_A(\tangent t)\alpha\,dt=
\int_0^1\frac{d}{du}\Upsilon_u^*\alpha\big|_{u=0}\,dt=
i_1^*\alpha-i_0^*\alpha,
\end{align*}
where $\Upsilon_t$ is as in~\eqref{equation;translate}.  This proves
\begin{equation}\label{equation;cylinder}
[\kappa,d_A]=i_1^*-i_0^*.
\end{equation}

\begin{definition}\label{definition;homotopy}
A \emph{Lie algebroid homotopy} is a Lie algebroid morphism
$\phi\colon T[0,1]\times A\to B$, where $A\to M$ and $B\to N$ are Lie
algebroids.
\end{definition}

Let $\phi$ a Lie algebroid homotopy.  The base map
$\mr\phi\colon[0,1]\times M\to N$ is then a homotopy of manifolds in
the usual sense.  We denote the Lie algebroid morphism
$i_t\circ\phi\colon A\to B$, which is defined for $0\le t\le1$, by
$\phi_t$, and we say that $\phi$ is a homotopy \emph{between $\phi_0$
  and $\phi_1$}.  Put
\begin{equation}\label{equation;de-rham-homotopy}
\kappa_\phi=\kappa\circ\phi^*,
\end{equation}
where $\kappa$ is the operator~\eqref{equation;integral};
then~\eqref{equation;cylinder} yields the \emph{homotopy formula}
\begin{equation}\label{equation;homotopy}
[\kappa_\phi,d_A]=\phi_1^*-\phi_0^*,
\end{equation}
which tells us that
$\kappa_\phi\colon\Omega_B^\bu(N)\to\Omega_A^{{\bu}-1}(M)$ is a
homotopy of complexes.
 
\begin{definition}\label{definition;equivalence}
A Lie algebroid morphism $\phi\colon A\to B$ is a \emph{homotopy
  equivalence} if it has a \emph{homotopy inverse}, i.e.\ a morphism
$\psi\colon B\to A$ such that $\psi\circ\phi$ is homotopic to $\id_A$
and $\phi\circ\psi$ is homotopic to~$\id_B$.
\end{definition}

The homotopy formula has the usual consequences, for instance the
following lemma.

\begin{lemma}\label{lemma;homotopy}
Let $A\to M$ and $B\to N$ be Lie algebroids.
\begin{enumerate}
\item\label{item;homotopic}
Homotopic morphisms $\phi_0$, $\phi_1\colon A\to B$ induce the same
morphism in cohomology $\phi_0^*=\phi_1^*\colon H_B^\bu(N)\to
H_A^\bu(M)$.
\item\label{item;equivalence}
A homotopy equivalence $\phi\colon A\to B$ induces an isomorphism in
cohomology
\[
\begin{tikzcd}[cramped,sep=small]\phi^*\colon
  H_B^\bu(N)\ar[r,"\cong"]&H_A^\bu(M)\end{tikzcd}.
\]
\end{enumerate}
\end{lemma}

\begin{remark}\label{remark;poincare}
Lemma~\ref{lemma;homotopy} gives the following version of the
Poincar\'e lemma: if\/ $A$ is homotopy equivalent to a ``point'',
i.e.\ a Lie subalgebra $\lie{a}$ of\/ $\stab(A,x)$ for some $x\in M$,
then $H_A^\bu(M)$ is isomorphic to $H^\bu(\lie{a})$, the Lie algebra
cohomology of\/ $\lie{a}$.  So if\/ $\lie{a}$ is acyclic, then
$H_A^0(M)=\R$ and $H_A^k(M)=0$ for $k\ge1$.
\end{remark}

\begin{definition}\label{definition;weak}
Let $B\to N$ be a Lie subalgebroid of\/ $A$.  A \emph{weak deformation
  retraction of\/ $A$ onto $B$} is a homotopy $\varrho\colon
T[0,1]\times A\to A$ such that $\varrho_1=\id_A$, $\varrho_0(A)=B$,
and $\varrho_0|_B=i_B$.
\end{definition}

The next statement follows immediately from
Lemma~\ref{lemma;homotopy}.

\begin{corollary}\label{corollary;retract}
Let $A\to M$ be a Lie algebroid and let $B\to N$ be a Lie subalgebroid
of\/ $A$.  Suppose there exists a weak deformation retraction
$\varrho\colon T[0,1]\times A\to A$ onto $B$.  The inclusion
$i_B\colon B\to A$ induces an isomorphism
$\begin{tikzcd}[cramped,sep=small]i_B^*\colon
  H_A^\bu(M)\ar[r,"\cong"]&H_B^\bu(N)\end{tikzcd}$
with inverse~$\varrho_0^*$.
\end{corollary}

\begin{definition}\label{definition;retraction}
Let $B\to N$ be a Lie subalgebroid of\/ $A$.  A \emph{deformation
  retraction of\/ $A$ onto $B$} is a weak deformation retraction
$\varrho\colon T[0,1]\times A\to A$ with the additional property that
$\varrho(t,u,b)=b$ for all $(t,u)\in T[0,1]=[0,1]\times\R$ and all
$b\in B$.
\end{definition}

A deformation retraction $\varrho$ satisfies $\varrho(\tangent
t,0_x)=0_x$ for all $x\in N$, where we think of\/ $\tangent t$ as the
section $(t,1)$ of\/ $T[0,1]$ and where $0_x$ denotes the origin of the
fibre $A_x$.  Its base homotopy $\mr\varrho\colon[0,1]\times M\to M$
restricts to the identity map $\mr\varrho_t|_N=\id_N$ for all~$t$,
i.e.\ $\mr\varrho$ is a deformation retraction in the usual sense.
These facts have the following significance for the
homotopy~$\kappa_\varrho$.

\begin{lemma}\label{lemma;retract}
Let $A\to M$ be a Lie algebroid and $B\to N$ a Lie subalgebroid.  Let
$\varrho\colon T[0,1]\times A\to A$ be a deformation retraction onto
$B$.  Then $(\kappa_\varrho\alpha)_x=0$ for all $A$-forms
$\alpha\in\Omega_A^k(M)$ and all $x\in N$.  If\/ $\alpha$ vanishes at a
point $x_0\in N$, then
$(\ca{L}_A(\sigma)\kappa_\varrho\alpha)_{x_0}=0$ for all sections
$\sigma$ of~$A$.
\end{lemma}

\begin{proof}
Let $\sigma_1$, $\sigma_2,\dots$,~$\sigma_{k-1}$ be sections of\/ $A$
and let $f$ be the smooth function
\[
f=\iota_A(\sigma_{k-1})\cdots\iota_A(\sigma_2)\iota_A(\sigma_1)
\kappa_\varrho\alpha
\]
on $M$.  From~\eqref{equation;integral}
and~\eqref{equation;de-rham-homotopy} we obtain
\begin{equation}\label{equation;coefficient}
f(x)=\int_0^1\alpha_{\mr\varrho(t,x)} \bigl(\varrho(\tangent
t,0_x),\varrho_t(\sigma_1(x)),
\varrho_t(\sigma_2(x)),\dots,\varrho_t(\sigma_{k-1}(x))\bigr)\,dt
\end{equation}
for $x\in M$.  Since $\varrho$ is a deformation retraction, we have
$\varrho(\tangent t,0_x)=0_x$ for $x\in N$, so it follows
from~\eqref{equation;coefficient} that the function $f$ vanishes on
$N$.  Thus $(\kappa_\varrho\alpha)_x=0$ for all $x\in N$.  If
$\alpha_{x_0}=0$ for some $x_0\in N$, then
$\alpha_{\mr\varrho(t,x_0)}=\alpha_{x_0}=0$ for all $t$, so we see
from~\eqref{equation;coefficient} that the function $f$ vanishes to
second order at~$x_0$.  Since $\ca{L}_A(\sigma)$ is a first-order
differential operator, this tells us that $(\ca{L}_A(\sigma)
f)(x_0)=0$ for all sections $\sigma$ of~$A$.  Using
$[\ca{L}_A(\sigma),\iota_A(\sigma_i)]=\iota_A([\sigma,\sigma_i])$ we
obtain $(\ca{L}_A(\sigma)\kappa_\varrho\alpha)_{x_0}=0$.
\end{proof}

\subsection{Isotopies and deformation retractions from sections}
\label{subsection;retract}

For the results of Section~\ref{subsection;homotopy} to be useful we
need methods for producing Lie algebroid homotopies.  We are aware of
two such methods: (1)~the flow of a time-dependent section produces an
isotopy, and (2)~if the section is Euler-like near a subalgebroid in
the sense of Bischoff et al.~\cite{bischoff-bursztyn-lima-meinrenken}
and Bursztyn et al.~\cite{bursztyn-lima-meinrenken;splitting}, we can
reparametrize the isotopy to produce a deformation retraction onto the
subalgebroid.  We explain both methods in turn.

\begin{definition}\label{definition;isotopy}
Let $A\to M$ be a Lie algebroid.  A \emph{(global) Lie algebroid
  isotopy} is a Lie algebroid homotopy $\phi\colon T[0,1]\times A\to
A$ such that $\phi_0=\id_A$ and each $\phi_t$ is an isomorphism.
\end{definition}

Just as time-dependent vector fields give rise to isotopies of
manifolds, so time-dependent sections give rise to isotopies of Lie
algebroids.  

\begin{proposition}\label{proposition;isotopy}
Let $A\to M$ be a Lie algebroid.  Let $\sigma_t$ be a time-dependent
section of\/ $A$ defined for $0\le t\le1$, let\/ $\bad_A(\sigma_t)$ be
the time-dependent vector field on $A$ defined in~\eqref{equation;ad},
and let\/ $\Phi_t$ be its flow with initial condition\/
$\Phi_0=\id_A$.  Suppose that\/ $\Phi_t$ is defined globally on $A$
for all $0\le t\le1$.  Then\/ $\Phi\colon[0,1]\times A\to A$ extends
to a Lie algebroid isotopy $\phi\colon T[0,1]\times A\to A$ given by
\[\phi(t,u,a)=u\sigma_t\bigl(\mr\Phi_t(\pi_A(a))\bigr)+\Phi_t(a)\]
for $(t,u)\in T[0,1]=[0,1]\times\R$ and $a\in A$.
\end{proposition}

\begin{proof}
The equalities
\[\phi_t(a)=\phi\circ i_t(a)=\phi(t,0,a)=\Phi_t(a),\]
where $i_t$ is as in~\eqref{equation;cylinder-slice}, show that
$\phi_0=\id_A$, that $\phi_t=\Phi_t$ is a Lie algebroid isomorphism
for each~$t$, and that the map $\phi$ extends~$\Phi$. It remains to
show that $\phi$ is a Lie algebroid morphism.  Let $\hat{A}$ be the
product Lie algebroid $T[0,1]\times A$ and define
$\hat\phi\colon\hat{A}\to\hat{A}$ by
\[\hat\phi(t,u,a)=(t,u,\phi(t,u,a)).\]
Then $\phi=\pr_A\circ\hat\phi$, where $\pr_A\colon\hat{A}\to A$ is the
projection, so it is enough to show that $\hat\phi$ is a Lie algebroid
morphism.  For $0\le s\le1$ define $\hat\phi_s\colon\hat{A}\to\hat{A}$
by
\[
\hat\phi_s(t,u,a)=
\Bigl(t,u,su\sigma_{st}\bigl(\mr\Phi_{st}(\pi_A(a))\bigr)+
\Phi_{st}(a)\Bigr).
\]
Then $\hat\phi_0=\id_{\hat{A}}$, $\hat\phi_1=\hat\phi$, and
\begin{multline}\label{equation;ode}
\frac{d\hat\phi_s}{ds}(t,u,a)=
u\sigma_{st}\bigl(\mr\Phi_{st}(\pi_A(a))\bigr)+
stu\dot{\sigma}_{st}\bigl(\mr\Phi_{st}(\pi_A(a))\bigr)+
\\
stuT\sigma_{st}\bigl(\an_A(\sigma_{st}(\pi_A(a)))\bigr)+
t\bad_A(\sigma_{st})(\Phi_{st}(a)),
\end{multline}
where we used the differential equations
\[
\frac{d\Phi_t}{dt}(a)=\bad_A(\sigma_t)(\Phi_t(a))
\qquad\text{and}\qquad
\frac{d\mr\Phi_t}{dt}(x)=\an_A(\sigma_t)\bigl(\mr\Phi_t(x)\bigr).
\]
Define an $s$-dependent section $\hat{\sigma}_s\in\Gamma(\hat{A})$ by
$\hat{\sigma}_s(t,x)=t\sigma_{st}(x)$.  A computation
using~\eqref{equation;ad} shows that
\[
\bad_{\hat{A}}(\hat{\sigma}_s)(t,u,a)=
t\bad_A(\sigma_{st})(a)+u\sigma_{st}(\pi_A(a))+
stu\dot{\sigma}_{st}(\pi_A(a)).
\]
Comparing with~\eqref{equation;ode} we see that
\[
\frac{d\hat\phi_s}{ds}(t,u,a)=
\bad_{\hat{A}}(\hat\sigma_s)(\hat\phi_s(t,u,a)).
\]
Thus $\hat\phi_s$ is the flow of the vector field
$\bad_{\hat{A}}(\hat\sigma_s)$ on $\hat{A}$.  It now follows from
Lemma~\ref{lemma;pull-auto}\eqref{item;automorphism} that $\hat\phi_s$
is a Lie algebroid automorphism for each $s$.  In particular
$\hat\phi=\hat\phi_1$ is a Lie algebroid automorphism.
\end{proof}

A vector field $v$ on a manifold $M$ which is tangent to a submanifold
$N$ can be viewed as a smooth map of pairs $v\colon(M,N)\to(TM,TN)$.
Its tangent map is then a map of pairs $Tv\colon(TM,TN)\to(TTM,TTN)$,
and hence descends to a map of normal bundles
\[\ca{N}(v)\colon\ca{N}(M,N)\longto\ca{N}(TM,TN)\cong T\ca{N}(M,N),\]
i.e.\ a vector field on the normal bundle $\ca{N}(M,N)$ that is linear
along the fibres.  We call $\ca{N}(v)$ the \emph{linear approximation}
to $v$ along~$N$.  A vector field $v$ on $M$ is called
\emph{Euler-like} along $N$ if\/ $v$ is complete, vanishes on $N$, and
$\ca{N}(v)$ is equal to the Euler vector field of the vector bundle
$\ca{N}(M,N)$.  The significance of this notion lies in the following
``quantitative'' version of the tubular neighbourhood
theorem~\cite[Proposition~2.7]{bursztyn-lima-meinrenken;splitting}:
for every Euler-like vector field $v$ there exists a unique tubular
neighbourhood embedding $i_v\colon\ca{N}(M,N)\to M$ such that $i_v^*v$
is equal to the Euler vector field of\/ $\ca{N}(M,N)$.  We call $i_v$
the \emph{tubular neighbourhood embedding associated with $v$}.
\glossary{NMNW@$\ca{N}(M,N)$, normal bundle of\/ $N$ in $M$}
\glossary{Nv@$\ca{N}(v)$, linearization of vector field along $N$}
\glossary{iv@$i_v\colon\ca{N}(M,N)\to M$, tubular neighbourhood
  embedding associated with Euler-like vector field}

\begin{definition}\label{definition;euler-like}
Let $A\to M$ be a Lie algebroid and $B\to N$ a Lie subalgebroid.  An
\emph{Euler-like} section of\/ $A$ along $B$ is a section
$\eps\in\Gamma(A)$ such that $\eps|_N=0$ and the vector field
$\bad_A(\eps)$ on $A$ defined in~\eqref{equation;ad} is Euler-like
along the submanifold $B$ of~$A$.
\end{definition}

\begin{lemma}\label{lemma;euler-like}
Let $A\to M$ be a Lie algebroid, $B\to N$ a Lie subalgebroid, and
$\eps\in\Gamma(A)$ a section of~$A$.
\begin{enumerate}
\item\label{item;euler-like}
If the vector field $\bad_A(\eps)$ on $A$ is Euler-like along $B$,
then the vector field $\an_A(\eps)$ on $M$ is Euler-like along~$N$.
\item\label{item;euler-like-clean}
Suppose $N$ cleanly intersects $A$ and let $B=i_N^!A$.  If\/ $\eps$
vanishes along $N$ and the vector field $\an_A(\eps)$ on $M$ is
Euler-like along $N$, then the section $\eps$ is Euler-like along~$B$.
\end{enumerate}
\end{lemma}

\begin{proof}
We will use the following simple facts: (1)~a linear vector field on a
vector bundle $E\to M$ is complete if and only if the associated base
vector field on $M$ is complete; (2)~if\/ $F\to N$ is a second vector
bundle and $\phi\colon E\to F$ is a vector bundle map, then the Euler
vector fields of\/ $E$ and $F$ are $\phi$-related; (3)~if
$v\in\Gamma(TE)$ and $w\in\Gamma(TF)$ are linear vector fields and
$v\sim_\phi w$, then $v$ is uniquely determined by $w$ if\/ $F$ is
fibrewise injective, and $w$ is uniquely determined by $v$ if\/ $F$ is
fibrewise surjective.

\eqref{item;euler-like}~The vector field $\bad_A(\eps)$ is a linear
lift of the vector field $\an_A(\eps)$, so $\an_A(\eps)$ is complete
because of fact~(1) above.  Also $\an_A(\eps)|_N=0$ because
$\bad_A(\eps)|_B=0$.  From $\bad_A(\eps)\sim_\pi\an_A(\eps)$ we get
\[\ca{N}(\bad_A(\eps))\sim_{\ca{N}(\pi)}\ca{N}(\an_A(\eps)).\]
Since $\ca{N}(\pi)\colon\ca{N}(A,B)\to\ca{N}(M,N)$ is a surjective
vector bundle map and $\ca{N}(\bad_A(\eps))$ is the Euler vector field
of\/ $\ca{N}(A,B)$, it follows from facts~(2) and~(3) above that
$\ca{N}(\an_A(\eps))$ is the Euler vector field of\/ $\ca{N}(M,N)$.

\eqref{item;euler-like-clean}~We must show that the vector field
$\bad_A(\eps)$ on $A$ is Euler-like along $B=i_N^!A=\an_A^{-1}(N)$.
This is verified in~\cite[\S\,3.6]{bursztyn-lima-meinrenken;splitting}
under the assumption that $N$ is transverse to~$A$.  The proof in the
clean case is almost the same and goes as follows.  The module of
relative sections $\Gamma(A;B)$ defined in~\eqref{equation;relative}
is a Lie subalgebra of\/ $\Gamma(A)$ and the submodule $\Gamma(A;0_N)$
is an ideal of\/ $\Gamma(A;B)$.  By hypothesis $\eps|_N=0$, so
$[\eps,\upsilon]|_N=0$ for all $\upsilon\in\Gamma(A;B)$.  It now
follows from~\eqref{equation;ad} that the vector field $\bad_A(\eps)$
vanishes on $B$.  It remains to show that its linear approximation
$\ca{N}(\bad_A(\eps))$ is the Euler vector field of the normal bundle
$\ca{N}(A,B)$.  Let $T\an(\eps)\in\Gamma(TTM)$ be the tangent lift of
$\an(\eps)\in\Gamma(TM)$.  The vector field $T\an(\eps)$ is Euler-like
along $TN$, so its linear approximation $\ca{N}(T\an(\eps))$ is the
Euler vector field of the bundle $\ca{N}(TM,TN)$.  The vector fields
$\bad_A(\eps)$ and $T\an(\eps)$ are related via the anchor:
$\bad_A(\eps)\sim_{\an}T\an(\eps)$.  Applying the normal bundle
functor we obtain
\begin{equation}\label{equation;euler}
\ca{N}(\bad_A(\eps))\sim_{\ca{N}(\an)}\ca{N}(T\an(\eps)).
\end{equation}
The map $\ca{N}(\an)$ fits into a commutative diagram of vector bundle
maps
\[
\begin{tikzcd}
\pi_B^*(i_N^*A/B)\ar[r,hook]\ar[d]&\ca{N}(A,B)\ar[r,two
  heads]\ar[d,"\ca{N}(\an)"]&\pi_B^*\ca{N}(M,N)\ar[d]
\\
\pi_{TN}^*\ca{N}(M,N)\ar[r,hook]&\ca{N}(TM,TN)\ar[r,two
  heads]&\pi_{TN}^*\ca{N}(M,N),
\end{tikzcd}
\]
the rows of which are exact.  The vertical map on the right is a
fibrewise isomorphism.  By
Lemma~\ref{lemma;constant-clean}\eqref{item;submanifold} the vertical
map on the left is fibrewise injective because $N$ cleanly
intersects~$A$.  Hence $\ca{N}(\an)$ is fibrewise injective.  It now
follows from~\eqref{equation;euler} and from facts~(2) and~(3) that
$\ca{N}(\bad_A(\eps))$ is the Euler vector field of\/ $\ca{N}(A,B)$.
\end{proof}

The utility of Euler-like sections is that their flows give rise to
deformation retractions.

\begin{theorem}[retraction theorem]\label{theorem;retraction}
Let $A\to M$ be a Lie algebroid and $B\to N$ a Lie subalgebroid.
Suppose there exists an Euler-like section $\eps$ of\/ $A$ along~$B$.
Let $U$ be the tubular neighbourhood of\/ $N$ associated with the
Euler-like vector field $\an_A(\eps)$.  Then there exist a deformation
retraction of\/ $A|_U$ onto~$B$.  It follows that $i_B^*\colon
H_A^\bu(U)\to H_B^\bu(N)$ is an isomorphism.
\end{theorem}

\begin{proof}
The proof is an adaptation
of~\cite[\S\,3.6]{bursztyn-lima-meinrenken;splitting}.  Let $\Phi_t$
be the flow of the vector field $\bad_A(\eps)$ and let $\phi\colon
T\R\times A\to A$ be the corresponding Lie algebroid isotopy as in
Proposition~\ref{proposition;isotopy}.  Since $\bad_A(\eps)$ is
Euler-like, near $B$ this isotopy is exponentially expanding in the
directions normal to~$B$.  We turn it into a retraction by rescaling
the time variable.  Put $f(s)=\log s$ for $0<s\le1$; then the Lie
algebroid homotopy
\[
\begin{tikzcd}
\varrho\colon T(0,1]\times A\ar[r,"Tf\times\id_A"]& T\R\times
  A\ar[r,"\phi"]&A
\end{tikzcd}
\]
is given by the formula
\begin{equation}\label{equation;retract}
\varrho(s,v,a)=\frac{v}{s}\eps\bigl(\mr\Phi_{\log s}(\pi(a))\bigr)+
\Phi_{\log s}(a)
\end{equation}
for $(s,v)\in T(0,1]=(0,1]\times\R$ and $a\in A$.  The base homotopy
    $\mr\varrho\colon(0,1]\times M\to M$ is given by
      $\mr\varrho(s,x)=\mr\Phi_{\log s}(x)$.
By Lemma~\ref{lemma;euler-like}\eqref{item;euler-like} the vector
field $\an(\eps)$ is Euler-like along $N$ and therefore determines a
tubular neighbourhood embedding $i\colon\ca{N}(M,N)\to U$.  Let us
restrict $\mr\varrho$ to $U$ and $\varrho$ to $A|_U$.  If we identify
$U$ with the normal bundle via the embedding $i$, the base homotopy
$\mr\varrho$ is given by fibrewise multiplication by $s$, so it
extends smoothly to a homotopy $\mr\varrho\colon[0,1]\times U\to U$,
which retracts $U$ onto~$N$.  Likewise, if we identify $A|_U$ with the
normal bundle $\ca{N}(A,B)$ via the tubular neighbourhood embedding
$\ca{N}(A,B)\to A|_U$ determined by $\bad_A(\eps)$, the flow
$\Phi_{\log s}$ is given by fibrewise multiplication by $s$, which
retracts $A|_U$ onto $B$.  By hypothesis $\eps$ vanishes on $N$, so
the singularity at $s=0$ in the term
$\frac{v}{s}\eps\bigl(\mr\varrho(s,\pi(a))\bigr)$ is removable.  It
follows that the homotopy~\eqref{equation;retract} extends smoothly to
a homotopy $T[0,1]\times A|_U\to A|_U$.  We have
$\varrho(1,0,a)=\Phi_0(a)=a$ for all $a\in A$, so $\varrho_1=\id_A$.
If\/ $b\in B$, then $\varrho(s,v,b)=b\in B$ for all $(s,v)\in T[0,1]$.
Also $\varrho(0,0,a)\in B$ for all $a\in A$, so $\varrho_0(A)=B$.  We
have shown that $\varrho$ is a deformation retraction of\/ $A|_U$ onto
$B$.  The second assertion of the theorem is now an immediate
consequence of Corollary~\ref{corollary;retract}.
\end{proof}

Given a Lie algebroid $A\to M$ and a Lie subalgebroid $B\to N$, the
normal bundle $\ca{N}(A,B)$ has the structure of a Lie algebroid over
the normal bundle $\ca{N}(M,N)$.  The space $\ca{N}(A,B)$ is a double
vector bundle
\[
\begin{tikzcd}[column sep=small]
\ca{N}(A,B)\ar[r]\ar[d]&\ca{N}(M,N)\ar[d]\\
B\ar[r]&N
\end{tikzcd}
\]
and the vector bundle operations on the vertical bundle
$\ca{N}(A,B)\to B$ are Lie algebroid morphisms of\/ $\ca{N}(A,B)$;
see~\cite[Theorem~A.7]{meinrenken;groupoids-algebroids}.

\begin{definition}\label{definition;linear}
The Lie algebroid $\ca{N}(A,B)\to\ca{N}(M,N)$ is the
\emph{linearization of\/ $A$ at $B$}.  The Lie algebroid $A$ is
\emph{linearizable at the Lie subalgebroid $B$} if there is a
\emph{linearization map} $\ca{N}(A,B)\to A$, i.e.\ a map that is
simultaneously a Lie algebroid morphism and a tubular neighbourhood
embedding.
\end{definition}

Under a slightly weaker hypothesis than
Theorem~\ref{theorem;retraction} (we don't need $\eps|_N=0$) we obtain
a linearization theorem.

\begin{theorem}[linearization theorem]\label{theorem;linearization}
Let $A\to M$ be a Lie algebroid and $B\to N$ a Lie subalgebroid.
Suppose there exists a section $\eps$ of\/ $A$ such that the vector
field $\bad_A(\eps)$ on $A$ is Euler-like along~$B$.  The tubular
neighbourhood embedding $i_v\colon\ca{N}(A,B)\to A$ associated with
$v=\bad_A(\eps)$ on $A$ is a linearization map, whose base map is the
tubular neighbourhood embedding $\iota_{\mr{v}}\colon\ca{N}(M,N)\to M$
associated with the Euler-like vector field $\mr{v}=\an_A(\eps)$ on
$M$.
\end{theorem}

\begin{proof}
Let $\ca{D}(M,N)$ be the deformation space of the pair $(M,N)$.  As a
set this is the disjoint union
\[\ca{D}(M,N)=\ca{N}(M,N)\sqcup(\R^\times\times M).\]
See~\cite[\S\S\,2--3]{bischoff-bursztyn-lima-meinrenken}
and~\cite[Appendix~A]{meinrenken;groupoids-algebroids} for the
manifold structure of\/ $\ca{D}(M,N)$ and for the fact that the
deformation space $\ca{D}(A,B)$ is a Lie algebroid over $\ca{D}(M,N)$.
The map $\pi\colon\ca{D}(A,B)\to\R$ which sends $\ca{N}(A,B)$ to $0$
and agrees with the projection $\R^\times\times M\to\R^\times$ on the
complement of\/ $\ca{N}(A,B)$ is a smooth submersion.  The fibres
$\pi^{-1}(t)$ are Lie algebroids isomorphic to $A\to M$ for $t\ne0$
and to $\ca{N}(A,B)\to\ca{N}(M,N)$ for $t=0$.  In this sense
$\ca{D}(A,B)$ is a family of Lie algebroids that deforms $A$ to its
normal bundle in $B$.  Since the vector field $\bad_A(\eps)$ is
Euler-like, the vector field $\tangent t+t^{-1}\bad(\eps)$, which is
defined on the complement of\/ $\ca{N}(A,B)$ in $\ca{D}(A,B)$, extends
smoothly to a complete vector field $w$ on $\ca{D}(A,B)$.  The vector
field $w$ is an infinitesimal Lie algebroid automorphism and is
$\pi$-related to the vector field $\tangent t$ on $\R$, so its flow at
time~$1$ is a Lie algebroid morphism from $\pi^{-1}(0)=\ca{N}(A,B)$ to
$\pi^{-1}(1)=A$, which agrees with the tubular neighbourhood embedding
$i_v$.  These statements are proved
in~\cite[\S\,A.6]{meinrenken;groupoids-algebroids} when $B$ is the
pullback to a transversal $N$, but the same proofs work for an
arbitary Lie subalgebroid $B$, as long as we have a section $\eps$
with $\bad_A(\eps)$ Euler-like along $B$.
\end{proof}

\begin{remarks}\label{remark;transverse}
Let $A\to M$ be a Lie algebroid and $B\to N$ a Lie subalgebroid.
\begin{numerate}
\item\label{item;basis}
If\/ $\eps$ is an Euler-like section of\/ $A$ along $N$, then so is
$f\eps$ for any bounded smooth function $f$ which is equal to $1$ on a
neighbourhood of~$N$.  Hence, by Theorem~\ref{theorem;retraction},
there exist deformation retractions of\/ $A|_U$ onto $B$ for all $U$
ranging over a basis of tubular neighbourhoods of~$N$.
\item\label{item;split}
Let $N$ be an embedded submanifold that is transverse to $A$ and let
$B=i_N^!A$ be the pullback Lie algebroid.  Then there automatically
exists an Euler-like section of\/ $A$ along $B$, as proved
in~\cite[Lemma~3.9]{bursztyn-lima-meinrenken;splitting}, and therefore
the retraction and linearization theorems apply.  In fact, in this
case the normal bundle is $\ca{N}(A,B)=\pi_{M,N}^!i_N^!A$, where
$\pi_{M,N}\colon\ca{N}(M,N)\to N$ is the projection, so the
linearization theorem amounts to
\[\pi_{M,N}^!i_N^!A\cong A|_U.\]
This is the Dufour-Fernandes-Weinstein splitting theorem.  Our
Theorems~\ref{theorem;retraction} and~\ref{theorem;linearization} were
prompted by the treatment of the splitting theorem given by Bischoff
et al.~\cite{bischoff-bursztyn-lima-meinrenken} and Bursztyn et
al.~\cite[Theorem~4.1]{bursztyn-lima-meinrenken;splitting}.  See also
Meinrenken's notes~\cite[Theorem~8.13,
  Appendix~A]{meinrenken;groupoids-algebroids}.  The isomorphism
$H_A^\bu(U)\cong H_B^\bu(N)$ of Theorem~\ref{theorem;retraction} was
obtained by
Frejlich~\cite[Theorem~B]{frejlich;submersions-lie-algebroids} in the
transverse case.  If the intersection is clean but not transverse,
typically $\pi_{M,N}^!i_N^!A$ has higher rank than $A$, so the
splitting theorem fails, but sometimes an Euler-like section still
exists.
\item\label{item;foliation}
Let $A=T\F$ be the tangent bundle of a (regular) foliation $\F$ of
$M$.  Let $B=i_N^!A$ be the pullback of\/ $A$ to an embedded submanifold
$N$ of\/ $M$ that cleanly intersects $A$.  By
Proposition~\ref{proposition;pull} the dimension of the intersection
$T_xN\cap T_x\F$ is independent of\/ $x\in N$, so $\F$ induces a
foliation $\F_N$ of\/ $N$ and $B$ is its tangent bundle $T\F_N$.  An
Euler-like section of\/ $A$ along $B$ is just an Euler-like vector field
along $N$ which is tangent to the leaves of\/ $\F$.  Such a vector field
exists if and only if\/ $N$ is transverse to the leaves.  In
non-transverse cases (e.g.\ when $N$ is a single point)
Theorems~\ref{theorem;retraction} and~\ref{theorem;linearization} do
not apply.  But see Remark~\ref{remark;foliation}.
\end{numerate}
\end{remarks}

In situations where there exists no Euler-like section the following
modification of the retraction and linearization theorems may be of
use.

\begin{corollary}\label{corollary;retraction}
Let $\bar{A}\to M$ be a Lie algebroid.  Let $A\to M$ be a Lie
subalgebroid of\/ $\bar{A}$ over the same base and let $B\to N$ be a Lie
subalgebroid of\/ $A$.  Suppose there exists a section $\eps$ of
$\bar{A}$ along~$B$ which normalizes $\Gamma(A)$.  If the vector field
$\bad_{\bar{A}}(\eps)$ on $\bar{A}$ is Euler-like along $B$, the
linearization map $\ca{N}(\bar{A},B)\to\bar{A}$ given by the
Euler-like vector field $\bad_{\bar{A}}(\eps)$ on $\bar{A}$ restricts
to a linearization map $\ca{N}(A,B)\to A$.  If additionally
$\eps|_N=0$, there exist a deformation retraction of\/ $A|_U$ onto $B$,
where $U$ is the tubular neighbourhood of\/ $N$ associated with the
Euler-like vector field $\an_{\bar{A}}(\eps)$ on $M$, and hence
$i_B^*\colon H_A^\bu(U)\to H_B^\bu(N)$ is an isomorphism.
\end{corollary}

\begin{proof}
The proof of Theorem~\ref{theorem;linearization} tells us that the
time~$1$ flow of the vector field $w=\tangent
t+t^{-1}\bad{\bar{A}}(\eps)$ on the deformation space
$\ca{D}(\bar{A},B)$ induces a linearization
$\ca{N}(\bar{A},B)\to\bar{A}$.  It follows from
Lemma~\ref{lemma;pull-auto}\eqref{item;preserve} that the flow of\/ $w$
preserves $\ca{D}(A,B)$ and therefore gives us a linearization
$\ca{N}(A,B)\to A$.  Similarly, if\/ $\eps|_N=0$ then
Theorem~\ref{theorem;retraction} tells us that the flow of
$\bad_{\bar{A}}(\eps)$ induces a deformation retraction of
$\bar{A}|_U$ onto $B$.  Again by
Lemma~\ref{lemma;pull-auto}\eqref{item;preserve} this deformation
retraction preserves $A$ and therefore gives us a deformation
retraction of\/ $A|_U$ onto $B$.
\end{proof}

\begin{remark}\label{remark;foliation}
Let $\F$ be a foliation of\/ $M$ and let $A=T\F$ be the tangent bundle
of\/ $\F$.  Let $B=i_N^!A$ be the pullback of\/ $A$ to an embedded
submanifold $N$ of\/ $M$ that cleanly intersects $A$.  We take $\bar{A}$
to be the tangent bundle of\/ $M$.  A section of\/ $\bar{A}$, i.e.\ vector
field on $M$, normalizes $\Gamma(A)$ if and only if its flow maps
leaves to leaves.  For instance, if\/ $N=\{x\}$ is a single point, then
in a foliation chart $U$ centred at $x$ in which $\F$ is spanned by
the vector fields $\tangent x_1$, $\tangent x_2,\dots$,~$\tangent x_p$
the Euler vector field $\eps=\sum_ix_i\tangent x_i$ normalizes
$\Gamma(A)$.  Hence the conclusions of
Corollary~\ref{corollary;retraction} apply to the pair $(A,B=0_x)$.
An arbitrary clean submanifold $N$ can be covered with foliation
charts $U$ in which $\F$ is spanned by the vector fields $\tangent
x_1$, $\tangent x_2,\dots$,~$\tangent x_p$ and in which $N$ is given
by the equations $x_{i_1}=x_{i_2}=\cdots=x_{i_k}=0$.  In such a chart
the vector field $\eps_U=\sum_{l=1}^kx_{i_l}\tangent x_{i_l}$ is
Euler-like along $N$ and normalizes $\Gamma(A)$.  It would be
interesting to have conditions under which the $\eps_U$ can be glued
to a global vector field $\eps$ that is Euler-like along $N$ and
normalizes $\Gamma(A)$.
\end{remark}





\subsection{A Darboux-Moser-Weinstein theorem for Lie algebroids}
\label{subsection;moser}

Here is a version of Moser's trick for Lie algebroids.  See
Definition~\ref{definition;isotopy} for Lie algebroid isotopies.

\begin{proposition}\label{proposition;moser}
Let $A\to M$ be a Lie algebroid.  Let\/ $\omega_t\in\Omega_A^2(M)$ be
a smooth path of $A$-symplectic forms defined for\/ $0\le t\le1$.  Let
$\alpha_t\in\Omega_A^1(M)$ be a smooth path of\/ $1$-forms such that
$\dot\omega_t=-d_A\alpha_t$.
\begin{enumerate}
\item\label{item;compact}
If $M$ is compact, there exists an isotopy $\phi$ of $A$ such that
$\phi_t^*\omega_t=\omega_0$.
\item\label{item;moser-subalgebroid}
If $B\to N$ is a Lie subalgebroid of $A$ and\/ $(\alpha_t)_x=0$ for
all $x\in N$ and for all $t$, there exist a neighbourhood $U$ of $N$
and an isotopy $\phi$ of $A|_U$ such that $\phi_t^*\omega_t=\omega_0$
and $\phi_t|_B=\id_B$.
\end{enumerate}
\end{proposition}

\begin{proof}
We omit the proof of~\eqref{item;compact}, which is similar to that
of~\eqref{item;moser-subalgebroid}.  The time-dependent section
$\sigma_t=\omega_t^\sharp\alpha_t$ of $A$ satisfies
\begin{equation}\label{equation;flow}
d_A\iota_A(\sigma_t)\omega_t=d_A\alpha_t=-\dot{\omega}_t.
\end{equation}
By hypothesis the form $\alpha_t$ vanishes at every point of~$N$, and
therefore so does the section $\sigma_t$.  Hence the vector field
$\an_A(\sigma_t)$ on $M$ integrates to a flow $\mr\Phi_t$ defined for
$0\le t\le1$ on a suitable neighbourhood of $N$ which leaves $N$
fixed.  The vector field $\bad_A(\sigma_t)$ on $A$ is a linear lift of
$\an_A(\sigma_t)$ and integrates to a flow $\Phi_t$ with initial
condition $\phi_0=\id$ defined for $0\le t\le1$ in a neighbourhood of
$B$.  Since $\bad_A(\sigma_t)$ vanishes along $B$, the flow leaves $B$
fixed.  According to Proposition~\ref{proposition;isotopy} the flow
$\Phi_t$ determines a Lie algebroid isotopy $\phi$.
Using~\eqref{equation;pull}, \eqref{equation;cartan},
and~\eqref{equation;flow} we obtain $\phi_t^*\omega_t=\omega_0$ for
$0\le t\le1$.
\end{proof}

Next we give a Lie algebroid version of the Darboux-Moser-Weinstein
theorem~\cite[\S\,5]{weinstein;lectures-symplectic}, which says that
two symplectic forms that agree at all points of a submanifold are
isomorphic in a neighbourhood of the submanifold.  See
Definition~\ref{definition;retraction} for the notion of a Lie
algebroid retraction.

\begin{theorem}\label{theorem;dmw}
Let $A\to M$ be a Lie algebroid and let $B\to N$ be a Lie subalgebroid
such that $A|_U$ admits a deformation retraction onto~$B$ for some
neighbourhood $U$ of $N$.  Let\/ $\omega_0$ and\/ $\omega_1$ be
$A$-symplectic forms on $M$ satisfying $\omega_{0,x}=\omega_{1,x}$ for
all $x\in N$\@.  There exist open neighbourhoods $U_0$ and $U_1$ of
$N$ in $M$ and a Lie algebroid isomorphism $\phi\colon A|_{U_0}\to
A_{U_1}$ such that $\phi^*\omega_1=\omega_0$, $\mr\phi|_N=\id_N$, and
$\phi|_{A_x}=\id_{A_x}$ for all $x\in N$.
\end{theorem}

\begin{proof}
Let $\omega_t=(1-t)\omega_0+t\omega_1$.  Choose an open neighbourhood
$U$ of $N$ such that the form $\omega_t$ is symplectic on $A|_U$ for
all $t$ and such that $A|_U$ admits a deformation retraction
$\varrho=\varrho_U\colon A|_U\to A|_U$ onto~$B$.  Put
$\alpha=\kappa_\varrho(\omega_1-\omega_0)\in\Omega_A^1(M)$, where
$\kappa_\varrho$ is as in~\eqref{equation;de-rham-homotopy}.  The
homotopy formula~\eqref{equation;homotopy} yields
\[
d_A\alpha=d_A\kappa_\varrho(\omega_1-\omega_0)=
\varrho_1^*(\omega_1-\omega_0)-\varrho_0^*(\omega_1-\omega_0)=
\omega_0-\omega_1=-\dot{\omega}_t.
\]
By Lemma~\ref{lemma;retract} the form $\alpha$ vanishes at every point
of~$N$, and therefore
Proposition~\ref{proposition;moser}\eqref{item;moser-subalgebroid}
applies, giving us an isotopy $\phi_t$ of a possibly smaller
neighbourhood $U$ fixing $B$ and satisfying
$\phi_t^*\omega_t=\omega_0$.  Taking $\phi=\phi_1$ gives
$\phi^*\omega_1=\omega_0$.  Let $x\in N$.  To prove that $\phi\colon
A_x\to A_x$ is the identity map we show that
\begin{equation}\label{equation;fix}
\phi_t(\tau(x))=\tau(x)
\end{equation}
for all $t\in[0,1]$ and all sections $\tau$ of $A$.  The formula
$\iota_A([\tau,\sigma_t])=[\ca{L}_A(\tau),\iota_A(\sigma_t)]$ yields
\[
\iota_A([\tau,\sigma_t])\omega_t=
\ca{L}_A(\tau)\iota_A(\sigma_t)\omega_t-
\iota_A(\sigma_t)\ca{L}_A(\tau)\omega_t=
\ca{L}_A(\tau)\alpha-\iota_A(\sigma_t)\ca{L}_A(\tau)\omega_t,
\]
so
\[
\bigl(\iota_A([\tau,\sigma_t])\omega_t\bigr)_x=
(\ca{L}_A(\tau)\alpha)_x-(\iota_A(\sigma_t)\ca{L}_A(\tau)\omega_t)_x
=0,
\]
where we used Lemma~\ref{lemma;retract} and the fact that $\sigma_t$
vanishes at~$x$.  Since $\omega_t$ is symplectic, it follows that the
section $[\tau,\sigma_t]$ of $A$ vanishes at~$x$.  In view
of~\eqref{equation;lie-section} this implies
\[
\frac{d}{dt}\phi_t(\tau(x))=0
\]
for all $t$.  Since $\phi_0=\id_A$, this proves~\eqref{equation;fix}.
\end{proof}

\begin{remarks}
\begin{numerate}
%
\item\label{item;retract}
If $A$ admits an Euler-like section along $B$, then a deformation
retraction from $A|_U$ onto $B$ exists for some neighbourhood $U$ of
$N$ by Theorem~\ref{theorem;retraction} and
Remark~\ref{remark;transverse}\eqref{item;basis}.  Most importantly,
this is the case if $N$ is transverse to $A$ and $B=i_N^!A$
(Remark~\ref{remark;transverse}\eqref{item;split}).  See
Corollary~\ref{corollary;retraction} for another condition under which
a local deformation retraction is guaranteed to exist.
\item\label{item;klaasse-lanius}
A local deformation retraction is too much to ask for in many
situations; all one needs to prove Theorem~\ref{theorem;dmw} is a
suitable primitive $\alpha\in\Omega_A^1(U)$ of $\omega_1-\omega_0$.
See Scott~\cite[\S\,5]{scott;geometry-bk}, Klaasse and
Lanius~\cite[\S\,4.3]{klaasse;geometric-structures-lie-algebroids},
\cite{klaasse-lanius;splitting-poisson}, and Miranda and
Scott~\cite[\S\,2]{miranda-scott;e-manifolds} for instances of such
situations.
\item\label{item;poisson}
The map $\phi$ of Theorem~\ref{theorem;dmw} is a Poisson isomorphism
relative to the Poisson structures on $M$ induced by $\omega_0$ and
$\omega_1$.
\end{numerate}
\end{remarks}

\subsection{The coisotropic embedding theorem}\label{subsection;gotay}

Theorem~\ref{theorem;dmw} yields Lie algebroid analogues of the
familiar local normal forms of symplectic geometry.  In this section
we offer the Lie algebroid version of Gotay's coisotropic embedding
theorem~\cite{gotay;coisotropic-presymplectic}, which is a complement
to the symplectization theorem, Theorem~\ref{theorem;symplectization}.
It states that a symplectic Lie algebroid $A\to M$ is in a
neighbourhood of a transverse coisotropic submanifold $i_N\colon N\to
M$ completely determined by the pullback Lie algebroid $B=i_N^!A$ and
by the $B$-presymplectic form on $N$.  Our proof uses the proof of the
Lie algebroid splitting theorem given by Bursztyn et
al.~\cite{bursztyn-lima-meinrenken;splitting}.  A version of
Theorem~\ref{theorem;coisotropic} for $b$-symplectic manifolds was
established by Geudens and
Zambon~\cite{geudens-zambon;coisotropic-b-symplectic}.

\begin{theorem}[coisotropic embeddings]\label{theorem;coisotropic}
Let\/ $(B\to N,\omega_B)$ be a presymplectic Lie algebroid.  Let\/
$(A_0\to M_0,\omega_0)$ and $(A_1\to M_1,\omega_1)$ be symplectic Lie
algebroids and let $i_0\colon N\to M_0$ and $i_1\colon N\to M_1$ be
transverse coisotropic embeddings such that $i_0^!A_0$ and $i_1^!A_1$
are isomorphic to $B$ and
$\omega_B=(i_0)_!^*\omega_0=(i_1)_!^*\omega_1$.  There exist open
neighbourhoods $U_0$ of $N$ in $M_0$ and\/ $U_1$ of $N$ in\/ $M_1$ and
an isomorphism of Lie algebroids\/ $\phi\colon A_0|_{U_0}\to
A_1|_{U_1}$ such that\/ $\phi\circ(i_0)_!=(i_1)_!$ and\/
$\phi^*\omega_1=\omega_0$.
\end{theorem}

\begin{proof}
We will prove the theorem in two special cases, which taken together
establish the general case.  Let $K$ be the Lie subalgebroid
$\ker(\omega_B)$ of $B$.  Recall the model Lie algebroid
$\bb{A}\to\bb{M}$ of Section~\ref{subsection;symplectization}, whose
base manifold $\bb{M}$ is the dual bundle $K^*$ and whose total space
$\bb{A}$ is the pullback of $B$ to $\bb{M}$.  We have a transverse
coisotropic embedding $\bb{j}\colon N\to\bb{M}$ and a family of
symplectic structures $\omega^s\in\Omega_{\bb{A}}^2(\bb{M})$
parametrized by splittings $s\colon K^*\to B^*$ of the surjective
vector bundle map $B^*\to K^*$.

\emph{Case}~1.  Let $A_0$ and $A_1$ be two copies of the model Lie
algebroid $\bb{A}$ and let $i_0=i_1$ be the canonical embedding
$\bb{j}\colon N\to\bb{M}$.  We equip $\bb{A}$ with two symplectic
forms $\omega_0=\omega^{s_0}$ and $\omega_1=\omega^{s_1}$
corresponding to two different splittings $s_0$, $s_1\colon K^*\to
B^*$.  The theorem in this case follows from Remark~\ref{remark;path}
and
Proposition~\ref{proposition;moser}\eqref{item;moser-subalgebroid}.

\emph{Case}~2.  Let $(A\to M,\omega)$ be an arbitrary symplectic Lie
algebroid equipped with a transverse coisotropic embedding $i\colon
N\to M$ such that $B\cong i^!A$ and $i_!^*\omega=\omega_B$.  We assert
that the theorem is true for the pair
\[
(A_0,\omega_0,i_0)=\bigl(\bb{A},\omega^s,\bb{j}\bigr)
\qquad\text{and}\qquad(A_1,\omega_1,i_1)=(A,\omega,i)
\]
for an appropriate choice of splitting $s\colon K^*\to B^*$ depending
on $A$.  We choose $s$ as follows.  Since $B$ is a coisotropic
subbundle of the symplectic vector bundle $A|_N=i^*\!A$, the map
\[
\begin{tikzcd}
i^*\!A\ar[r,"\omega^\sharp"]&(i^*\!A)^*\ar[r]&(B^\omega)^*=K^*
\end{tikzcd}
\]
has kernel $B$, which gives a short exact sequence
\begin{equation}\label{equation;coisotropic}
\begin{tikzcd}
B=i^!A\ar[r,hook]&i^*\!A\ar[r,two heads]&K^*
\end{tikzcd}
\end{equation}
of vector bundles on $N$.  Now choose an isotropic subbundle $L$ of
$i^*\!A$ which is complementary to $B$.  The map $i^*\!A\to K^*$
restricts to a vector bundle isomorphism
\begin{equation}\label{equation;complement}
\begin{tikzcd}
L\ar[r,"\cong"]&K^*.
\end{tikzcd}
\end{equation}
The subbundle $K\oplus L$ of $i^*\!A$ is symplectic and is isomorphic,
via the map~\eqref{equation;complement}, to $K\oplus K^*$ equipped
with its standard symplectic form.  We define $E=(K\oplus L)^\omega$
to be its symplectic orthogonal, so that we have an orthogonal
decomposition
\begin{equation}\label{equation;sum}
i^*\!A=E\oplus(K\oplus L)\cong E\oplus(K\oplus K^*).
\end{equation}
The subbundle $E$ of $B$ is complementary to $K$ and therefore defines
a splitting $s\colon K^*\to B^*$ of the surjection $B^*\to K^*$.  We
equip $\bb{A}$ with the form $\omega^s$ defined by this splitting $s$.
Next we explain how to map $\bb{A}$ to $A$.  The embedding $i$ is
transverse to $A$, so by
Lemma~\ref{lemma;constant-clean}\eqref{item;submanifold} the anchor
induces a bundle isomorphism $i^*\!A/B\cong\ca{N}(M,N)$.  Comparing
with~\eqref{equation;coisotropic} we obtain a canonical identification
\begin{equation}\label{equation;transverse-coisotropic}
\bb{M}=K^*\cong\ca{N}(M,N)
\end{equation}
between the model manifold and the normal bundle of $N$.  Via this
identification the model Lie algebroid $\bb{A}$ is just the pullback
$\pi_{M,N}^!B$, where $\pi_{M,N}\colon\ca{N}(M,N)\to N$ is the
projection.  In fact,
by~\cite[Lemma~3.8]{bursztyn-lima-meinrenken;splitting} we have a
natural isomorphism $\bb{A}\cong\ca{N}(A,B)$ between $\bb{A}$ and the
normal bundle of $B$ in $A$, which is a Lie algebroid over
$\ca{N}(M,N)$.  Let $\eps\in\Gamma(A)$ be a section which vanishes
along $N$ and whose normal derivative
\[\ca{N}(\eps)\colon\ca{N}(M,N)\longto\ca{N}(i^*\!A,0_M)\cong i^*\!A\]
is equal to the splitting of~\eqref{equation;coisotropic} given by the
complement~$L$.  Such a section $\eps$ exists and is Euler-like along
$B$; see the proof
of~\cite[Lemma~3.9]{bursztyn-lima-meinrenken;splitting}.  Let
\[\psi\colon\bb{A}\longto A\]
be the tubular neighbourhood embedding determined by the Euler-like
vector field $\bad_A(\eps)$ on~$A$.  The triangle
\begin{equation}\label{equation;embedding}
\begin{tikzcd}
&B\ar[dl,"\bb{j}_!"']\ar[dr,"i_!"]&\\
\bb{A}\ar[rr,"\psi"]&&A
\end{tikzcd}
\end{equation}
commutes and, according to~\cite[Theorem~3.13, Remark~3.19,
  Theorem~4.1]{bursztyn-lima-meinrenken;splitting}, $\psi$ is a
morphism of Lie algebroids.  The base map $\mr\psi\colon\bb{M}\to M$
is the tubular neighbourhood embedding defined by the Euler-like
vector field $\an_A(\eps)$ on~$M$, so $\psi$ is a Lie algebroid
isomorphism $\bb{A}\cong A|_U$, where $U=\mr\psi(\bb{M})$.  The
restriction of the bundle $\bb{A}\to\bb{M}$ to the submanifold $N$ is
$\bb{j}^*\bb{A}=B\oplus K^*$.  The map
\[\psi|_N\colon\bb{j}^*\bb{A}=B\oplus K^*\longto i^*\!A=B\oplus L\]
is the identity on $B$ (because of~\eqref{equation;embedding}) and on
$K^*$ is equal to the normal derivative $\ca{N}(\eps)$, i.e.\ the
inverse of the isomorphism~\eqref{equation;complement}.  So $\psi|_N$
is identical to the symplectic bundle
isomorphism~\eqref{equation;sum}, which shows that
$\psi^*\omega_x=\omega_x^s$ for all $x\in N$.  The Lie algebroid
$\bb{A}$ deformation retracts onto $B$ in view of
Remark~\ref{remark;transverse}\eqref{item;split}.  The theorem now
follows from Theorem~\ref{theorem;dmw} applied to the Lie algebroid
$\bb{A}$ and the symplectic forms $\omega^s$ and $\psi^*\omega$.
\end{proof}

\begin{remarks}\phantomsection\label{remark;coisotropic}
\begin{numerate}
\item\label{item;first-order}
Theorems~\ref{theorem;symplectization} and~\ref{theorem;coisotropic}
together say that locally near $B$ the symplectic Lie algebroid
$(A,\omega)$ is isomorphic to its first-order approximation
$(\bb{A},\omega^s)$ along $B$.
\item\label{item;clean}
There is no hope of obtaining a similar result for arbitrary clean
coisotropic submanifolds.  For instance, let $(A\to M,\omega)$ be an
arbitrary symplectic Lie algebroid.  Its associated \emph{adiabatic
  Lie algebroid} $\tilde{A}\to\tilde{M}$ has base manifold
$\tilde{M}=\R\times M$ and total space $\tilde{A}=\pr_2^*A=\R\times
A$.  The anchor $\an_{\tilde{A}}\colon\tilde{A}\to T\tilde{M}$ is
defined by $\an_{\tilde{A}}(t,a)=t\an_A(a)$ for $t\in\R$ and $a\in A$.
The Lie bracket $[\sigma,\tau]_{\tilde{A}}\in\Gamma(\tilde{A})$ for
sections $\sigma$, $\tau\in\Gamma(A)$ is defined by
\[[\sigma,\tau]_{\tilde{A}}(t,x)=t[\sigma,\tau]_A(x),\]
where $(t,x)\in\tilde{M}$.  This bracket extends uniquely to a bracket
on $\Gamma(\tilde{A})$ satisfying the Leibniz rule and so makes
$\tilde{A}$ a Lie algebroid.  The form
$\tilde\omega=\pr_2^*\omega\in\Omega_{\tilde{A}}^2(\tilde{M})$ is
$\tilde{A}$-symplectic.  Let $N=\{(0,x)\}$, where $x$ is any point in
$M$.  Then $N$ is an orbit of $\tilde{A}$ and hence is clean
coisotropic.  The induced Lie algebroid is $B=i_N^!\tilde{A}=A_x$
equipped with the zero Lie bracket and the form $\omega_B=\omega_x$.
It retains no memory of the Lie bracket on $A$ and therefore cannot
determine the structure of $\tilde{A}$ in a neighbourhood of $N$.  (Of
course the adiabatic Lie algebroid is not linearizable at $N$.  It
might be possible to build a first-order model for a coisotropic Lie
subalgebroid that admits an Euler-like section.)
\item\label{item;lagrange}
In the Lagrangian case (i.e.\ $\omega_B=0$) the model symplectic Lie
algebroid $\bb{A}$ is the phase space $\pi^!B$ of
Section~\ref{subsection;phase}.  The form on $\bb{A}$ is the canonical
symplectic form $\omega_\can$, which, in contrast to the general
coisotropic case, is linear along the fibres and independent of any
choices.  Theorem~\ref{theorem;coisotropic} tells us that a
neighbourhood of a transverse Lagrangian submanifold $N$ of a
symplectic Lie algebroid $(A,\omega)$ is equivalent to a neighbourhood
of the zero section in $(\bb{A},\omega_\can)$.  See
Smilde~\cite{smilde;linearization-poisson} for further linearization
theorems.
\end{numerate}
\end{remarks}

\section{Log symplectic manifolds}\label{section;log}

In this section we state a symplectic reduction theorem 
and a local normal form theorem, Theorem~\ref{theorem;log-normal}, in
the setting of log symplectic manifolds.  Our results extend some of
the work of Geudens and
Zambon~\cite{geudens-zambon;coisotropic-b-symplectic}, Gualtieri et
al.~\cite{gualtieri-li-pelayo-ratiu;tropical-toric-log} and Guillemin
et al.~\cite[\S\,6]{guillemin-miranda-pires;symplectic-poisson-b}.
See also Braddell et
al.~\cite{braddell-kiesenhofer-miranda;b-lie-poisson},
\cite{braddell-kiesenhofer-miranda;b-symplectic-slice}, and Matveeva
and Miranda~\cite{matveeva-miranda;reduction-singular} for related
recent work.

\subsection{Divisors with normal crossings}\label{subsection;divisor}

The following definition is a $\ca{C}^\infty$ analogue of a familiar notion
from algebraic geometry.

\begin{definition}\label{definition;normal-crossing}
Let $M$ be a manifold.  A \emph{(simple) normal crossing divisor} is a
locally finite collection $\ca{Z}$ of hypersurfaces (connected closed
embedded submanifolds of codimension $1$) in $M$, called the
\emph{components} of $\ca{Z}$, which intersect transversely in the
following sense: for all $x\in M$, if $Z_1$, $Z_2,\dots$,~$Z_k$ is the
list of all components of $\ca{Z}$ containing $x$ and if $f_i$ is a
defining function for $Z_i$ near $x$, then the differentials $df_1$,
$df_2,\dots$,~$df_k$ are independent at $x$.  The \emph{support} of a
normal crossing divisor $\ca{Z}$ is the union of all its components
and is denoted by $\abs{\ca{Z}}$.
\end{definition}

For the remainder of this section we will fix a manifold $M$ with a
normal crossing divisor $\ca{Z}$.

The collection $\ca{X}_{\ca{Z}}(M)$ of all vector fields on $M$ that
are tangent to (every component of) $\ca{Z}$ is a Lie subalgebra of
$\ca{X}(M)=\Gamma(TM)$ and is a finitely generated projective
$\ca{C}^\infty(M)$-module of rank $n=\dim(M)$.  Therefore
$\ca{X}_{\ca{Z}}(M)$ is the space of sections of a Lie algebroid
$T_{\ca{Z}}M$ of rank~$n$, which we call the \emph{logarithmic tangent
  bundle} of the pair $(M,\ca{Z})$.  If $\ca{Z}$ consists of a single
component $Z$, we also speak of the \emph{$b$-tangent bundle} of
$(M,Z)$ and we write $T_{\ca{Z}}M=T_ZM$.

Define $\ca{Z}^{(k)}$ to be the collection of all points of $M$ that
are contained in exactly $k$ distinct components of~$\ca{Z}$.  Then
$\ca{Z}^{(k)}$ is a submanifold of codimension $k$ of $M$, which we
call the \emph{codimension $k$ stratum}, and we have inclusions
\[
\overline{\ca{Z}^{(0)}}=M\supseteq\overline{\ca{Z}^{(1)}}=
\abs{\ca{Z}}\supseteq\overline{\ca{Z}^{(2)}}\supseteq\cdots.
\]
The image of the anchor $\an\colon T_{\ca{Z}}M\to TM$ at
$x\in\ca{Z}^{(k)}$ is equal to $T_x\ca{Z}^{(k)}$.  Thus the orbits of
the logarithmic tangent bundle are equal to the connected components
of the strata.  In particular the Lie algebroid $T_{\ca{Z}}M$
determines the divisor~$\ca{Z}$.

For each $x\in M$ the tangent spaces $T_xZ$ for $Z\in\ca{Z}$ define a
normal crossing divisor of the tangent space $T_xM$.  The manifold $M$
admits an atlas consisting of \emph{normal crossing charts},
i.e.\ charts $\phi\colon U\to\R^n$ with the property that $\phi(U\cap
Z_i)=\phi(U)\cap H_i$, where $Z_1$, $Z_2,\dots$,~$Z_k$ are the
components of $\ca{Z}$ intersecting $U$ and $H_i=\{\,x\in\R^n\mid
x_i=0\,\}$ is the $i$th coordinate hyperplane.

For each component $Z$ of $\ca{Z}$ the vector bundle $T_{\ca{Z}}M|_Z$
has a distinguished nowhere vanishing section $\xi_Z$, which at
generic points of $Z$ spans the kernel of the anchor.  In a
neighbourhood $U$ of a point $x\in Z$ this section is given by
$\xi_Z|_{Z\cap U}=(fv)|_{Z\cap U}$, where $f\colon U\to\R$ is a local
defining function for $Z$ and $v$ is any vector field on $U$ with
$\ca{L}(v)(f)=1$.  In a normal crossing chart at $x$ in which $Z$ is
given by $x_1=0$ we have $\xi_Z=x_1\pardif{}{x_1}\big|_Z$.  We define
\begin{equation}\label{equation;line}
L_Z=\sspan(\xi_Z)
\end{equation}
to be the trivial real line bundle on $Z$ spanned by $\xi_Z$.

Under appropriate conditions the intersection of a normal crossing
divisor with a submanifold $N$ is a normal crossing divisor of~$N$.

\begin{lemma}\label{lemma;cross}
Let $M$ be a manifold with normal crossing divisor $\ca{Z}$ and let
$A=T_{\ca{Z}}M$ be the logarithmic tangent bundle.  Let $i_N\colon
N\to M$ be a connected embedded submanifold which intersects $A$
cleanly.  Let $Z_1$, $Z_2,\dots$,~$Z_l$ be the list of all components
$Z$ of\/ $\ca{Z}$ such that $N\subseteq Z$.  Let $\ca{Z}_N$ be the
collection consisting of all intersections $Z\cap N$ with
$Z\in\ca{Z}\backslash\{Z_1,Z_2,\dots,Z_l\}$.  Then $\ca{Z}_N$ is a
normal crossing divisor of~$N$.  Let $L_j=L_{Z_j}\big|_N$, where
$L_{Z_j}$ is as in~\eqref{equation;line}, and let $V$ be the rank $l$
trivial vector bundle $L_1\oplus L_2\oplus\cdots\oplus L_l$,
considered as a Lie algebroid over $N$ with zero anchor $V\to TN$ and
trivial Lie bracket.  The pullback Lie algebroid $B=i_N^!A$ is
isomorphic to the direct sum $V\oplus T_{\ca{Z}_N}N$.  If $N$
intersects $A$ transversely, we have $l=0$ and $B\cong T_{\ca{Z}_N}N$.
\end{lemma}

\begin{proof}
Let $x\in N$.  It follows from
Proposition~\ref{proposition;pull}\eqref{item;pull-clean-transverse}
that there is a normal crossing chart $\phi\colon U\to\R^n$ at $x$
such that $\phi(U\cap N)$ is a coordinate subspace of~$\R^n$.  Let
$Z_1,\dots$,~$Z_l$ be those components $Z$ of $\ca{Z}$ for which the
image $\phi(U\cap Z)$ is a coordinate hyperplane of~$\R^n$ containing
$\phi(U\cap N)$.  (Such $Z$ do not exist if $N$ intersects $A$
transversely, so $l=0$ in the transverse case.)  Then $Z_j\cap N$ is
open and closed in $N$ for all $j$, so $Z_j\cap N=N$.  Moreover, for
all $Z\in\ca{Z}\backslash\{Z_1,\dots,Z_l\}$ the image $\phi(U\cap Z)$
is either empty or a coordinate hyperplane transverse to $N$.  Thus
$Z\cap N$ is a hypersurface in $N$ for each
$Z\in\ca{Z}\backslash\{Z_1,Z_2,\dots,Z_l\}$.  Restricting the normal
crossing chart $\phi$ to $N$ gives a normal crossing chart for the
pair $(N,\ca{Z}_N)$.  Since
\[\Gamma(i_N^*A)=\Gamma(A)/\Gamma(A;0_N),\]
we can think of a section of $A|_N=i_N^*A$ as an equivalence class
$\bar{v}=v\bmod\Gamma(A;0_N)$ of a section $v\in\Gamma(A)$.  By the
definition of pullbacks we have
\[
\Gamma(B)=\{\,\bar{v}\mid\text{$v\in\Gamma(A)$, $v$ is tangent to
  $N$}\,\}.
\]
Every $v\in\Gamma(A)$ which is tangent to $N$ is tangent to
$\ca{Z}_N$, so we have a natural $C^\infty$-linear Lie algebra
morphism $\phi\colon\Gamma(B)\to\Gamma(T_{\ca{Z}_N}N)$.  The kernel of
this morphism is
\[
\ker(\phi)=\{\,\bar{v}\mid\text{$v\in\Gamma(A)$, $v=0$ on
  $N$}\,\}=\Gamma(V).
\]
The ideal $\Gamma(V)$ of $\Gamma(B)$ is central, because sections of
$V$ induce trivial flows on $N$.  Every vector field $w$ on $N$ which
is tangent to $\ca{Z}_N$ extends to a vector field $\tilde{w}$ on $M$
which is tangent to $\ca{Z}$, and $\tilde{w}$ is unique modulo
$\Gamma(A;0_N)$.  Thus $\phi$ has a canonical splitting
$\Gamma(T_{\ca{Z}_N}N)\to\Gamma(B)$, which is a $C^\infty$-linear Lie
algebra morphism.  We conclude that $B=V\oplus T_{\ca{Z}_N}N$ as Lie
algebroids.
\end{proof}

Let $P$ be a second manifold with normal crossing divisor $\ca{D}$.
For a map $\phi\colon M\to P$ to induce a Lie algebroid morphism
$T_{\ca{Z}}M\to T_{\ca{D}}P$ it must map strata to strata, but that is
not enough.

\begin{example}\label{example;log-tangent}
Let $M=P=\R$ and let $\ca{Z}=\ca{D}=\{0\}$.  Let $\phi\colon M\to P$
be a smooth function with $\phi(0)=0$ and $\phi(x)\ne0$ if $x\ne0$.
Then $\phi$ maps strata to strata, but $\phi$ lifts to a Lie algebroid
morphism $T_{\ca{Z}}M\to T_{\ca{D}}P$ if and only if the function
$x\mapsto x\phi'(x)$ is smoothly divisible by $\phi$.  This is the
case if and only if $\phi$ is not flat at $0$.
\end{example}

\begin{definition}\label{definition;log-tangent}
A \emph{morphism} $\phi\colon(M,\ca{Z})\to(P,\ca{D})$ of manifolds
with normal crossing divisors is a smooth map $\phi\colon M\to P$ with
the following property: every component $D$ of $\ca{D}$ has a covering
$\eu{V}$ consisting of open subsets of $P$ such that for each
$V\in\eu{V}$ either (1)~$\phi^{-1}(D\cap V)=\phi^{-1}(V)$, or
(2)~there is a component $Z$ of\/ $\ca{Z}$ with $\phi^{-1}(D\cap
V)=Z\cap\phi^{-1}(V)$.  In case~(2) we demand additionally that if $g$
is a local defining function for $D$, then there exist an integer
$\nu\ge1$ and a local defining function $f$ for $Z$ such that
$\phi^*g=f^\nu$.
\end{definition}

We call the lift $T^{\log}\phi\colon T_{\ca{Z}}M\to T_{\ca{D}}P$ of a
morphism $\phi$ guaranteed by the next lemma the \emph{log tangent
  map} of~$\phi$.

\begin{lemma}\label{lemma;log-tangent}
Let $\phi\colon(M,\ca{Z})\to(P,\ca{D})$ be a morphism of manifolds
with normal crossing divisors.
\begin{enumerate}
\item\label{item;log-lift}
The map $\phi$ lifts to a unique Lie algebroid morphism
$T^{\log}\phi\colon T_{\ca{Z}}M\to T_{\ca{D}}P$ whose restriction to
$M\backslash\abs{\ca{Z}}$ agrees with the usual tangent map of\/
$\phi$.  In particular, if $M^0$ is a connected component of\/ $M$
whose image $\phi(M^0)$ is contained in a component $D$ of\/ $\ca{D}$,
then $T^{\log}\phi$ maps $T_{\ca{Z}}M|_{M^0}$ to the Lie subalgebroid
$T_{\ca{D}_D}D$ of\/ $T_{\ca{D}}P$, where $\ca{D}_D$ is the normal
crossing divisor of\/ $D$ defined in Lemma~\ref{lemma;cross}.
\item\label{item;order}
Let $Z$ be a component of\/ $\ca{Z}$ and $\xi_Z$ the distinguished
section of\/ $T_{\ca{Z}}M|_Z$.  Suppose there is a component $D$ of\/
$\ca{D}$ such that $\phi(Z)\subseteq D$ and for some local defining
function $g$ of\/ $D$ the function $\phi^*g$ vanishes at $Z$ to
constant order $\nu_Z<\infty$.  Then
$T^{\log}\phi(\xi_Z)=\nu_Z\eta_D$, where $\eta_D$ is the distinguished
section of\/ $T_{\ca{D}}P|_D$.  If there is no such component $D$, we
have $T^{\log}\phi(\xi_Z)=0$.
\end{enumerate}
\end{lemma}

\begin{proof}
\eqref{item;log-lift}~Let $R=C^\infty(M)$ and $S=C^\infty(P)$.  We
regard $R$ as an $S$-module via the pullback map $\phi^*\colon S\to
R$.  The usual tangent map $T\phi\colon TM\to TP$ is the vector bundle
map induced by the pushforward map
\[
\phi_*\colon\ca{X}(M)=\Der(R)\longto\Gamma(\phi^*TP)=R\otimes_S\ca{X}(P)=
\Der(S,R),
\]
which maps $D\in\Der(R)$ to the derivation $\phi_*D\in\Der(S,R)$
defined by $(\phi_*D)(g)=D(\phi^*g)$ for $g\in S$.  We must show that
$\phi_*$ lifts to an $R$-linear Lie algebra map $\phi_*^{\log}$,
\[
\begin{tikzcd}
\ca{X}_{\ca{Z}}(M)\ar[r,dashed,"\phi_*^{\log}"]
\ar[d,hook,"a_M"']&R\otimes_S\ca{X}_{\ca{D}}(P)
\ar[d,"a_P"]\\
\ca{X}(M)\ar[r,"\phi_*"]&R\otimes_S\ca{X}(P),
\end{tikzcd}
\]
where $a_M=\an_M$ is the anchor map for $(M,\ca{Z})$ and
$a_P=\id_R\otimes\an_P$, with $\an_P$ being the anchor map for
$(P,\ca{D})$.  First suppose that $M$ is an open neighbourhood of the
origin in $\R^n$ with normal crossing divisor
$\ca{Z}=\{Z_1,Z_2,\dots,Z_l\}$, where $Z_j=\{\,x\in\R^n\mid
x_j=0\,\}$, and that $P$ is an open neighbourhood of the origin in
$\R^m$ with normal crossing divisor $\ca{D}=\{D_1,D_2,\dots,D_k\}$,
where $D_i=\{\,y\in\R^n\mid y_i=0\,\}$, and that $\phi(0)=0$.  Let
$\phi_1$, $\phi_2,\dots$,~$\phi_m\in R$ be the components of $\phi$.
The $R$-modules $\ca{X}(M)$ and $R\otimes_S\ca{X}(P)$ are free with
bases
\[
\pardif{}{x_1},\dots,\pardif{}{x_n},\quad\text{resp.}\quad
1\otimes\pardif{}{y_1},\dots,1\otimes\pardif{}{y_m},
\]
The $R$-modules $\ca{X}_{\ca{Z}}(M)$ and
$R\otimes_S\ca{X}_{\ca{D}}(P)$ are also free with bases
\[
x_1\pardif{}{x_1},\dots,x_l\pardif{}{x_l},\pardif{}{x_{l+1}},\dots,
\pardif{}{x_n},\quad\text{resp.}\quad 1\otimes
y_1\pardif{}{y_1},\dots,1\otimes
y_k\pardif{}{y_k},\pardif{}{y_{k+1}},\dots, \pardif{}{y_m}.
\]
Therefore, finding the map $\phi_*^{\log}$ amounts to finding an
$m\times n$-matrix $\Psi=(\psi_{ij})$ with entries in the ring $R$
such that $A_P\Psi=\Phi A_M$, where
$\Phi=\bigl(\pardif{\phi_i}{x_j}\bigr)$ is the matrix of $\phi_*$ and
\[
A_M=\diag(x_1,x_2\dots,x_l,1,1,\dots,1),\qquad
A_P=\diag(\phi_1,\phi_2,\dots,\phi_k,1,1,\dots,1)
\]
are the matrices of $a_M$, resp.\ $a_P$.  We obtain the following
equations for the $\psi_{ij}$:
\begin{align*}
\phi_i\psi_{ij}&=\begin{cases}x_j\partial\phi_i/\partial x_j&\text{for
  $1\le i\le k$,\quad $1\le j\le l$}\\\partial\phi_i/\partial
x_j&\text{for $1\le i\le k$,\quad $l+1\le j\le n$}\end{cases}\\
\psi_{ij}&=\begin{cases}x_j\partial\phi_i/\partial x_j&\text{for
  $k+1\le i\le m$,\quad $1\le j\le l$}\\\partial\phi_i/\partial
x_j&\text{for $k+1\le i\le m$,\quad $l+1\le j\le n$.}\end{cases}
\end{align*}
Clearly $\psi_{ij}$ is uniquely determined for $k+1\le i\le m$.  Let
$1\le i\le k$.  According to Definition~\ref{definition;log-tangent}
we are in either of two cases.  In case~1 we have
$\phi_i=\phi^*y_i=0$, i.e.\ $\phi(M)$ is contained in $D_i$.  In this
case we want $\phi_*^{\log}$ to map $\ca{X}_{\ca{Z}}(M)$ to
$\ca{X}_{\ca{D}_{D_i}}(D_i)$.  In other words, if we take any
$v\in\ca{X}_{\ca{Z}}(M)$ and express $\phi_*^{\log}(v)$ in the basis
of $\ca{X}_{\ca{D}}(P)$, we want the coefficient of
$y_i\pardif{}{y_i}$ to be equal to $0$.  This forces us to put
$\psi_{ij}=0$ for all $j$.  In case~2 we have a unique $j=j_i\le l$
such that $\phi_i$ vanishes to constant finite order $\nu_i$ along
$Z_{j_i}$.  Thus we have $\phi_i(x)=x_{j_i}^{\nu_i}f_i(x)$ for some
$f_i\in R$ that vanishes nowhere on $Z_{j_i}$.  Hence for $1\le j\le
l$
\begin{equation}\label{equation;order}
\psi_{ij}=\frac{x_j\partial\phi_i/\partial x_j}{\phi_i}=
\begin{cases}
\frac{x_j\partial f_i/\partial x_j}{f_i}&\text{if $j\ne j_i$}\\
\nu_i+\frac{x_{j_i}\partial f_i/\partial x_{j_i}}{f_i}&\text{if
  $j=j_i$,}
\end{cases}
\end{equation}
while for $l+1\le j\le n$
\[
\psi_{ij}=\frac{\partial\phi_i/\partial x_j}{\phi_i}=\frac{\partial
  f_i/\partial x_j}{f_i}.
\]
We see that for $1\le i\le m$ and $1\le j\le n$ the function
$\psi_{ij}$ is a well-defined and uniquely determined element of $R$.
This proves the existence of a unique $R$-linear lift $\phi_*^{\log}$
locally, and hence globally by a gluing argument.  The map
$\phi_*^{\log}$ is a Lie algebra homomorphism, because on the open
dense set $M\backslash\abs{\ca{Z}}$ it agrees with the pushforward map
$\phi_*$.  It follows that the associated vector bundle map
$T^{\log}\phi$ is a Lie algebroid morphism.

\eqref{item;order}~Setting $x_{j_i}=0$ in~\eqref{equation;order}
yields $\psi_{ij_i}=\nu_i$ on $Z=Z_{j_i}$,
i.e.\ $T^{\log}\phi(\xi_Z)=\nu_i\eta_{D_i}$.
\end{proof}

A \emph{log Poisson structure} on $(M,\ca{Z})$ is an $A$-Poisson
structure (Definition~\ref{definition;poisson}) on $M$, where $A$ is
the log tangent bundle $T_{\ca{Z}}M$.  A \emph{log symplectic form} on
$(M,\ca{Z})$ is an $A$-symplectic form
(Definition~\ref{definition;symplectic}) on $M$.  If $\ca{Z}$ consists
of a single component, we also speak of \emph{$b$-Poisson} and
\emph{$b$-symplectic} structures.

The reduction theorem, Theorem~\ref{theorem;mikami-weinstein},
specializes to the following result for a log Poisson morphism
$(M,\ca{Z})\to(P,\ca{D})$.  We can reduce $M$ at any ``point''
$\f\subseteq\stab(T_{\ca{D}}P,p)$, but the resulting quotient is not
necessarily log symplectic unless we reduce at the trivial subalgebra
$\f=0$.  If we reduce at a nonzero ``point'', the reduced symplectic
Lie algebroid will be a direct sum of a log tangent bundle and a
trivial Lie algebroid.

\begin{theorem}[log symplectic reduction]%
\label{theorem;log-mikami-weinstein}
Let\/ $(M,\ca{Z},\omega)$ be a log symplectic manifold and let
$(P,\ca{D},\lambda)$ be a log Poisson manifold.  Let
$\mu\colon(M,\ca{Z})\to(P,\ca{D})$ be a log Poisson morphism and let
$\gamma\colon\Gamma(T_{\ca{D}}^*P)\to\Gamma(T_{\ca{Z}}M)$ be the
associated Hamiltonian action.  Let $p\in P$ and $N=\mu^{-1}(p)$.
Suppose that $p$ is a regular value of\/ $\mu\colon M\to P$ and that
the action $\gamma$ is locally free at $x$ for all $x\in N$.
\begin{enumerate}
\item\label{item;log-lie-fibre-act}
The submanifold $N$ of~$M$ intersects the Lie algebroid $T_{\ca{Z}}M$
cleanly, the $2$-form $\omega_N=i_N^*\omega$ is logarithmic relative
to the normal crossing divisor $\ca{Z}_N$ and is presymplectic.  The
null Lie algebroid $K=\ker(\omega_N)$ is a foliation Lie algebroid and
the Lie algebra
\[\h=\stab(T^*_{\ca{D}}P,p)\]
acts transitively on~$K$.
\item\label{item;log-reduction}
Suppose that the leaf space $Q=N/\h$ is a manifold.  Then the
collection $\ca{Z}_Q$ consisting of all quotients $Z/\h$\/ for
$Z\in\ca{Z}_N$ is a normal crossing divisor of\/ $Q$ and $\omega_N$
descends to a log symplectic structure on $(Q,\ca{Z}_Q)$.
\end{enumerate}
\end{theorem}

\begin{proof}
Theorem~\ref{theorem;mikami-weinstein}\eqref{item;lie-fibre-act},
applied to the Lie algebroids $A=T_{\ca{Z}}M$ and $E=T_{\ca{D}}P$,
says that $T^{\log}\mu$ is transverse to any Lie subalgebra $\f$ of\/
$\stab(E,p)$.  Taking $\phi=T^{\log}\mu$ and $\f=\stab(E,p)$ in
Corollary~\ref{corollary;clean-regular} we find that $N$ cleanly
intersects $A$ and that $i_N^!A=(T^{\log}\mu)^{-1}(\stab(E,p))$.  Let
$D_1$, $D_2,\dots$,~$D_k$ be the components of $\ca{D}$ that contain
$p$ and let $\eta_i$ be the distinguished section of\/
$T_{\ca{D}}P|_{D_i}$.  The elements $\eta_1(p)$,
$\eta_2(p),\dots$,~$\eta_k(p)\in E_p$ form a basis of $\stab(E,p)$.
Let $x\in N$ and let $N_0$ be the connected component of $N$
containing $x$.  Since $\mu$ is a morphism, there exist components
$Z_1$, $Z_2,\dots$,~$Z_k$ of $\ca{Z}$ that contain $N_0$ and satisfy
$\mu(Z_i)\subseteq D_i$.  Since $p$ is a regular value of $\mu$, we
have $T^{\log}\mu(\xi_i)=\eta_i$, where $\xi_i$ is the distinguished
section of\/ $T_{\ca{Z}}M|_{Z_i}$.  By Lemma~\ref{lemma;log-tangent},
for any component $Z$ of $\ca{Z}\backslash\{Z_1,Z_2,\dots,Z_k\}$ that
contains $x$ we have $T^{\log}\mu(\xi_Z)=0$.  It now follows from
Lemma~\ref{lemma;cross} that $i_N^!A=V\oplus T_{\ca{Z}_N}N$, where $V$
is a trivial $k$-dimensional vector bundle equipped with a trivial Lie
algebroid structure and $\ca{Z}_N$ is the induced divisor of~$N$.
Moreover, the preimage of the zero ``point'' $\f=0_p$ is the log
tangent bundle of the induced normal crossing divisor $\ca{Z}_N$,
\[
T_{\ca{Z}_N}N=(T^{\log}\mu)^{-1}(0_p).
\]
Theorem~\ref{theorem;mikami-weinstein}\eqref{item;lie-fibre-act} (with
$\f=0_p$) shows that the Lie algebra $\h=\stab(T^*_{\ca{D}}P,p)$ acts
transitively on~$K=\ker(\omega_N)$.  Now suppose that the leaf space
$Q=N/\h$ is a manifold.  To see that $Z_Q$ is a normal crossing
divisor of~$Q$, choose an open subset $U$ of $Q$ such that the
quotient map $q\colon N\to Q$ admits a section $s\colon U\to
V=q^{-1}(U)$.  Then $s$ is transverse to the logarithmic tangent
bundle $T_{\ca{Z}_N}N$, so by Lemma~\ref{lemma;cross} $s$ is
transverse to all strata of $\ca{Z}_N$ and $\ca{Z}_Q|_U$ is a normal
crossing divisor of~$U$.  Thus we have a well-defined log tangent
bundle $T_{\ca{Z}_Q}Q$ (the holonomy condition of
Theorem~\ref{theorem;mikami-weinstein}\eqref{item;reduction} is
vacuous here), which is the quotient of $T_{\ca{Z}_N}N$ by $K$, and
the log presymplectic form $\omega_N$ descends to a log symplectic
form on $Q$.
\end{proof}

\begin{example}
Let $M$ be the plane $\R^2$, $Z$ the $y$-axis,
$\omega=x^{-1}\,dx\wedge dy$, and $A=T_ZM$.  Let $P=\R$, $D=\{0\}$,
$\lambda=0$ the zero Poisson structure, and $E=T_DP$.  The map
$\mu(x,y)=x$ is log Poisson.  We have $\stab(E,p)=0$ for $p\in
P\backslash\{0\}$ and $\stab(E,0)=\R$.  Let us reduce $M$ at the
``point'' $\f=0_p$ for any $p\in P$.  We have
$N=\mu^{-1}(p)=\{p\}\times\R$, $(T^{\log}\mu)^{-1}(\f)=TN$,
$\h=\stab(E^*,0)=\R$, which acts on $M$ as the vertical vector field
$\partial/\partial y$.  The reduced space $Q=N/\h$ is a point equipped
with a zero-dimensional Lie algebroid.  However, if we reduce at $p=0$
with respect to the ``point'' $\f=\stab(E,p)=\R$, we have
$N=\mu^{-1}(p)=Z$, $(T^{\log}\mu)^{-1}(\f)=A|_Z$, and $\h=0$.  The
reduced space is now the divisor $Q=Z/\h=Z$ equipped with the
symplectic Lie algebroid~$A|_Z$.
\end{example}

See Section~\ref{subsection;example} for further examples.  The log
symplectic version of Lemma~\ref{lemma;regular} is the following.

\begin{lemma}\label{lemma;log-regular}
In the context of Theorem~\ref{theorem;log-mikami-weinstein} suppose
that $p$ is contained in the complement of the divisor $\ca{D}$.  Let
$x\in\mu^{-1}(p)$ and let $L$ be the symplectic leaf of $x$ with
respect to the Poisson structure on $M$ determined by the log
symplectic form $\omega$.  Then the Hamiltonian action\/ $\gamma$ is
locally free at $x$ if and only if $x$ is a regular point of the map
$\mu|_L$.
\end{lemma}

The following normal form theorem, which extends a result of Guillemin
and Sternberg~\cite[\S\,1]{guillemin-sternberg;birational}, is a
direct consequence of Theorems~\ref{theorem;coisotropic}
and~\ref{theorem;log-mikami-weinstein}.  It involves a log Poisson
morphism $\mu\colon M\to\g^*$, where $\g^*$ is the dual of a Lie
algebra $\g$ (equipped with the empty divisor).  If the infinitesimal
Hamiltonian action $\gamma\colon\g\to T_{\ca{Z}}M$ determined by $\mu$
integrates to an action of a connected Lie group $G$ with Lie algebra
$\g$, we call the tuple $(M,\ca{Z},\omega,\mu)$ a \emph{log
  Hamiltonian $G$-manifold}.

\begin{theorem}[normal form near zero fibre for log symplectic forms]
\label{theorem;log-normal}
Let\/ $(M,\ca{Z},\omega,\mu)$ be a log Hamiltonian $G$-manifold and
let $N=\mu^{-1}(0)$ be the zero fibre of $\mu$.  Suppose that the
$G$-action on $N$ is proper and free.  Let $Q=N/G$ be the quotient
manifold, $\ca{Z}_Q$ its normal crossing divisor, and\/ $\omega_Q$ its
log symplectic form.  Choose a connection\/ $1$-form\/
$\theta\in\Omega^1(N,\g)^G$ on the principal $G$-bundle $q\colon N\to
Q$.  Let\/ $\bb{M}$ be the product $N\times\g^*$ with projections\/
${\pr}_1\colon\bb{M}\to N$ and\/ ${\pr}_2\colon\bb{M}\to\g^*$.  Let
$\ca{Z}_{\bb{M}}$ be the normal crossing divisor
$i_N^!\ca{Z}\times\g^*$ of~$\bb{M}$.  The closed\/ $2$-form
\[
\omega_{\bb{M}}=\pr_1^*q^*\omega_Q+
d\biginner{\pr_2,\pr_1^*\theta}\in\Omega^2\bigl(\bb{M}\backslash
\abs{\ca{Z}_{\bb{M}}}\bigr)
\]
has logarithmic singularities at $Z_{\bb{M}}$ and is symplectic in a
neighbourhood of $N=N\times\{0\}$ in~$\bb{M}$.  The $G$-action on $N$
and the coadjoint action on\/ $\g^*$ combine to a Hamiltonian
$G$-action\/ $\gamma_{\bb{M}}$ on $\bb{M}$ with moment map\/
$\mu_{\bb{M}}=\pr_2$.  In an open neighbourhood of $N$ the log
Hamiltonian $G$-manifold\/ $(M,Z,\omega,\mu)$ is isomorphic to\/
$(\bb{M},Z_{\bb{M}},\omega_{\bb{M}},\mu_{\bb{M}})$.
\end{theorem}  

\subsection{Some examples of log symplectic reduction}
\label{subsection;example}

In this section $(M,\ca{Z},\omega)$ denotes a log symplectic manifold
with log tangent bundle $A=T_{\ca{Z}}M$.

\begin{example}[reduction with respect to the identity map]
\label{example;log-identity}
Let $\mu=\id_M$ be the identity map of\/ $M$.  The symplectic quotient
of\/ $M$ at a ``point'' $(x,\liea)$, where $\liea$ is a Lie subalgebra
of\/ $\stab(A,x)$, is the point $x$ equipped with the symplectic Lie
algebra $\liea/(\liea\cap\liea^\omega)$; see
Example~\ref{example;identity}.  If\/ $x\in\ca{Z}^{(k)}$, then
$\stab(A,x)$ is a $k$-dimensional abelian Lie algebra, spanned in a
normal crossing chart by the elements $x_1\tangent x_1$, $x_2\tangent
x_2,\dots$,~$x_k\tangent x_k$.  Hence $\liea/(\liea\cap\liea^\omega)$
is abelian as well.
\end{example}

\begin{example}[log cotangent reduction]\label{example;log-cotangent}
Let $X$ be a manifold and $\ca{Z}_X$ a normal crossing divisor of~$X$.
If\/ $\pi\colon Y\to X$ is a submersion, then
$\ca{Z}_Y=\pi^{-1}(\ca{Z}_X)$ is a normal crossing divisor of\/ $Y$
and the log tangent bundle of\/ $(Y,\ca{Z}_Y)$ is the pullback Lie
algebroid $\pi^!B$, where $B=T_{\ca{Z}_X}X$.  Taking $Y=B^*$ and $\pi$
to be the bundle projection $B^*\to X$, in which case $\pi^!B$ is the
phase space of the Lie algebroid $B$ (see
Section~\ref{subsection;phase}), we obtain that the log cotangent
bundle $M=T_{\ca{Z}_X}^*X$ equipped with the divisor
$\ca{Z}=\pi^{-1}(\ca{Z}_X)$ and the form $\omega_\can=-d\alpha_\can$
is log symplectic.  Let $G$ be a Lie group and $G\times X\to X$ an
action that preserves $\abs{\ca{Z}_X}$.  The action lifts naturally to
an action $G\times M\to M$, which preserves $\abs{\ca{Z}}$ and leaves
the Liouville form $\alpha_\can$ invariant.  Therefore the lifted
action is log Hamiltonian with moment map $\mu\colon M\to\g^*$ given
by $\inner{\mu,\xi}=\iota_B(\xi_X)\alpha_\can$ for $\xi\in\g$.  If the
$G$-action on $X$ is proper and free, the quotient $Q=X/G$ is a
manifold with normal crossing divisor $\ca{Z}_Q=\ca{Z}_X/G$, and the
symplectic quotient of\/ $M$ at $0$ is naturally isomorphic to the log
cotangent bundle $T_{\ca{Z}_Q}^*Q$ of~$Q$.
\end{example}

\begin{example}[log linear Poisson structures]
\label{example;log-affine} 
Let $\g$ be a finite-dimensional real Lie algebra and let
$\bb{v}\colon\g\to\R^l$ be a surjective Lie algebra homomorphism to
the abelian Lie algebra $\R^l$.  We will regard $\bb{v}$ as a linearly
independent $l$-tuple $v_1$, $v_2,\dots$,~$v_l\in\g^*$ of characters
of~$\g$.  Extend $v_1$,~$v_2,\dots$,~$v_l$ to a basis
$v_1$,~$v_2,\dots$,~$v_n$ of\/ $\g^*$ and let
$c_{ij}^k=v_k([v^*_i,v^*_j])$ be the corresponding structure
constants.  Then $c_{ij}^k=0$ for $k\le l$ because $v_k$ is a
character for $k\le l$.  The usual linear Poisson structure on $\g^*$
is given by
\[
\lambda_0=\sum_{1\le i<j\le
  n}c_{ij}(y)\pardif{}{y_i}\wedge\pardif{}{y_j},
\]
where $c_{ij}(y)=\sum_{k=l+1}^nc_{ij}^ky_k$.  Substituting
\begin{equation}\label{equation;log-substitution}
y_k=\log\abs{x_k}\quad\text{for $k\le l$},\qquad y_k=x_k\quad\text{for
  $k\ge l+1$}
\end{equation}
into $\lambda_0$ yields the cubic Poisson structure
\begin{multline}\label{equation;log-poisson}
\lambda=\sum_{1\le i<j\le
  l}x_ix_jc_{ij}(x)\pardif{}{x_i}\wedge\pardif{}{x_j}+
\sum_{i=1}^l\sum_{j=l+1}^nx_ic_{ij}(x)
\pardif{}{x_i}\wedge\pardif{}{x_j}+
\\
\sum_{l+1\le i<j\le n}c_{ij}(x)\pardif{}{x_i}\wedge\pardif{}{x_j}.
\end{multline}
This is a priori only defined for $x_1x_2\cdots x_l\ne0$, but
manifestly extends to a smooth Poisson structure on $P=\R^n$, which is
log Poisson with respect to the divisor
$\ca{D}=\{D_1,D_2,\dots,D_l\}$, where $D_k$ is the coordinate
hyperplane $\{x_k=0\}$.  Thus $(P,\ca{D},\lambda)$ is a log Poisson
manifold.  We will also denote this log Poisson manifold by
$P(\g,\bb{v})$.  The top stratum of\/ $P$, i.e.\ the complement of
$\abs{\ca{D}}$, is the disjoint union of\/ $2^l$ copies of the linear
Poisson space $(\g^*,\lambda_0)$, so we can think of\/ $P$ as the
result of gluing $2^l$ copies of\/ $\g^*$ along hyperplanes.  Each of
the components of the divisor $\ca{D}$ is a log Poisson manifold in
its own right, namely $D_k=P(\g_k,\bb{v}_k)$, where $\g_k$ is the
ideal $\ker(v_k)$ of\/ $\g$ and $\bb{v}_k$ is the $l-1$-tuple
\[
v_1|_{\g_k},\,v_2|_{\g_k},\dots,\,v_{k-1}|_{\g_k},\,v_{k+1}|_{\g_k},\dots,
\,v_l|_{\g_k}
\]
of characters of\/ $\g_k$.  Just as the Poisson manifold $\g^*$
integrates to a symplectic groupoid, so the log Poisson manifold $P$
integrates to a log symplectic groupoid $M$, as follows.  Let $G$ be a
connected Lie group with Lie algebra $\g$ such that the characters
$v_k$ exponentiate to characters $\bar{v}_k\colon G\to\R$ (e.g.\ take
$G$ to be simply connected).  We denote the tuple of characters
$(\bar{v}_1,\bar{v}_2,\dots,\bar{v}_k)$ by $\bar{\bb{v}}$.  Making the
substitution~\eqref{equation;log-substitution} into the coadjoint
action of\/ $G$ on $\g^*$ yields a smooth action of\/ $G$ on $P$ that
preserves the divisor $\ca{D}$ and the Poisson structure~$\lambda$.
We define
\[M=M(G,\bar{\bb{v}})=G\ltimes P\]
to be the corresponding action groupoid and $\ca{Z}$ the divisor
$G\times\ca{D}$ of~$M$.  Recall that the action groupoid
$G\ltimes\g^*=T^*G$ of the coadjoint action on $\g^*$, equipped with
the cotangent symplectic form $\omega_\can=-d\alpha_\can$, is a
symplectic groupoid.  Let $\theta\in\Omega^1(G,\g)$ be the
left-invariant Maurer-Cartan form.  We have
$\alpha_\can=\sum_ky_k\theta_k$, where the $\theta_k$ are the
components of\/ $\theta$ relative to the basis~$v_k^*$, so the
Maurer-Cartan equation gives
\[
\omega_\can=\sum_{k=1}^n(\theta_k\wedge dy_k-y_k\,d\theta_k)=
\sum_{k=1}^n\theta_k\wedge dy_k-\frac12\sum_{1\le i<j\le
  n}c_{ij}(y)\theta_i\wedge\theta_j.
\]
Substituting~\eqref{equation;log-substitution} into
$\omega_\can$ yields the form
\begin{equation}\label{equation;log-affine}
\omega=\sum_{k=1}^l\theta_k\wedge\frac{dx_k}{x_k}+
\sum_{k=l+1}^n\theta_k\wedge dx_k-\frac12\sum_{1\le i<j\le
  n}c_{ij}(x)\theta_i\wedge\theta_j
\end{equation}
on $M$, which is log symplectic with respect to the divisor $\ca{Z}$.
The log symplectic groupoid $M(G,\bar{\bb{v}})$ integrates the log
Poisson manifold $P(\g,\bb{v})$ in the sense
of~\cite[\S\,III]{coste-dazord-weinstein;groupoides}.  The projection
$\mu=\pr_2\colon M\to P$, i.e.\ the source map of the groupoid, is a
log Poisson map.  The Hamiltonian action of\/ $T_{\ca{D}}^*P=\g\ltimes
P$ on $T_{\ca{Z}}M$ with moment $\mu$ is the infinitesimal left
translation action of\/ $\g$ on~$G$.  The fibre of\/ $\mu$ at a point
$p\in P$ is $\mu^{-1}(p)=G\times\{p\}$.  The stabilizer
$\h=\stab(T_{\ca{D}}^*P,p)\subseteq\g$ of\/ $p=(p_1,p_2,\dots,p_n)\in
P=\R^n$ is equal to the coadjoint stabilizer of the element
$v(p)=\sum_{k=1}^np_kv_k\in\g^*$.  Thus the reduced space of\/ $M$
with respect to $\mu$ at $p$ is $G/\h$, which is up to a covering map
the coadjoint $G$-orbit of\/ $v(p)$.
\end{example}

\begin{example}\label{example;log-cotangent-affine}
The log symplectic groupoid $M(G,\bar{\bb{v}})$ of
Example~\ref{example;log-affine} can itself be obtained by a log
symplectic reduction as follows.  Let $\ca{Z}_0$ be the normal
crossing divisor $q_1q_2\cdots q_l=0$ of~$\R^l$.  The action of the
Lie group $H=\R^l$ on $\R^l$ given by
\[h\cdot q=(e^{h_1}q_1,e^{h_2}q_2,\dots,e^{h_l}q_l)\]
preserves the divisor.  By Example~\ref{example;log-cotangent} the
lift of the action to the log cotangent bundle $T_{\ca{Z}_0}^*\R^l$ is
log Hamiltonian with respect to the log symplectic form
$\omega_0=\sum_{k=1}^lq_i^{-1}\,dq_i\wedge dp_i$, with moment map
$\phi_0\colon T_{\ca{Z}_0}^*\R^l\to\Lie(H)=\R^l$ given by
$\phi_0(q,p)=p$.  The action is not proper and we cannot form a good
quotient.  But let $G$ be another Lie group and let
$\bar{\bb{v}}\colon G\to\R^l$ be a surjective Lie group homomorphism.
We regard $\bar{\bb{v}}$ as an $l$-tuple of characters
$\bar{v}_k\colon G\to\R$; the infinitesimal characters
$v_k=d_1\bar{v}_k\colon\g\to\R$ are then linearly independent for
$1\le k\le l$.  The $H$-action on $T_{\ca{Z}_0}\R^l\times
T^*G=T_{\ca{Z}_0}\R^l\times G\times\g^*$ defined by
\[
h\cdot(q,p,g,y)=\biggl(h\cdot q,p,g,y-\sum_{k=1}^lh_kv_k\biggr)
\]
for $h\in H$, $q\in\R^l$, $g\in G$, and $y\in\g^*$ is proper and free,
because the translation action of\/ $\g^*$ on itself is proper and free.
This $H$-action is log Hamiltonian relative to the divisor
$\ca{Z}_0\times T^*G$ and the log symplectic form
$\omega_0+\omega_\can$, where $\omega_\can$ is the cotangent
symplectic form on $T^*G$.  The moment map $\phi\colon
T_{\ca{Z}_0}\R^l\times T^*G\to\R^l$ is given by
\[\phi(q,p,g,y)=p+\bar{\bb{v}}(g),\]
the symplectic quotient at $0$ is
\[(\R^l\times T^*G)/H=G\times(\R^l\times\g^*)/H=M(G,\bar{\bb{v}}),\]
and the reduced log symplectic form agrees
with~\eqref{equation;log-affine}.  Similarly, the log Poisson manifold
$P(\g,\bb{v})$ is the quotient $(\R^l\times\g^*)/H$, where we give
$\R^l$ the zero Poisson structure and $\g^*$ its linear Poisson
structure $\lambda_0$.  The action of\/ $\g^*$ on $\R^l\times\g^*$ given
by translation in the second factor descends to a $\g^*$-action on
$P(\g,\bb{v})$, the orbits of which are the strata.  The spaces
$P(\g,\bb{v})$ are a nonabelian version of the tropical welded spaces
of Gualtieri et
al.~\cite[\S\,3.2]{gualtieri-li-pelayo-ratiu;tropical-toric-log} and
the moment codomains of Guillemin et
al.~\cite[\S\,5]{guillemin-miranda-pires-scott;toric-b-symplectic}.
\end{example}

\appendix

\section{Lie algebroids}\label{section;algebroid}

This appendix is an exposition of some elementary properties of Lie
algebroids, including fibred products, pullbacks, a regular value
theorem, and differential forms.  This is mostly standard material
that is covered more fully
in~\cite{fernandes;lie-algebroids-holonomy},
\cite{higgins-mackenzie;algebraic-lie-algebroids},
\cite{mackenzie;lie-groupoids},
and~\cite{meinrenken;groupoids-algebroids}.

\subsection{Lie algebroids and morphisms}\label{subsection;morphism}

A \emph{Lie algebroid} is a real vector bundle $\pi=\pi_A\colon A\to
M$ equipped with a Lie bracket on the space of sections, which we
denote by
\glossary{[]@$[{\cdot},{\cdot}]_A$, Lie algebroid bracket}
\[
[{\cdot},{\cdot}]=
[{\cdot},{\cdot}]_A\colon\Gamma(A)\times\Gamma(A)\longto\Gamma(A),
\]
and a vector bundle map $\an=\an_A\colon A\to TM$ called the
\emph{anchor}.  The bracket is required to satisfy the \emph{Leibniz
  rule} with respect to the anchor:
$[\sigma,f\tau]=f[\sigma,\tau]+(\sigma\cdot f)\tau$ for all sections
$\sigma$, $\tau\in\Gamma(A)$ and functions $f\in\ca{C}^\infty(M)$,
where $\sigma\cdot f$ denotes the derivative of\/ $f$ along the vector
field~$\an(\sigma)$.
\glossary{AtoM@$A\to M$, Lie algebroid}
\glossary{an@$\an_A\colon A\to TM$, anchor of Lie algebroid}

For a vector bundle $E\to M$ and a (not necessarily embedded)
subbundle $F\to N$ we define a \emph{relative section} to be a section
of\/ $E$ whose restriction to $N$ is a section of\/ $F$.  We denote
the $\ca{C}^\infty(M)$-module of relative sections by
\glossary{GammaEF@$\Gamma(E;F)$, relative sections of vector bundle
  pair}
\begin{equation}\label{equation;relative}
\Gamma(E;F)=\{\,\sigma\in\Gamma(E)\mid\sigma|_N\in\Gamma(F)\,\}.
\end{equation}
Let $0_N$ denote the zero vector bundle over~$N$.
\glossary{0@$0_M$, zero vector bundle over $M$}
Then $\Gamma(E;0_N)$ is the module of sections that vanish at $N$.  We
call a pair of open subsets $U\subseteq M$ and $V\subseteq N$
\emph{adapted} to the submanifold $N$ if $V$ is a closed embedded
submanifold of $U$.  If the pair $(U,V)$ is adapted, then
\begin{equation}\label{equation;relative-adapted}
\Gamma(V,F)\cong\Gamma(E|_U;F|_V)/\Gamma(E|_U;0_V).
\end{equation}

A \emph{Lie subalgebroid} of a Lie algebroid $A\to M$ is a (not
necessarily embedded) subbundle $B\to N$ such that $\an_A(B)\subseteq
TN$ and $\Gamma(A|_U;B|_V)$ is a Lie subalgebra of $\Gamma(A|_U)$ for
all adapted pairs $(U,V)$.  If $B$ is a Lie subalgebroid of $A$, then
by the Leibniz rule $\Gamma(A|_U;0_V)$ is an ideal of
$\Gamma(A|_U;B|_V)$ for every adapted pair $(U,V)$.
Using~\eqref{equation;relative-adapted} we see that
$\Gamma(V,B)\cong\Gamma(A|_U;B|_V)/\Gamma(A|_U;0_V)$ is a Lie algebra.
This makes $B$ a Lie algebroid over $N$ with anchor $\an_B=\an_A|_B$
(\cite[Proposition~6.14]{meinrenken;groupoids-algebroids}).

A \emph{morphism} from $A$ to another Lie algebroid $\pi_B\colon B\to
N$ is most conveniently defined as a vector bundle morphism
$\phi\colon A\to B$ whose graph is a Lie subalgebroid of the direct
product Lie algebroid $A\times B$
(\cite[\S\,7.1]{meinrenken;groupoids-algebroids}).  We will denote the
base map induced by a Lie algebroid morphism $\phi$ by
\glossary{phi@$\mr\phi$, base map of vector bundle map $\phi$}
\[\mr\phi\colon M\longto N.\]
The anchor of $A$ itself is a morphism of Lie algebroids $\an_A\colon
A\to TM$ with base map $\mr\an_A=\id_M$.  Lie algebroids and their
morphisms form a category.

\subsection{Fibred products}\label{subsection;fibred}

Given Lie algebroids $A_0\to M_0$, $A_1\to M_1$, and $A_2\to M_2$ and
Lie algebroid morphisms $\phi_1\colon A_1\to A_0$ and $\phi_2\colon
A_2\to A_0$, we form the fibred product
\begin{equation}\label{equation;fibred-product}
A=A_1\times_{A_0}A_2=\{\,(a_1,a_2)\in A_1\times
A_2\mid\phi_1(a_1)=\phi_2(a_2)\,\},
\end{equation}
which we regard as a subspace of the direct product Lie algebroid
$A_1\times A_2$ over $M_1\times M_2$.  The image of the natural
projection $A\to M_1\times M_2$ is the fibred product
$M=M_1\times_{M_0}M_2$.  Though $M$ may not be a smooth manifold and
$A$ may not be a vector bundle, for every $x=(x_1,x_2)\in M$ the fibre
of the projection $A\to M$ is a vector space, namely
\begin{equation}\label{equation;fibred-fibre}
A_x=A_{1,x_1}\times_{A_{0,x_0}}A_{2,x_2},
\end{equation}
where $x_0=\mr\phi_1(x_1)=\mr\phi_2(x_2)$.  The next result is taken
from Meinrenken's
notes~\cite[Proposition~7.14]{meinrenken;groupoids-algebroids}.

\begin{proposition}\label{proposition;fibred-product}
Let $A_0\to M_0$, $A_1\to M_1$, and $A_2\to M_2$ be Lie algebroids and
let $\phi_1\colon A_1\to A_0$ and $\phi_2\colon A_2\to A_0$ be Lie
algebroid morphisms.  If $\phi_1$ and $\phi_2$ intersect cleanly, the
fibred product $A=A_1\times_{A_0}A_2$ is a Lie subalgebroid of
$A_1\times A_2$ over the submanifold $M=M_1\times_{M_0}M_2$ of
$M_1\times M_2$.  The fibre of $A$ at $x\in M$ is the vector
space~\eqref{equation;fibred-fibre}.
\end{proposition}

A surprisingly recent result of Grabowski and
Rotkiewicz~\cite{grabowski-rotkiewicz;higher-vector-bundles} provides
us with a convenient criterion for two vector bundle maps to intersect
cleanly or transversely.

\begin{lemma}\label{lemma;vector-clean}
Let $X_0$, $X_1$, and $X_2$ be manifolds and let $E_0\to X_0$, $E_1\to
X_1$, and $E_2\to X_2$ be vector bundles.  Let\/ $\phi_1\colon E_1\to
E_0$, resp.\ $\phi_2\colon E_2\to E_0$, be vector bundle morphisms
with base maps $\mr\phi_1\colon X_1\to X_0$, resp.\ $\mr\phi_2\colon
X_2\to X_0$.  Form the fibred products
\begin{gather*}
X=X_1\times_{X_0}X_2=\{\,(x_1,x_2)\in X_1\times
X_2\mid\mr\phi_1(x_1)=\mr\phi_2(x_2)\,\},
\\
E=E_1\times_{E_0}E_2=\{\,(u_1,u_2)\in E_1\times
E_2\mid\phi_1(u_1)=\phi_2(u_2)\,\}.
\end{gather*}
\[
\begin{tikzcd}
&E\arrow[dl]\arrow[rr]\arrow[dd]&&E_2\arrow[dl,"\phi_2"]\arrow[dd]\\
E_1\arrow[rr,crossing over,"\phi_1"near end]\arrow[dd]&&E_0\\
&X\arrow[dl]\arrow[rr]&&X_2\arrow[dl,"\mr\phi_2"]\\
X_1\arrow[rr,"\mr\phi_1"]&&X_0\arrow[from=uu,crossing over]\\
\end{tikzcd}
\]
Define $\pi_E\colon E\to X$ by
$\pi_E(u)=(\pi_{E_1}(u_1),\pi_{E_2}(u_2))$ for $u=(u_1,u_2)\in E$ and
define $E_x=\pi_E^{-1}(x)$ for $x\in X$.
\begin{enumerate}
\item\label{item;fibred-clean}
The morphisms\/ $\phi_1$ and\/ $\phi_2$ intersect cleanly if and only
if the base maps $\mr\phi_1$ and $\mr\phi_2$ intersect cleanly and the
dimension of the vector space $E_x$ is independent of $x\in X$.
\item\label{item;fibred-transverse}
The morphisms\/ $\phi_1$ and\/ $\phi_2$ intersect transversely if and
only if $\mr\phi_1$ and $\mr\phi_2$ intersect transversely and
$E_{0,x_0}=\phi_1(E_{1,x_1})+\phi_2(E_{2,x_2})$ for all triples
$(x_0,x_1,x_2)\in X_0\times X_1\times X_2$ with
$x_0=\mr\phi_1(x_1)=\mr\phi_2(x_2)$.
\item\label{item;fibred-product}
If\/ $\phi_1$ and\/ $\phi_2$ intersect cleanly, then $X$ is a
submanifold of $X_1\times X_2$ and $E$ is a subbundle of $E_1\times
E_2$ with base $X$ and bundle projection~$\pi_E$.
\end{enumerate}
\end{lemma}

\begin{proof}
\eqref{item;fibred-clean}~and \eqref{item;fibred-product}~We view
$E_1\times E_2$ as a vector bundle over $X_1\times X_2$ and we
identify $X_1\times X_2$ with the zero section of $E_1\times E_2$.
The intersection of $E$ with the zero section is $E\cap(X_1\times
X_2)=X$ and $E$ is invariant under fibrewise scalar multiplication on
$E_1\times E_2$.  Suppose $\phi_1$ and $\phi_2$ intersect cleanly.
Then $E$ is a submanifold of $E_1\times E_2$.  In view
of~\cite[Theorem~2.3]{grabowski-rotkiewicz;higher-vector-bundles} this
implies the following two facts: (1)~$E$ cleanly intersects the zero
section (so $X_0$ is a submanifold of $X_1\times X_2$); and (2)~$E$ is
a subbundle of $E_1\times E_2$ over the submanifold~$X$.  It follows
from fact~(1) that $\mr\phi_1$ and $\mr\phi_2$ intersect cleanly.  It
follows from fact~(2) that the fibre $E_x$ has constant dimension for
$x\in X$.  To prove the converse, now suppose $\phi_1$ and $\phi_2$
intersect cleanly and that $E_x$ has constant dimension for $x\in X$.
Then $X$ is a submanifold of $X_1\times X_2$.  For $i=0$, $1$, $2$ let
$F_i$ be the pullback of $E_i$ to $X$, and for $i=1$, $2$ let
$\psi_i\colon F_i\to F_0$ be the vector bundle map induced
by~$\phi_i$.  Let $\psi\colon F_1\oplus F_2\to F_0$ be the vector
bundle map defined by $\psi=(\psi_1,-\psi_2)$.  Then $E$ is the kernel
of $\psi$, which by hypothesis has constant rank, so $E$ is a vector
bundle over~$X$.  To show that $\phi_1$ and $\phi_2$ intersect cleanly
we must calculate the tangent space to $E$ at an arbitrary point $u\in
E$.  Let $x\in X$ be the basepoint of $u$; then $x$ is a pair
$x=(x_1,x_2)\in X_1\times X_2$ with $\mr\phi_1(x_1)=\mr\phi_2(x_2)$.
Likewise $u$ is a pair $u=(u_1,u_2)\in E_1\times E_2$ with
$\phi_1(u_1)=\phi_2(u_2)$.  Put $x_0=\mr\phi_1(x_1)=\mr\phi_2(x_2)\in
X_0$ and $u_0=\phi_1(u_1)=\phi_2(u_2)\in E_0$; then $x_0$ is the
basepoint of~$u_0$.  The tangent space $T_uE$ fits into a short exact
sequence
\[
\begin{tikzcd}[column sep=large]
E_x\ar[r,hook]&T_uE\ar[r,two heads]&T_xX.
\end{tikzcd}
\]
We have similar sequences for $E_0$, $E_1$, and $E_2$, which combine
into a commutative diagram
\begin{equation}\label{equation;vector-clean}
\begin{tikzcd}[row sep=large]
E_x\ar[r,hook]\ar[d,hook]&T_uE\ar[r,two
  heads]\ar[d,hook]&T_xX\ar[d,hook]
\\
E_{1,x_1}\times E_{2,x_2}\ar[r,hook]\ar[d,"{(\phi_1,-\phi_2)}"']&
T_{u_1}E_1\times T_{u_2}E_2\ar[r,two
  heads]\ar[d,"{(T\phi_1,-T\phi_2)}"']& T_{x_1}X_1\times
T_{x_2}X_2\ar[d,"{(T\mr\phi_1,-T\mr\phi_2)}"']
\\
E_{0,x_0}\ar[r,hook]&T_{u_0}E_0\ar[r,two heads]&T_{x_0}X_0,
\end{tikzcd}
\end{equation}
which has exact rows.  The right column of the diagram is exact
because $\mr\phi_1$ and $\mr\phi_2$ intersect cleanly.  The left
column is exact because $E=\ker(\psi)$.  Hence the middle column is
exact, which shows that
$T_uE=T_{u_1}E_1\times_{T_{u_0}E_0}T_{u_2}E_2$, proving that $\phi_1$
and $\phi_2$ intersect cleanly.

\eqref{item;fibred-transverse}~This follows
from~\eqref{item;fibred-clean} plus the observation that the arrows
$(\phi_1,-\phi_2)$ and $(T\mr\phi_1,-T\mr\phi_2)$
in~\eqref{equation;vector-clean} are surjective if and only if the
arrow $(T\phi_1,-T\phi_2)$ is surjective.
\end{proof}

\subsection{Pullbacks}\label{subsection;pull}

Given a smooth map $F\colon P\to M$ to the base of the Lie algebroid
$A$, the fibred product
\glossary{"!@{$"!$}, Lie algebroid pullback}
\glossary{FA@{$F^"!A$}, pullback of Lie algebroid along map}
\begin{equation}\label{equation;pullback}
F^!A=TP\times_{TM}A=\{\,(v,a)\in TP\times A\mid TF(v)=\an_A(a)\,\}
\end{equation}
of the Lie algebroid morphisms $TF\colon TP\to TM$ and $\an_A\colon
A\to TM$ is called the \emph{(Higgins-Mackenzie) pullback} of $A$
to~$P$.  We regard $F^!A$ as a subspace of the direct sum bundle
$TP\oplus F^*A$ over $P\times_MM=P$.  For every $p\in P$ the fibre of
the natural projection $F^!A\to P$ is the vector space
\begin{equation}\label{equation;pull-fibre}
(F^!A)_p=T_pP\times_{T_{F(p)}M}A_{F(p)}.
\end{equation}
The map $F$ lifts to a map
\glossary{F"!@{$F_{\,"!}$}, natural morphism $F^"!A\to A$}
\begin{equation}\label{equation;pull-morphism}
F_{\,!}\colon F^!A\to A
\end{equation}
given by $F_{\,!}(v,a)=a$, which is linear on the fibres.

\begin{definition}\label{definition;clean}
A smooth map $F\colon P\to M$ \emph{cleanly intersects} the Lie
algebroid $A\to M$ if the tangent map $TF\colon TP\to TM$ cleanly
intersects the anchor $\an\colon A\to TM$.  The map $F$
\emph{transversely intersects} $A$, or \emph{is transverse to} $A$, if
$TF$ is transverse to~$\an$.  If $P$ is a submanifold of $M$ and $F$
is the inclusion map, we say $P$ \emph{cleanly}
(resp.\ \emph{transversely}) \emph{intersects} $A$ if $F$ cleanly
(resp.\ transversely) intersects~$A$.
\end{definition}

The next statement follows from
Proposition~\ref{proposition;fibred-product} and
Lemma~\ref{lemma;vector-clean}.  The \emph{orbits} of $A$ are the
integral manifolds of the (singular) foliation $\im(\an_A)\subseteq
TM$; see e.g.\ \cite[\S\,8.6]{meinrenken;groupoids-algebroids}.

\begin{proposition}\label{proposition;pull}
Let $A\to M$ be a Lie algebroid and let $F\colon P\to M$ be a smooth
map.
\begin{enumerate}
\item\label{item;pull-clean-transverse}
The map $F$ intersects $A$ cleanly if and only if the dimension of the
vector space\/ $(F^!A)_p$ given by~\eqref{equation;pull-fibre} is
independent of $p\in P$.  The map $F$ is transverse to $A$ if and only
if\/ $F$ is transverse to all orbits of~$A$.
\item\label{item;pull}
If $F$ intersects $A$ cleanly, the pullback $F^!A$ is a Lie algebroid
over $P$, whose fibre at $p\in P$ is the vector
space~\eqref{equation;pull-fibre}.  The map $F_{\,!}\colon F^!A\to A$
given by~\eqref{equation;pull-morphism} is a Lie algebroid morphism
with base map equal to $\mr F_{\,!}=F$.
\end{enumerate}
\end{proposition}  

A section of $F^!A$ is a pair $(v,\sigma)$ consisting of a vector
field $v$ on $P$ and a section of the (ordinary) pullback bundle
$\sigma\in\Gamma(F^*A)$ such that $T_pF(v(p))=\an_A(\sigma(p))$ for
all $p\in P$.  The anchor of $F^!A$ is given by
$\an_{F^!A}(v,\sigma)=v$.  Using the isomorphism
\[\Gamma(F^*A)\cong\ca{C}^\infty(P)\otimes_{\ca{C}^\infty(M)}\Gamma(A)\]
we can write any section of $F^!A$ as a pair
$\bigl(v,\sum_if_i\sigma_i\bigr)$, where $v\in\Gamma(TP)$,
$f_i\in\ca{C}^\infty(P)$, and $\sigma_i\in\Gamma(A)$ satisfy
$T_pF(v(p))=\sum_if_i(p)\an_A(\sigma_i(F(p)))$ for all $p\in P$.  The
Lie bracket of two sections $(v,\sigma)$ and $(w,\tau)$ of $F^!A$ with
$\sigma=\sum_if_i\sigma_i$ and $\tau=\sum_jg_j\tau_j$ is given by
\begin{equation}\label{equation;pull-bracket}
\bigl[(v,\sigma),(w,\tau)\bigr]=
\biggl([v,w],\sum_{i,j}f_ig_j[\sigma_i,\tau_j]+\sum_j(v\cdot
F^*g_j)\tau_j-\sum_i(w\cdot F^*f_i)\sigma_i\biggr).
\end{equation}

For maps of constant rank we can reformulate
Proposition~\ref{proposition;pull}\eqref{item;pull-clean-transverse}
as follows.

\begin{lemma}\label{lemma;constant-clean}
Let $A\to M$ be a Lie algebroid.
\begin{enumerate}
\item\label{item;constant-rank}
Let $F\colon P\to M$ be a smooth map of constant rank and let
$\ca{N}(F)=F^*TM/\im(TF)$ be the normal bundle of~$F$.  The anchor
$\an\colon A\to TM$ induces a vector bundle map $\overline{\an}\colon
F^*A\to\ca{N}(F)$.  The map $F$ intersects $A$ cleanly if and only
if\/ $\overline{\an}$ has constant rank.  In that case we have an
exact sequence of vector bundles over $P$,
\[
\begin{tikzcd}
\ker(TF)\ar[r,hook]&F^!A\ar[r]&F^*A\ar[r,"\overline{\an}"]&\ca{N}(F).
\end{tikzcd}
\]
The map $F$ is transverse to $A$ if and only if\/ $\overline{\an}$ is
surjective.
\item\label{item;submanifold}
Let $i_N\colon N\to M$ be a submanifold and let $\ca{N}(M,N)$ be the
normal bundle of $N$ in~$M$.  The anchor $\an\colon A\to TM$ induces a
vector bundle map $\overline{\an}\colon i_N^*A\to\ca{N}(M,N)$.  The
submanifold $N$ intersects $A$ cleanly if and only if\/
$\overline{\an}$ has constant rank.  In that case we have an embedding
of vector bundles over $N$,
\[
\begin{tikzcd}
i_N^*A/i_N^!A\ar[r,hook]&\ca{N}(M,N).
\end{tikzcd}
\]
The submanifold $N$ is transverse to $A$ if and only if this embedding
is an isomorphism.
\end{enumerate}
\end{lemma}



\begin{remarks}\phantomsection\label{remark;clean}
\begin{numerate}
\item\label{item;open}
If the submanifold $i_N\colon N\to M$ is open, then $N$ is transverse
to $A$ and $i_N^!A=i_N^*A$.
\item\label{item;point-orbit}
If $N$ is an orbit of $A$, then $N$ cleanly intersects $A$ (but the
intersection is not transverse unless the orbit is open) and
$i_N^!A=i_N^*A$.  Any submanifold of an orbit of $A$ intersects $A$
cleanly.  In particular, if $N=\{x\}$ consists of a single point, then
$N$ cleanly intersects $A$ (but the intersection is not transverse
unless the anchor is surjective at $x$).  So $i_N^!A=\ker(\an_x)$ is a
Lie algebroid over $x$, i.e.\ a Lie algebra, known as the
\emph{isotropy} or \emph{stabilizer} $\stab(A,x)$ of the Lie algebroid
$A$ at~$x$.
\glossary{stabAx@$\stab(A,x)$, Lie algebroid stabilizer of point}
\end{numerate}
\end{remarks}

\subsection{A regular value theorem}\label{subsection;regular}

The following special case of
Proposition~\ref{proposition;fibred-product} is a version of the
regular value theorem for Lie algebroids.

\begin{proposition}\label{proposition;clean-regular}
Let $A\to M$ and $E\to P$ be Lie algebroids, let $\phi\colon A\to E$
be a morphism, and let $F\to Q$ be a Lie subalgebroid of $E$.  Suppose
$\phi$ intersects $F$ cleanly.  Then $A\times_EF=\phi^{-1}(F)$ is a
Lie subalgebroid of $A$, whose base is the submanifold
$\mr\phi^{-1}(Q)$ of $M$.  In particular, if $p$ is a single point of
$P$, $\f$ is a Lie subalgebra of $\stab(E,p)$, and $\phi$ intersects
$\f$ cleanly, then $\phi^{-1}(\f)$ is a Lie subalgebroid of $A$, whose
base is the submanifold $\mr\phi^{-1}(p)$.
\end{proposition}

\begin{remark}\label{remark;clean-regular}
By Lemma~\ref{lemma;vector-clean}\eqref{item;fibred-clean}, the map
$\phi$ intersects $\f$ cleanly if and only if $p$ is a clean value of
$\mr\phi$ and the subspace $\phi_x^{-1}(\f)$ has constant dimension
for $x\in\mr\phi^{-1}(p)$.  By
Lemma~\ref{lemma;vector-clean}\eqref{item;fibred-transverse}, $\phi$
intersects $\f$ transversely if and only $p$ is a regular value of
$\mr\phi$ and $\phi_x(A_x)+\f=E_p$ for all $x\in\mr\phi^{-1}(p)$.
\end{remark}

The next result says that ``pullback commutes with fibred products''.

\begin{proposition}\label{proposition;pull-fibred}
Let $A\to M$ and $E\to P$ be Lie algebroids and $\phi\colon A\to E$ a
morphism.  Let $g\colon Q\to P$ be a smooth map which intersects $E$
cleanly (resp.\ transversely).  Suppose the natural morphism
$g_!\colon g^!E\to E$ intersects $\phi$ cleanly (resp.\ transversely).
Then the map $\bar{g}\colon M\times_PQ\to M$ induced by the projection
$M\times Q\to M$ intersects $A$ cleanly (resp.\ transversely), and we
have an isomorphism of Lie algebroids $\bar{g}^!A\cong A\times_Eg^!E$.
\end{proposition}

\begin{proof}
Since $g$ intersects $E$ cleanly, by
Proposition~\ref{proposition;pull} the Lie algebroid $g^!E$ and the
morphism $g_!$ are well-defined.  Since $\phi$ intersects $g_!$
cleanly, by Lemma~\ref{lemma;vector-clean}\eqref{item;fibred-clean}
the base maps $\mr\phi$ and $g$ intersect cleanly, so that the fibred
product $\barM=M\times_PQ$ is a manifold and the map
$\bar{g}\colon\barM\to M$ is smooth.  Moreover, the map
\[l_{(x,q)}=(\phi_x,-(g_!)_q)\colon A_x\times(g^!E)_q\longto E_p\]
has constant rank for all triples $(x,q,p)\in M\times Q\times P$ satisfying
\begin{equation}\label{equation;triple}
\mr\phi(x)=g(q)=p.
\end{equation}
The kernel of $l_{(x,q)}$ is the vector space
$A_x\times_{E_p}(g^!E)_q$.  By another application of
Lemma~\ref{lemma;vector-clean}\eqref{item;fibred-clean}, to show that
$\bar{g}$ intersects $A$ cleanly it is enough to show that the map
\[
\bar{l}_{(x,q)}=(\an_{A,x},-T_{(x,q)}\bar{g})\colon A_x\times
T_{(x,q)}\barM\longto T_xM
\]
has constant rank for all pairs $(x,q)\in\barM$.  The kernel of
$\bar{l}_{(x,q)}$ is the vector space $(\bar{g}^!A)_{(x,q)}$.  Using
$T\barM=TM\times_{TP}TQ$ we find a natural isomorphism
\begin{align*}
\ker\bigl(\bar{l}_{(x,q)}\bigr)=(\bar{g}^!A)_{(x,q)}&
=A_x\times_{T_xM}T_{(x,q)}\barM=
A_x\times_{T_xM}(T_xM\times_{T_pP}T_qQ)\\
&\cong A_x\times_{T_pP}T_qQ\cong
  A_x\times_{E_p}(E_p\times_{T_pP}T_qQ)\\
&\cong A_x\times_{E_p}(g^!E)_q=\ker\bigl(l_{(x,q)}\bigr)
\end{align*}
for all triples $(x,q,p)$ satisfying~\eqref{equation;triple}.  So
$\bar{l}$ has constant rank because $l$ does, and we have the
isomorphism $\bar{g}^!A\cong A\times_Eg^!E$.  This establishes the
proposition in the clean case.  If $g$ intersects $E$ transversely and
$g_!$ intersects $\phi$ transversely, then by
Lemma~\ref{lemma;vector-clean}\eqref{item;fibred-transverse} the map
\[(\an_{E,p},-T_qg)\colon E_p\times T_qQ\longto T_pP\]
and the map $l_{(x,q)}$ are surjective for all triples $(x,q,p)$
satisfying~\eqref{equation;triple}.  One deduces from this that
$\bar{l}_{(x,q)}$ is surjective, and hence that $\bar{g}$ is
transverse to~$A$.
\end{proof}

Taking $F=i_Q^!E$ in Proposition~\ref{proposition;clean-regular} and
applying Proposition~\ref{proposition;pull-fibred} gives the
following.

\begin{corollary}\label{corollary;clean-regular}
Let $A\to M$ and $E\to P$ be Lie algebroids and $\phi\colon A\to E$ a
morphism.  Let $Q$ be a submanifold of $P$ which intersects $E$
cleanly (resp.\ transversely).  Suppose $\phi$ intersects the Lie
subalgebroid $i_Q^!E$ of $E$ cleanly (resp.\ transversely).  Then
$N=\mr\phi^{-1}(Q)$ is a submanifold of $M$ which intersects the Lie
algebroid $A$ cleanly (resp.\ transversely), and
$i_N^!A=\phi^{-1}(i_Q^!E)$.  In particular, let $p$ be a point in $P$
such that $\phi$ intersects the subspace $\stab(E,p)$ of $E$ cleanly
(resp.\ transversely).  Then the submanifold $N=\mr\phi^{-1}(p)$ of
$M$ intersects $A$ cleanly (resp.\ transversely), and
$i_N^!A=\phi^{-1}(\stab(E,p))$.
\end{corollary}

\subsection{Lie algebroid differential forms and Cartan calculus}
\label{subsection;cartan}

Let $\pi\colon A\to M$ be a Lie algebroid.  Let $\Lambda^\bu A^*$ be
the exterior algebra of the dual bundle $A^*$ and let $U$ be an open
subset of~$M$.  We denote the graded vector space of sections
$\Gamma(U,\Lambda^\bu A^*)$ by $\Omega^\bu_A(U)$ and we call elements
of $\Omega^\bu_A(U)$ \emph{Lie algebroid differential forms}, or
\emph{$A$-forms}, or just \emph{forms} on~$U$.  The \emph{exterior
  derivative} of an $A$-form $\alpha\in\Omega_A^k(U)$ is the $A$-form
$d_A\alpha\in\Omega_A^{k+1}(U)$ given by
\glossary{dA@$d_A$, Lie algebroid exterior derivative}
\glossary{Omega@$\Omega^\bu_A(M)$, Lie algebroid de Rham complex}
\begin{equation}\label{equation;d}
\begin{aligned}
d_A\alpha(\sigma_1,\sigma_2,\dots,\sigma_{k+1})&=
\sum_{i=1}^{k+1}(-1)^{i+1}\sigma_i\cdot
\alpha(\sigma_1,\sigma_2,\dots,\hat{\sigma}_i,\dots,\sigma_{k+1})+{}\\
&\sum_{1\le i<j\le
  k+1}(-1)^{i+j}\alpha\bigl([\sigma_i,\sigma_j],\sigma_1,\sigma_2,
\dots,\hat{\sigma}_i,\dots,\hat{\sigma}_j,\dots,\sigma_{k+1}\bigr)
\end{aligned}
\end{equation}
for sections $\sigma_1$,
$\sigma_2,\dots$,~$\sigma_{k+1}\in\Gamma(U,A)$.  Here $\sigma\cdot f$
denotes the Lie derivative of a function $f$ along the vector
field~$\an(\sigma)$.  The pair $\bigl(\Omega_A^\bu,d_A\bigr)$ is a
sheaf of commutative differential graded algebras (\textsc{cdga}),
which we call the \emph{de Rham} (or \emph{Chevalley-Eilenberg})
\emph{complex} of~$A$.  We denote the cohomology of the complex of
sections $\bigl(\Omega_A^\bu(U),d_A\bigr)$ by $H_A^\bu(U)$.

A Lie algebroid morphism $\phi\colon A\to B$ induces a morphism of
sheaves of \textsc{cdga}
\[\Omega^\bu(\phi)\colon\Omega_B^\bu\longto\phi_*\Omega_A^\bu.\]
On global sections this gives a pullback map
$\Omega^\bu(\phi)\colon\Omega_B^\bu(N)\to\Omega_A^\bu(M)$ and an
induced map in cohomology $H^\bu(\phi)\colon H_B^\bu(N)\to
H_A^\bu(M)$.  Usually we will denote both $\Omega^\bu(\phi)$ and
$H^\bu(\phi)$ by~$\phi^*$.  Taking $\phi$ to be the anchor $\an\colon
A\to TM$ of $A$ we get pullback maps
$\an^*\colon\Omega^\bu(M)\to\Omega_A^\bu(M)$ and $\an^*\colon
H^\bu(M)\to H_A^\bu(M)$, where $\Omega^\bu(M)$ denotes the usual de
Rham complex of $M$ and $H^\bu(M)$ its cohomology.

Let $\sigma$ be a global section of~$A$.  Associated with $\sigma$ are
various natural objects and operations.  The first is the
\emph{contraction operator}
$\iota_A(\sigma)\colon\Omega_A^k\to\Omega_A^{k-1}$ defined by
\glossary{iotaA@$\iota_A$, Lie algebroid interior product}
\begin{equation}\label{equation;iota}
\iota_A(\sigma)\alpha(\sigma_1,\sigma_2,\dots,\sigma_{k-1})=
\alpha(\sigma,\sigma_1,\sigma_2,\dots,\sigma_{k-1})
\end{equation}
for sections $\sigma_1$, $\sigma_2,\dots$,~$\sigma_{k-1}$ of~$A$.  It
is a derivation of degree $-1$ of the de Rham complex.  The section
$\sigma$ also determines a linear vector field $\bad_A(\sigma)$ on the
total space of $A$ that is given by the formula
\glossary{ad@$\bad_A\colon\Gamma(A)\to\Gamma(TA)$, ``adjoint
  representation'' of Lie algebroid}
\begin{equation}\label{equation;ad}
\bad_A(\sigma)(a)=
T_x\upsilon\bigl(\an_A(\sigma)(x)\bigr)-[\sigma,\upsilon](x)
\end{equation}
for $a\in A$ and $x=\pi_A(a)$, where $\upsilon\in\Gamma(A)$ is any
section with $\upsilon(x)=a$.  The vector field $\bad_A(\sigma)$ is
characterized by the following two requirements: first, the projection
$A\to M$ intertwines the flow $\Phi_t$ of $\bad_A(\sigma)$ with the
flow $\mr\Phi_t$ of the vector field $\an(\sigma)$ on $M$ and second,
\begin{equation}\label{equation;lie-section}
\frac{d}{dt}\Phi_t^*\tau=\Phi_t^*[\sigma,\tau]
\end{equation}
for all sections $\tau\in\Gamma(A)$
(\cite[\S\,7.6]{meinrenken;groupoids-algebroids}).  Here
$\Phi_t^*\tau$ denotes the section
$\Phi_t^{-1}\circ\tau\circ\mr\Phi_t$ of $A$.  We also write
$[\sigma,\tau]=\ca{L}_A(\sigma)\tau$ and call the operator
$\ca{L}_A(\sigma)$ the \emph{Lie derivative}.

A time-dependent section $\sigma_t$ of $A$ (defined for $t$ in some
interval containing $0$) gives rise to a time-dependent linear vector
field $\bad_A(\sigma_t)$ and hence to a flow $\Phi_t$ on $A$ with
initial condition $\Phi_0=\id_A$.  The
identity~\eqref{equation;lie-section} generalizes to
\begin{equation}\label{equation;lie-section-time}
\frac{d}{dt}\Phi_t^*\tau_t=
\Phi_t^*\bigl([\sigma_t,\tau_t]+\dot{\tau}_t\bigr)
\end{equation}
for all time-dependent sections $\tau_t$ of $A$.

\begin{lemma}\label{lemma;pull-auto}
Let $\sigma_t$ be a time-dependent section of a Lie algebroid $A\to M$
and let $\Phi_t$ be the flow of the vector field $\bad_A(\sigma_t)$
with initial condition $\Phi_0=\id_A$.
\begin{enumerate}
\item\label{item;automorphism}
For each $t$ the map
$\Phi_t$ is an automorphism of $A$ (where defined).
\item\label{item;preserve}
Let $B\to M$ be a Lie subalgebroid of $A$ over the same base which is
normalized by $\sigma_t$ in the sense that
$[\sigma_t,\tau]\in\Gamma(B)$ for all $t$ and all $\tau\in\Gamma(B)$.
Then the flow $\Phi_t$ preserves $B$.
\end{enumerate}
\end{lemma}

\begin{proof}
\eqref{item;automorphism}~Reading~\eqref{equation;lie-section-time} as
a linear differential equation
$\frac{d}{dt}\Phi_t^*=\Phi_t^*\circ\ad(\sigma_t)$ for the operator
$\Phi_t^*$ (acting on time-independent sections $\tau$) with initial
condition $\Phi_0^*=\id$, we see that $\Phi_t^*=\exp(\Sigma_t)$, where
$\Sigma_t=\int_0^t\ad(\sigma_u)\,du$.  Since $\Sigma_t$ is a
derivation of the Lie algebra $\Gamma(A)$, $\Phi_t^*$ is an
automorphism of $\Gamma(A)$, so $\Phi_t$ is an automorphism of $A$.

\eqref{item;preserve}~It follows from~\eqref{equation;ad} that
$\bad_A(\sigma_t)$ is tangent to $B$.  Hence $\Phi_t(B)\subseteq B$.
\end{proof}

Dually, for differential forms the Lie derivative
$\ca{L}_A(\sigma)\colon\Omega_A^\bu\to\Omega_A^\bu$ is defined by
\glossary{LA@$\ca{L}_A$, Lie algebroid Lie derivative}
\begin{equation}\label{equation;lie}
\ca{L}_A(\sigma)\alpha=\frac{d}{dt}\Phi_t^*\alpha\Big|_{t=0}.
\end{equation}
It is a derivation of degree $0$ of the de Rham complex.  For a
time-dependent section $\sigma_t$ and a time-dependent $A$-form
$\alpha_t$ we have
\begin{equation}\label{equation;pull}
\frac{d}{dt}\Phi_t^*\alpha_t=
\Phi_t^*\bigl(\ca{L}_A(\sigma_t)\alpha_t+\dot{\alpha}_t\bigr).
\end{equation}

The exterior derivative~\eqref{equation;d}, the
contractions~\eqref{equation;iota}, and the Lie
derivatives~\eqref{equation;lie} obey the usual rules of the Cartan
differential calculus, namely
\begin{equation}\label{equation;cartan}
\begin{aligned}
{[\iota_A(\sigma),\iota_A(\tau)]}&=0,&\quad
[\ca{L}_A(\sigma),\ca{L}_A(\tau)]&=\ca{L}_A([\sigma,\tau]),\\
[\ca{L}_A(\sigma),d_A]&=0,&\quad
[\ca{L}_A(\sigma),\iota_A(\tau)]&=\iota_A([\sigma,\tau]),\\
[d_A,d_A]&=0,&\quad[\iota_A(\sigma),d_A]&=\ca{L}_A(\sigma)
\end{aligned}
\end{equation}
for all $\sigma$, $\tau\in\Gamma(A)$, where the square brackets denote
graded commutators.

\section{Notation index}\label{section;notation}

\input{darboux.gls}


\bibliographystyle{amsplain}

\bibliography{hamilton}



\end{document}